\documentclass[9pt]{amsart}
\usepackage{amsmath}
\usepackage{amsxtra}
\usepackage{amscd}
\usepackage{amsthm}
\usepackage{amsfonts}
\usepackage{amssymb}
\usepackage{eucal}
\usepackage[all]{xy}
\usepackage{graphicx}

\newtheorem{cor}[subsubsection]{Corollary}
\newtheorem{lem}[subsubsection]{Lemma}
\newtheorem{prop}[subsubsection]{Proposition}
\newtheorem{thmconstr}[subsubsection]{Theorem-Construction}

\newtheorem{thm}[subsubsection]{Theorem}

\theoremstyle{definition}

\theoremstyle{remark}
\newtheorem{rem}[subsubsection]{Remark}

\newcommand{\thmref}[1]{Theorem~\ref{#1}}

\newcommand{\secref}[1]{Sect.~\ref{#1}}
\newcommand{\lemref}[1]{Lemma~\ref{#1}}
\newcommand{\propref}[1]{Proposition~\ref{#1}}
\newcommand{\corref}[1]{Corollary~\ref{#1}}

\numberwithin{equation}{section}

\newcommand{\nc}{\newcommand}
\nc{\renc}{\renewcommand}
\nc{\ssec}{\subsection}
\nc{\sssec}{\subsubsection}
\nc{\on}{\operatorname}

\nc\ol{\overline}
\nc\wt{\widetilde}
\nc\tboxtimes{\wt{\boxtimes}}
\nc\tstar{\wt{\star}}
\nc{\alp}{a}

\nc{\ZZ}{{\mathbb Z}}
\nc{\NN}{{\mathbb N}}
\nc{\OO}{{\mathbb O}}
\renc{\SS}{{\mathbb S}}
\nc{\DD}{{\mathbb D}}
\nc{\GG}{{\mathbb G}}

\nc{\Fq}{{\mathbb F}_q}
\nc{\Fqb}{\ol{{\mathbb F}_q}}
\nc{\Ql}{\ol{{\mathbb Q}_\ell}}
\nc{\id}{\text{id}}
\nc\X{\mathcal X}

\nc{\Hom}{\on{Hom}}
\nc{\Lie}{\on{Lie}}
\nc{\Loc}{\on{Loc}}
\nc{\Pic}{\on{Pic}}
\nc{\Bun}{\on{Bun}}
\nc{\IC}{\on{IC}}
\nc{\Fls}{\on{Fl}^{\frac{\infty}{2}}}
\nc{\ICs}{\on{IC}^{\frac{\infty}{2}}}
\nc{\ICsl}{\on{IC}^{\lambda+\frac{\infty}{2}}}
\nc{\ICslm}{\on{IC}^{\lambda+\frac{\infty}{2},-}}
\nc{\ICsm}{\on{IC}^{\frac{\infty}{2},-}}
\nc{\Aut}{\on{Aut}}
\nc{\rk}{\on{rk}}
\nc{\Sh}{\on{Sh}}
\nc{\Perv}{\on{Perv}}
\nc{\pos}{{\on{pos}}}
\nc{\Conv}{\on{Conv}}
\nc{\Sph}{\on{Sph}}
\nc{\Sat}{\on{Sat}}
\nc{\Sym}{\on{Sym}}
\nc{\BunBb}{\overline{\Bun}_B}
\nc{\BunNb}{\overline{\Bun}_N}
\nc{\BunTb}{\overline{\Bun}_T}
\nc{\BunBbm}{\overline{\Bun}_{B^-}}
\nc{\BunBbel}{\overline{\Bun}_{B,el}}
\nc{\BunBbmel}{\overline{\Bun}_{B^-,el}}
\nc{\Buno}{\overset{o}{\Bun}}
\nc{\BunPb}{{\overline{\Bun}_P}}
\nc{\BunBM}{\Bun_{B(M)}}
\nc{\BunBMb}{\overline{\Bun}_{B(M)}}
\nc{\BunPbw}{{\widetilde{\Bun}_P}}
\nc{\BunBP}{\widetilde{\Bun}_{B,P}}
\nc{\GUb}{\overline{G/U}}
\nc{\GUPb}{\overline{G/U(P)}}

\nc{\Hhom}{\underline{\on{Hom}}}
\nc\syminfty{\on{Sym}^{\infty}}
\nc\lal{\ol{\kappa_x}}
\nc\xl{\ol{x}}
\nc\thl{\ol{\theta}}
\nc\nul{\ol{\nu}}
\nc\mul{\ol{\mu}}
\nc\Sum\Sigma
\nc{\oX}{\overset{o}{X}{}}
\nc{\hl}{\overset{\leftarrow}h{}}
\nc{\hr}{\overset{\rightarrow}h{}}
\nc{\M}{{\mathcal M}}
\nc{\N}{{\mathcal N}}
\nc{\F}{{\mathcal F}}
\nc{\D}{{\mathcal D}}
\nc{\Q}{{\mathcal Q}}
\nc{\Y}{{\mathcal Y}}
\nc{\G}{{\mathcal G}}
\nc{\E}{{\mathcal E}}
\nc{\CalC}{{\mathcal C}}
\nc\Dh{\widehat{\D}}

\nc{\C}{{\mathcal C}}
\nc{\K}{{\mathcal K}}
\renewcommand{\H}{{\mathcal H}}

\nc{\T}{{\mathcal T}}
\nc{\V}{{\mathcal V}}
\renc{\P}{{\mathcal P}}
\nc{\A}{{\mathcal A}}
\nc{\B}{{\mathcal B}}
\nc{\U}{{\mathcal U}}

\nc{\Gr}{{\on{Gr}}}

\nc{\frn}{{\check{\mathfrak u}(P)}}

\nc{\fC}{\mathfrak C}
\nc{\p}{\mathfrak p}
\nc{\q}{\mathfrak q}
\nc\f{{\mathfrak f}}

\nc{\qo}{{\mathfrak q}}
\nc{\po}{{\mathfrak p}}
\nc{\s}{{\mathfrak s}}
\nc\w{\text{w}}

\renewcommand{\mod}{{\on{-mod}}}

\nc\mathi\iota
\nc\Spec{\on{Spec}}
\nc\Mod{\on{Mod}}
\nc{\tw}{\widetilde{\mathfrak t}}
\nc{\pw}{\widetilde{\mathfrak p}}
\nc{\qw}{\widetilde{\mathfrak q}}
\nc{\jw}{\widetilde j}

\nc{\grb}{\overline{\Gr}}

\nc{\kappach}{{\check\kappa_x}}
\nc{\Lambdach}{{\check\Lambda}{}}
\nc{\much}{{\check\mu}}
\nc{\omegach}{{\check\omega}}
\nc{\nuch}{{\check\nu}}
\nc{\etach}{{\check\eta}}
\nc{\alphach}{{\checka}}
\nc{\oblvtach}{{\check\oblvta}}
\nc{\pich}{{\check\pi}}
\nc{\ch}{{\check h}}

\nc{\Hb}{\overline{\H}}

\nc{\BA}{{\mathbb{A}}}
\nc{\BC}{{\mathbb{C}}}
\nc{\BE}{{\mathbb{E}}}
\nc{\BF}{{\mathbb{F}}}
\nc{\BG}{{\mathbb{G}}}
\nc{\BM}{{\mathbb{M}}}
\nc{\BO}{{\mathbb{O}}}
\nc{\BD}{{\mathbb{D}}}
\nc{\BN}{{\mathbb{N}}}
\nc{\BP}{{\mathbb{P}}}
\nc{\BQ}{{\mathbb{Q}}}
\nc{\BR}{{\mathbb{R}}}
\nc{\BZ}{{\mathbb{Z}}}
\nc{\BS}{{\mathbb{S}}}

\nc{\CA}{{\mathcal{A}}}
\nc{\CB}{{\mathcal{B}}}

\nc{\CE}{{\mathcal{E}}}
\nc{\CF}{{\mathcal{F}}}
\nc{\CG}{{\mathcal{G}}}
\nc{\CH}{{\mathcal{H}}}

\nc{\CL}{{\mathcal{L}}}
\nc{\CC}{{\mathcal{C}}}
\nc{\CM}{{\mathcal{M}}}
\nc{\CN}{{\mathcal{N}}}
\nc{\cCN}{\check{{\mathcal{N}}}}
\nc{\CK}{{\mathcal{K}}}
\nc{\CO}{{\mathcal{O}}}
\nc{\CP}{{\mathcal{P}}}
\nc{\CQ}{{\mathcal{Q}}}
\nc{\CR}{{\mathcal{R}}}
\nc{\CS}{{\mathcal{S}}}
\nc{\CT}{{\mathcal{T}}}
\nc{\CU}{{\mathcal{U}}}
\nc{\CV}{{\mathcal{V}}}
\nc{\CW}{{\mathcal{W}}}
\nc{\CX}{{\mathcal{X}}}
\nc{\CY}{{\mathcal{Y}}}
\nc{\CZ}{{\mathcal{Z}}}
\nc{\CI}{{\mathcal{I}}}
\nc{\CJ}{{\mathcal{J}}}

\nc{\csM}{{\check{\mathcal A}}{}}
\nc{\oM}{{\overset{\circ}{\mathcal M}}{}}
\nc{\obM}{{\overset{\circ}{\mathbf M}}{}}
\nc{\oCA}{{\overset{\circ}{\mathcal A}}{}}
\nc{\obA}{{\overset{\circ}{\mathbf A}}{}}
\nc{\ooM}{{\overset{\circ}{M}}{}}
\nc{\osM}{{\overset{\circ}{\mathsf M}}{}}
\nc{\vM}{{\overset{\bullet}{\mathcal M}}{}}
\nc{\nM}{{\underset{\bullet}{\mathcal M}}{}}
\nc{\oD}{{\overset{\circ}{\mathcal D}}{}}
\nc{\obD}{{\overset{\circ}{\mathbf D}}{}}
\nc{\oA}{{\overset{\circ}{\mathbb A}}{}}
\nc{\op}{{\overset{\bullet}{\mathbf p}}{}}
\nc{\cp}{{\overset{\circ}{\mathbf p}}{}}
\nc{\oU}{{\overset{\bullet}{\mathcal U}}{}}
\nc{\oZ}{{\overset{\circ}{\mathcal Z}}{}}
\nc{\ofZ}{{\overset{\circ}{\mathfrak Z}}{}}
\nc{\oF}{{\overset{\circ}{\fF}}}

\nc{\fa}{{\mathfrak{a}}}
\nc{\fb}{{\mathfrak{b}}}
\nc{\fd}{{\mathfrak{d}}}
\nc{\ff}{{\mathfrak{f}}}
\nc{\fg}{{\mathfrak{g}}}
\nc{\fgl}{{\mathfrak{gl}}}
\nc{\fh}{{\mathfrak{h}}}
\nc{\fj}{{\mathfrak{j}}}
\nc{\fl}{{\mathfrak{l}}}
\nc{\fm}{{\mathfrak{m}}}
\nc{\fn}{{\mathfrak{n}}}
\nc{\fu}{{\mathfrak{u}}}
\nc{\fp}{{\mathfrak{p}}}
\nc{\fr}{{\mathfrak{r}}}
\nc{\fs}{{\mathfrak{s}}}
\nc{\ft}{{\mathfrak{t}}}
\nc{\fz}{{\mathfrak{z}}}
\nc{\fsl}{{\mathfrak{sl}}}
\nc{\hsl}{{\widehat{\mathfrak{sl}}}}
\nc{\hgl}{{\widehat{\mathfrak{gl}}}}
\nc{\hg}{{\widehat{\mathfrak{g}}}}
\nc{\chg}{{\widehat{\mathfrak{g}}}{}^\vee}
\nc{\hn}{{\widehat{\mathfrak{n}}}}
\nc{\chn}{{\widehat{\mathfrak{n}}}{}^\vee}

\nc{\fA}{{\mathfrak{A}}}
\nc{\fB}{{\mathfrak{B}}}
\nc{\fD}{{\mathfrak{D}}}
\nc{\fE}{{\mathfrak{E}}}
\nc{\fF}{{\mathfrak{F}}}
\nc{\fG}{{\mathfrak{G}}}
\nc{\fK}{{\mathfrak{K}}}
\nc{\fL}{{\mathfrak{L}}}
\nc{\fM}{{\mathfrak{M}}}
\nc{\fN}{{\mathfrak{N}}}
\nc{\fP}{{\mathfrak{P}}}
\nc{\fU}{{\mathfrak{U}}}
\nc{\fV}{{\mathfrak{V}}}
\nc{\fZ}{{\mathfrak{Z}}}

\nc{\ba}{{\mathbf{a}}}
\nc{\bb}{{\mathbf{b}}}
\nc{\bc}{{\mathbf{c}}}
\nc{\bd}{{\mathbf{d}}}
\nc{\bbf}{{\mathbf{f}}}
\nc{\be}{{\mathbf{e}}}
\nc{\bi}{{\mathbf{i}}}
\nc{\bj}{{\mathbf{j}}}
\nc{\bn}{{\mathbf{n}}}
\nc{\bo}{{\mathbf{o}}}
\nc{\bp}{{\mathbf{p}}}
\nc{\bq}{{\mathbf{q}}}
\nc{\bu}{{\mathbf{u}}}
\nc{\bv}{{\mathbf{v}}}
\nc{\bx}{{\mathbf{x}}}
\nc{\bs}{{\mathbf{s}}}
\nc{\by}{{\mathbf{y}}}
\nc{\bw}{{\mathbf{w}}}
\nc{\bA}{{\mathbf{A}}}
\nc{\bK}{{\mathbf{K}}}
\nc{\bB}{{\mathbf{B}}}
\nc{\bF}{{\mathbf{F}}}
\nc{\bC}{{\mathbf{C}}}
\nc{\bG}{{\mathbf{G}}}
\nc{\bD}{{\mathbf{D}}}
\nc{\bE}{{\mathbf{E}}}
\nc{\bH}{{\mathbf{H}}}
\nc{\bI}{{\mathbf{I}}}
\nc{\bM}{{\mathbf{M}}}
\nc{\bN}{{\mathbf{N}}}
\nc{\bO}{{\mathbf{O}}}
\nc{\bV}{{\mathbf{V}}}
\nc{\bW}{{\mathbf{W}}}
\nc{\bX}{{\mathbf{X}}}
\nc{\bZ}{{\mathbf{Z}}}
\nc{\bS}{{\mathbf{S}}}

\nc{\sA}{{\mathsf{A}}}
\nc{\sB}{{\mathsf{B}}}
\nc{\sC}{{\mathsf{C}}}
\nc{\sD}{{\mathsf{D}}}
\nc{\sF}{{\mathsf{F}}}
\nc{\sK}{{\mathsf{K}}}
\nc{\sM}{{\mathsf{M}}}
\nc{\sO}{{\mathsf{O}}}
\nc{\sW}{{\mathsf{W}}}
\nc{\sQ}{{\mathsf{Q}}}
\nc{\sP}{{\mathsf{P}}}
\nc{\sZ}{{\mathsf{Z}}}
\nc{\sr}{{\mathsf{r}}}
\nc{\bk}{{\mathsf{k}}}
\nc{\sg}{{\mathsf{g}}}
\nc{\sff}{{\mathsf{f}}}
\nc{\sfe}{{\mathsf{e}}}
\nc{\sfj}{{\mathsf{j}}}
\nc{\sfb}{{\mathsf{b}}}
\nc{\sfc}{{\mathsf{c}}}
\nc{\sd}{{\mathsf{d}}}
\nc{\sv}{{\mathsf{v}}}

\nc{\BK}{{\bar{K}}}

\nc{\tA}{{\widetilde{\mathbf{A}}}}
\nc{\tB}{{\widetilde{\mathcal{B}}}}
\nc{\tg}{{\widetilde{\mathfrak{g}}}}
\nc{\tG}{{\widetilde{G}}}
\nc{\TM}{{\widetilde{\mathbb{M}}}{}}
\nc{\tO}{{\widetilde{\mathsf{O}}}{}}
\nc{\tU}{{\widetilde{\mathfrak{U}}}{}}
\nc{\TZ}{{\tilde{Z}}}
\nc{\tx}{{\tilde{x}}}
\nc{\tbv}{{\tilde{\bv}}}
\nc{\tfP}{{\widetilde{\mathfrak{P}}}{}}
\nc{\tz}{{\tilde{\zeta}}}
\nc{\tmu}{{\tilde{\mu}}}

\nc{\urho}{\underline{\pi}}
\nc{\uB}{\underline{B}}
\nc{\uC}{{\underline{\mathbb{C}}}}
\nc{\ui}{\underline{i}}
\nc{\uj}{\underline{j}}
\nc{\ofP}{{\overline{\mathfrak{P}}}}
\nc{\oB}{{\overline{\mathcal{B}}}}
\nc{\og}{{\overline{\mathfrak{g}}}}
\nc{\oI}{{\overline{I}}}

\nc{\eps}{\varepsilon}
\nc{\hrho}{{\hat{\pi}}}

\nc{\one}{{\mathbf{1}}}
\nc{\two}{{\mathbf{t}}}

\nc{\Rep}{{\mathop{\operatorname{\rm Rep}}}}
\nc{\Tot}{{\mathop{\operatorname{\rm Tot}}}}
\nc{\Ker}{{\mathop{\operatorname{\rm Ker}}}}
\nc{\Hilb}{{\mathop{\operatorname{\rm Hilb}}}}
\nc{\End}{{\mathop{\operatorname{\rm End}}}}
\nc{\Ext}{{\mathop{\operatorname{\rm Ext}}}}
\nc{\CHom}{{\mathop{\operatorname{{\mathcal{H}}\it om}}}}
\nc{\GL}{{\mathop{\operatorname{\rm GL}}}}
\nc{\gr}{{\mathop{\operatorname{\rm gr}}}}
\nc{\Ld}{{\mathop{\operatorname{\rm Id}}}}
\nc{\de}{{\mathop{\operatorname{\rm def}}}}
\nc{\length}{{\mathop{\operatorname{\rm length}}}}
\nc{\supp}{{\mathop{\operatorname{\rm supp}}}}

\nc{\Cliff}{{\mathsf{Cliff}}}
\nc{\Fl}{\on{Fl}}
\nc{\Fib}{{\mathsf{Fib}}}
\nc{\Coh}{{\mathsf{Coh}}}
\nc{\QCoh}{{\on{QCoh}}}
\nc{\LndCoh}{{\on{IndCoh}}}
\nc{\FCoh}{{\mathsf{FCoh}}}

\nc{\reg}{{\text{\rm reg}}}

\nc{\cplus}{{\mathbf{C}_+}}
\nc{\cminus}{{\mathbf{C}_-}}
\nc{\cthree}{{\mathbf{C}_*}}
\nc{\Qbar}{{\bar{Q}}}
\nc\Eis{\on{Eis}}
\nc\Eisb{\ol\Eis{}}
\nc\Eisr{\on{Eis}^{rat}{}}
\nc\wh{\widehat}
\nc{\Def}{\on{Def_{\check{\fb}}(E)}}
\nc{\barZ}{\overline{Z}{}}
\nc{\barbarZ}{\overline{\barZ}{}}
\nc{\barpi}{\overline\iota}
\nc{\barbarpi}{\overline\barpi}
\nc{\barpip}{\overline\iota{}^+}
\nc{\barpim}{\overline\iota{}^-}

\nc{\fq}{\mathfrak q}

\nc{\fqb}{\ol{\fq}{}}
\nc{\fpb}{\ol{\fp}{}}
\nc{\fpr}{{\fp^{rat}}{}}
\nc{\fqr}{{\fq^{rat}}{}}

\nc{\hattimes}{\wh\otimes}

\nc{\bh}{{\bar{h}}}
\nc{\bOmega}{{\overline{\Omega(\check \fn)}}}

\nc{\seq}[1]{\stackrel{#1}{\sim}}

\nc{\cT}{{\check{T}}}
\nc{\cG}{{\check{G}}}
\nc{\cM}{{\check{M}}}
\nc{\cB}{{\check{B}}}
\nc{\cN}{{\check{N}}}

\nc{\ct}{{\check{\mathfrak t}}}
\nc{\cg}{{\check{\fg}}}
\nc{\cb}{{\check{\fb}}}
\nc{\cn}{{\check{\fn}}}

\nc{\cLambda}{{\check\Lambda}}

\nc{\cla}{{\check\kappa_x}}
\nc{\cmu}{{\check\mu}}
\nc{\clambda}{{\check\lambda}}
\nc{\cnu}{{\check\nu}}
\nc{\ceta}{{\check\eta}}

\nc{\DefbE}{{\on{Def}_{\cB}(E_\cT)}}

\nc{\imathb}{{\ol{\imath}}}
\nc{\rlr}{\overset{\longrightarrow}{\underset{\longrightarrow}\longleftarrow}}

\nc{\KG}{K\backslash G}
\nc{\comult}{{co\text{-}mult}}
\nc{\counit}{{co\text{-}unit}}
\nc{\uHom}{{\underline{\Maps}}}
\nc{\dgSch}{\on{Sch}}
\nc{\Sch}{\on{Sch}}
\nc{\affdgSch}{\on{Sch}^{\on{aff}}}
\nc{\affSch}{\on{Sch}^{\on{aff}}}
\nc{\Groupoids}{\on{Grpd}}
\nc{\inftygroup}{\on{Spc}}
\nc{\inftyCat}{\infty\on{-Cat}}
\nc{\StinftyCat}{\inftyCat^{\on{St}}}
\nc{\MoninftyCat}{\infty\on{-Cat}^{\on{Mon}}}
\nc{\SymMoninftyCat}{\infty\on{-Cat}^{\on{SymMon}}}
\nc{\SymMonStinftyCat}{\on{DGCat}^{\on{SymMon}}}
\nc{\MonStinftyCat}{\on{DGCat}^{\on{Mon}}}
\nc{\inftystack}{\on{Stk}}
\nc{\inftystackalg}{St\sfe^{1\text{-}alg}}
\nc{\inftyprestack}{\on{PreStk}}
\nc{\inftydgnearstack}{\on{NearStk}}
\nc{\inftydgstack}{\on{Stk}}
\nc{\inftydgstackalg}{DGSt\sfe^{1\text{-}alg}}
\nc{\inftydgprestack}{\on{PreStk}}
\nc{\dgindSch}{\on{indSch}}
\nc{\indSch}{{}^{\on{cl}}\!\on{indSch}}
\nc{\infSch}{\on{infSch}}
\nc{\dr}{{\on{dR}}}

\nc{\mmod}{{\on{-}\!{\mathbf{mod}}}}

\nc{\starr}{\text{\dh}}
\nc{\Spectra}{\on{Spectra}}
\nc{\Crys}{\on{Crys}}
\nc{\oblv}{{\mathbf{oblv}}}
\nc{\ind}{{\mathbf{ind}}}
\nc{\coind}{{\mathbf{coind}}}
\nc{\inv}{{\mathbf{inv}}}
\nc{\triv}{{\mathbf{triv}}}
\nc{\CMaps}{{\mathcal Maps}}
\nc{\Maps}{\on{Maps}}
\nc{\bMaps}{\mathbf{Maps}}
\nc{\BMaps}{\ul{\on{Maps}}}
\nc{\Grid}{\on{Grid}}
\nc{\hGrid}{\on{Grid}^{\geq\,\on{dgnl}}}
\nc{\Diag}{\on{Diag}}
\nc{\bDelta}{\mathbf{\Delta}}
\nc{\tCateg}{(\infty\on{-2)-Cat}}
\nc{\ul}{\underline}
\nc{\Seg}{\on{Seq}}
\nc{\biSeg}{\on{bi-Seq}}
\nc{\triSeg}{\on{tri-Seq}}
\nc{\quadSeg}{\on{quad-Seq}}
\nc{\nSeg}{\on{n-Seq}}
\nc{\Segm}{\on{Seg}^{\on{mkd}}}
\nc{\fLm}{\fL^{\on{mkd}}}
\nc{\inftyCatm}{\inftyCat^{\on{mkd}}}
\nc{\Blocks}{\mathbf{Blocks}}
\nc{\Snakes}{\mathbf{Snakes}}
\nc{\bifL}{\on{bi-}\!\fL}
\nc{\Sets}{\on{Sets}}
\nc{\Ran}{{\on{Ran}}}
\nc{\Vect}{\on{Vect}}
\nc{\Shv}{\on{Shv}}
\nc{\unn}{\mathbf{union}}
\nc{\Spc}{\on{Spc}}
\nc{\ppart}{(\!(t)\!)}
\nc{\qqart}{[\![t]\!]}
\nc{\Dmod}{\on{D-mod}}
\nc{\cD}{\mathcal D}
\nc{\ocD}{\overset{\circ}{\cD}}
\nc{\sfo}{\mathsf{o}}
\nc{\sfob}{\mathsf{ob}}
\nc{\sfp}{\mathsf{p}}
\nc{\sfq}{\mathsf{q}}
\nc{\DGCat}{\on{DGCat}}
\renc{\det}{\on{det}}
\nc{\Conf}{\on{Conf}}
\nc{\Whit}{\on{Whit}}
\nc{\Reg}{\on{Reg}}
\nc{\Res}{\on{Res}}
\nc{\BunNbox}{(\overline\Bun_N^{\omega^\rho})_{\infty\cdot x}} 
\nc{\BunNmbox}{(\overline\Bun_{N^-}^{\omega^\rho})_{\infty\cdot x}}
\nc{\bHecke}{\on{Hecke}_{\cG,\cT}}
\nc{\Hecke}{\on{Hecke}}
\nc{\bCZ}{\ol\CZ}
\nc{\oCZ}{\overset{\circ}\CZ} 
\nc{\boCZ}{\ol{\oCZ}}
\nc{\sotimes}{\overset{!}\otimes}
\nc{\semiinf}{{\frac{\infty}{2}}}
\nc{\coInd}{\on{coInd}}
\nc{\Ind}{\on{Ind}}
\nc{\bCM}{\overset{\bullet}\CM{}}
\nc{\bCF}{\overset{\bullet}\CF{}}
\nc{\SI}{\on{SI}}
\nc{\KL}{\on{KL}}

\begin{document}

\title[The semi-infinite IC sheaf-II]{The semi-infinite intersection cohomology sheaf-II: \\ The Ran space version}

\dedicatory{To Sasha Beilinson and Vitya Ginzburg} 

\author{Dennis Gaitsgory}

\date{\today}

\begin{abstract}
This paper is a sequel to \cite{Ga1}. We define the semi-infinite category on the Ran version of the
affine Grassmannian, and study a particular object in it, denoted $\ICs_\Ran$, which we call 
the \emph{semi-infinite intersection cohomology sheaf}.

\medskip

Unlike the situation of \cite{Ga1}, this $\ICs_\Ran$ is defined as the middle of extension of the constant 
(more precisely, dualizing) sheaf on the basic stratum, in a certain t-structure.
We give several explicit descriptions and characterizations of $\ICs_\Ran$: we describe its !- and *- stalks;
we present it explicitly as a colimit; we relate it to the IC sheaf of Drinfeld's relative compactification $\BunNb$; we describe
$\ICs_\Ran$ via the Drinfeld-Plucker formalism.
\end{abstract} 

\maketitle

\tableofcontents

\section*{Introduction}

\ssec{What are trying to do?}

\sssec{}

This paper is a sequel of \cite{Ga1}. In {\it loc. cit.} an attempt was made to construct a certain object, denoted $\ICs$, 
in the (derived) category $\Shv(\Gr_G)$ of sheaves on the affine Grassmannian, whose existence had been predicted by G.~Lusztig. 

\medskip

Notionally, $\ICs$ is supposed to be the intersection cohomology complex on the closure $\ol{S}{}^0$
of the unit $N\ppart$-orbit $S^0\subset \Gr_G$. 
Its stalks are supposed to be given by periodic Kazhdan-Lusztig polynomials. Ideally, one would want the construction of 
$\ICs$ to have the following properties:

\medskip

\begin{itemize}

\item It should be local, i.e., only depend on the formal disc, where we are thinking of $\Gr_G$ as $G\ppart/G\qqart$;

\smallskip

\item When our formal disc is the formal neighborhood of a point $x$ in a global curve $X$, then $\ICs$ should be the
pullback along the map $\ol{S}{}^0\to \BunNb$ of the intersection cohomology sheaf of $\BunNb$, where the latter
is Drinfeld's relative compactification of the stack of $G$-bundles equipped with a reduction to $N$ (which is an 
algebraic stack locally of finite type, so $\IC_{\BunNb}$ is well defined). 

\end{itemize}  

\medskip

The construction in \cite{Ga1} indeed produced such an object, but with the following substantial drawback: in {\it loc. cit.},
$\ICs$ was given by an \emph{ad hoc} procedure; namely, it was written as a certain explicit direct limit. In particular,
$\ICs$ was \emph{not} the middle extension of the constant\footnote{Technically, not constant but rather dualizing.} 
sheaf on $S^0$ with respect to the natural t-structure on $\Shv(\Gr_G)$ (however, $\ICs$ does belong to the heart
of this t-structure). 

\sssec{}

In the present paper we will construct a variant of $\ICs$, denoted $\ICs_\Ran$, closely related to $\ICs$,
that is actually given by the procedure of middle extension in a certain t-structure. 

\medskip

Namely, instead of the single copy of the affine Grassmannian $\Gr_G$, we will consider its Ran space version,
denoted $\Gr_{G,\Ran}$. We will equip the corresponding category $\Shv(\Gr_{G,\Ran})$ with a t-structure,
and we will define $\ICs_\Ran$ as the middle extension of the dualizing sheaf on $S^0_\Ran\subset \Gr_{G,\Ran}$. 

\begin{rem}
Technically, the Ran space is attached to a smooth (but not necessarily complete) curve $X$, and one may think
that this somewhat compromises the locality property of the construction of $\ICs_\Ran$. However, if one day a formalism becomes
available for working with the Ran space of a formal disc, the construction of $\ICs_\Ran$ will become purely local.
\end{rem} 

\sssec{}

For a fixed point $x\in X$, we have the embedding 
$$\Gr_G\simeq \{x\}\underset{\Ran}\times \Gr_{G,\Ran}\hookrightarrow \Gr_{G,\Ran},$$
and we will show that the restriction of $\ICs_\Ran$ along this map recovers $\ICs$ from \cite{Ga1}.

\sssec{}

Our $\ICs_\Ran$ retains the relation to $\IC_{\BunNb}$. Namely, we have a natural map
$$\ol{S}^0_\Ran\to \BunNb$$
and we will prove that $\ICs_\Ran$ identifies with the pullback of $\IC_{\BunNb}$ along this map.

\medskip

In particular, this implies the isomorphism
$$\ICs\simeq \IC_{\BunNb}|_{\ol{S}^0},$$ 
which had been established in \cite{Ga1} by a different method. 

\sssec{}

To summarize, we can say that we still do not know how to intrinsically characterize $\ICs$ on an individual $\Gr_G$
as an intersection cohomology sheaf, but we can do it, once we allow the point of the curve to move along the
Ran space. 

\medskip

But \emph{ce n'est pas grave}: as was argued in \cite[Sect. 0.4]{Ga1}, our $\ICs_\Ran$, equipped with
its factorization structure, is perhaps a more fundamental object than the original $\ICs$.  

\ssec{What is done in this paper?}

The main constructions and results of this paper can be summarized as follows: 

\sssec{}

We define the \emph{semi-infinite} category on the Ran version of the affine Grassmannian, denoted $\SI_\Ran$, 
and equip it with a t-structure. This is largely parallel to \cite{Ga1}. 

\medskip

We define $\ICs_\Ran\in \SI_\Ran$ as the middle extension of the dualizing sheaf on the stratum $S^0_\Ran\subset \Gr_{G,\Ran}$. 
(We will also show that the corresponding !- and *- extensions both belong to the heart of the t-structure, see \propref{p:! and * perv};
this contrasts with the situation for $\ICs$, see \cite[Theorem 1.5.5]{Ga1}). 

\medskip

We describe explicitly the !- and *-restrictions of $\ICs_\Ran$ to the strata $S^\lambda_\Ran\subset \ol{S}{}^0_\Ran\subset \Gr_{G,\Ran}$
(here $\lambda$ is an element of $\Lambda^{\on{neg}}$, the negative span of positive simple coroots), see \thmref{t:descr of restr}.  
These descriptions are given in terms of the combinatorics of the Langlands dual Lie algebra; more precisely, in terms, of the
\emph{factorization algebras} attached to $\CO(\cN)$ and $U(\cn^-)$. 

\medskip

We give an explicit presentation of $\ICs_\Ran$ as a colimit (parallel to the \emph{definition} of $\ICs$ in \cite{Ga1}),
see \thmref{t:descr as colimit}. This implies the identification $\ICs_\Ran|_{\Gr_G}\simeq \ICs$, where $\ICs\in \Shv(\Gr_G)$
is the object constructed in \cite{Ga1}. 

\sssec{}

We show that $\ICs_\Ran$ identifies canonically (up to a cohomological shift by $[d]$, $d=\dim(\Bun_N)$) 
with the pullback of $\IC_{\BunNb}$ along the map
\begin{equation} \label{e:pi Ran}
\ol{S}{}^0\to \BunNb,
\end{equation} 
see \thmref{t:IC loc glob}. 

\medskip

In fact, we show that the above pullback functor is t-exact (up to the shift by $[d]$), when restricted to the subcategory 
$\SI^{\leq 0}_{\on{glob}}\subset \Shv(\BunNb)$ 
that consists of objects equivariant with respect to the action of the adelic $N$, see \corref{c:log glob exact}. 

\medskip 

The proof of this t-exactness property is based on applying Braden's theorem to $\Gr_{G,\Ran}$ and the
\emph{Zastava space}. 

\medskip

We note that,
unlike \cite{Ga1}, the resulting proof of the isomorphism 
\begin{equation} \label{e:IC glob vs loc Intro}
\IC_{\BunNb}|_{\ol{S}{}^0_\Ran}[d]\simeq \ICs_\Ran
\end{equation} 
\emph{does not} use the computation of $\IC_{\BunNb}$ from \cite{BFGM}, but rather reproves it. 

\medskip

As an aside we prove an important geometric fact that the map \eqref{e:pi Ran} is \emph{universally homologically contractible}
(=the pullback functor along any base change of this map is fully faithful), see \thmref{t:contr}. 

\sssec{}

We show that $\ICs$ has a \emph{unitality} property: it stays invariant under the operation of ``throwing in" more points in $\Ran$
without altering the $G$-bundle.

\medskip

We use the unitality property of $\ICs$ to equip it with a \emph{factorization structure}.

\sssec{}

We show that $\ICs_\Ran$ has an eigen-property with respect to the action of the Hecke functors for $G$ and $T$,
see \thmref{t:Hecke Ran}. 

\medskip

In the course of the proof of this theorem, we give yet another characterization of $\ICs_{\Ran}$ (which works for
$\ICs$ as well):

\medskip

We show that the $\delta$-function $\delta_{1_{\Gr,\Ran}}\in \Shv(\Gr_{G,\Ran})$ on the unit section $\Ran\to \Gr_{G,\Ran}$ 
possesses a natural \emph{Drinfeld-Plucker} structure with respect to the Hecke actions of $G$ and $T$ (see \secref{ss:DrPl} for what this means),
and that $\ICs_{\Ran}$ can be obtained from $\delta_{1_{\Gr,\Ran}}$ by applying the functor from the Drinfeld-Pl\"ucker category
to the graded Hecke category, left adjoint to the tautological forgetful functor (see \secref{ss:DrPl on IC}).

\medskip

Finally, we establish the compatibility of the isomorphism \eqref{e:IC glob vs loc Intro} with the Hecke eigen-structures on 
$\ICs_{\Ran}$ and $\IC_{\BunNb}$, respectively (see \thmref{t:Hecke compat}).

\ssec{Organization}

\sssec{}

In \secref{s:semiinf} we recall the definition of the Ran space $\Ran$, the Ran version of the affine Grassmannian $\Gr_{G,\Ran}$,
and the stratification of the closure $\ol{S}{}^0_\Ran$  of the adelic $N$-orbit $S^0_\Ran$ by locally closed substacks $S^\lambda_\Ran$. 

\medskip

We define the semi-infinite category $\SI_{\Ran}$ and study the standard functors that link it to the corresponding categories
on the strata. 

\sssec{}

In \secref{s:t} we define the t-structure on $\SI^{\leq 0}_{\Ran}$ and our main object of study, $\ICs_\Ran$. 

\medskip

We state \thmref{t:descr of restr} that describes the *- and !- restrictions of $\ICs_\Ran$ to the strata $S^\lambda_\Ran$.
The proof of the statement concerning *-restrictions will be given in this same section (it will result from \thmref{t:descr as colimit} 
mentioned below). The proof of the statement concerning !-restrictions will be given in \secref{s:glob}. 

\medskip

We state and prove \thmref{t:descr as colimit} that gives a presentation of $\ICs_\Ran$ as a colimit. 

\sssec{}

In \secref{s:glob}, we recall the definition of Drinfeld's relative compactification $\BunNb$. 

\medskip

We define the global semi-infinite category $\SI^{\leq 0}_{\on{glob}}\subset \Shv(\BunNb)$. We prove that the pullback functor
along \eqref{e:pi Ran}, viewed as a functor
$$\SI^{\leq 0}_{\on{glob}}\to \SI^{\leq 0}_\Ran,$$
is t-exact (up to the shift by $[d]$). From here we deduce the identification \eqref{e:IC glob vs loc Intro},
which is \thmref{t:IC loc glob}. 

\medskip

We also state \thmref{t:contr}, whose proof is given in \secref{s:app}.  

\sssec{}

In \secref{s:unital} we introduce the notion of \emph{unital} subcategory inside $\Shv(\Gr_{G,\Ran})$, $\Shv(\ol{S}^0_\Ran)$ and
$\SI^{\leq 0}_\Ran$, and we show that $\ICs$ belongs to $\SI^{\leq 0}_\Ran$. 

\medskip

We use this property of $\ICs$ to equip it with a factorization structure. 

\sssec{}

In \secref{s:H} we establish the Hecke eigen-property of $\ICs_\Ran$. In the process of doing so we discuss the
formalism of \emph{lax central objects} and \emph{Drinfeld-Pl\"ucker} structures, and their relation to the
Hecke eigen-structures.

\medskip

In \secref{s:H g-l} we prove the compatibility between the eigen-property of $\ICs_\Ran$ and that of $\IC_{\BunNb}$. 

\ssec{Background, conventions and notation}

The notations and conventions in this follow closely those of \cite{Ga1}. Here is a summary: 

\sssec{}

This paper uses higher category theory. It appears already in the definition of our
basic object of study, the \emph{category of sheaves} on the Ran Grassmannian, $\Gr_{G,\Ran}$. 

\medskip

Thus, we will assume that the reader is familiar with the basics of higher categories and higher algebra. The fundamental
reference is \cite{Lu1,Lu2}, but shorter expositions (or user guides) exist as well, for example, the first
chapter of \cite{GR}. 

\sssec{}

Our algebraic geometry happens over an arbitrary algebraically closed ground field $k$. Our
geometric objects are classical (i.e., this paper \emph{does not} need derived algebraic geometry). 

\medskip

We let $\affSch_{\on{ft}}$ denote the category of (classical) affine schemes
of finite type over $k$. 

\medskip

By a prestack (locally of finite type) we mean an arbitrary functor
\begin{equation} \label{e:prestack}
(\affSch_{\on{ft}})^{\on{op}}\to \on{Groupoids}
\end{equation}
(we de not need to consider higher groupoids). 

\medskip

We let $\on{PreSk}_{\on{lft}}$ denote the category of such prestacks.
It contains $\affSch_{\on{ft}}$ via the Yoneda embedding. All other types of geometric objects (schemes, algebraic stacks,
ind-schemes) are prestacks with some specific \emph{properties} (but \emph{not additional pieces of structure}). 

\sssec{}

We let $G$ be a connected reductive group over $k$. We fix a Borel subgroup $B\subset G$ and the opposite Borel $B^-\subset G$.
Let $N\subset B$ and $N^-\subset B^-$ denote their respective unipotent radicals.

\medskip

Set $T=B\cap B^-$; this is a Cartan subgroup of $G$. We use it to identify the quotients
$$B/N \simeq T \simeq B^-/N^-.$$

\medskip

We let $\Lambda$ denote the coweight lattice of $G$, i.e., the lattice of cocharacters of $T$.
We let $\Lambda^{\on{pos}}\subset \Lambda$ denote the sub-monoid consisting of 
linear combinations of positive simple roots with non-negative integral coefficients. 
We let $\Lambda^+\subset \Lambda$ denote the sub-monoid of \emph{dominant coweights}.

\sssec{}

While our geometry happens over a field $k$, the representation-theoretic categories
that we study are \emph{DG categories} over another field, denoted $\sfe$ (assumed algebraically closed
and of characteristic $0$). For a crash course on DG categories, the reader is referred to \cite[Chapter 1, Sect. 10]{GR}.  

\medskip

All our DG categories are assumed presentable. When considering functors,
we will only consider functors that preserve colimits. We denote the $\infty$-category of DG categories
by $\DGCat$.  It carries a symmetric monoidal structure (i.e., one can consider tensor products
of DG categories). The unit object is the DG category of complexes of $\sfe$-vector spaces,
denoted $\Vect$. 

\medskip

We will use the notion of t-structure on a DG category. Given a t-structure on $\CC$, we will denote by
$\CC^{\leq 0}$ the corresponding subcategory of cohomologically connective objects, and by $\CC^{>0}$
its right orthogonal. We let $\CC^\heartsuit$ denote the heart $\CC^{\leq 0}\cap \CC^{\geq 0}$. 

\sssec{}  \label{sss:sheaf theory}

The source of DG categories will be a \emph{sheaf theory}, which is a functor
$$\Shv:(\affSch_{\on{ft}})^{\on{op}}\to \DGCat, \quad Y\mapsto \Shv(Y).$$

For a morphism of affine schemes $f:Y_0\to Y_1$, the corresponding functor
$$\Shv(Y_1)\to \Shv(Y_0)$$
is the !-pullback $f^!$.  

\medskip

We will work with the following particular examples sheaf theories are:

\medskip

\noindent{(i)} We take $\sfe=\ol\BQ_\ell$ and we take $\Shv(Y)$ to be the ind-completion of the
(small) DG category of constructible $\ol\BQ_\ell$-sheaves.

\medskip

\noindent{(ii)} When $k=\BC$ and $\sfe$ arbitrary, we take $\Shv(Y)$ to be the ind-completion of the
(small) DG category of constructible $\sfe$-sheaves on $Y(\BC)$ in the analytic topology. 

\medskip

\noindent{(iii)} If $k$ has characteristic $0$, we take $\sfe=k$ and we take $\Shv(Y)$ to be the DG category
of holonomic D-modules on $S$;

\medskip

\noindent{(iv)} If $k$ has characteristic $0$, we take $\sfe=k$ and we take $\Shv(Y)$ to be the DG category
of D-modules on $Y$.

\medskip

We will refer to examples (i), (ii) and (iii) as \emph{a constructible} sheaf theories. 

\medskip

In the constructible case, the functor $f^!$ always has a left adjoint, denoted $f_!$. In example (iv) this is not
the case. However, the partially defined left adjoint $f_!$ is defined on holonomic objects.  It is 
defined on the entire category if $f$ is proper. 

\sssec{Sheaves on prestacks}   \label{sss:sheaves intro} 

We apply the procedure of right Kan extension along the embedding
$$(\affSch_{\on{ft}})^{\on{op}}\hookrightarrow (\on{PreStk}_{\on{lft}})^{\on{op}}$$
to the functor $\Shv$, and thus obtain a functor (denoted by the same symbol)
$$\Shv:(\on{PreStk}_{\on{lft}})^{\on{op}}\to \DGCat.$$

By definition, for $\CY\in \on{PreStk}_{\on{lft}}$ we have
\begin{equation} \label{e:shv on prestack}
\Shv(\CY)=\underset{S\in \affSch_{\on{ft}},y:S\to \CY}{\on{lim}}\, \Shv(S),
\end{equation} 
where the transition functors in the formation of the limit are the 
!-pullbacks\footnote{Note that even though the index category (i.e., $(\affSch_{\on{ft}})_{/\CY}$) is ordinary, the above limit
is formed in the $\infty$-category $\DGCat$. This is how $\infty$-categories appear in this paper.}.

\medskip

For a map of prestacks $f:\CY_0\to \CY_1$ we thus have a well-defined pullback functor
$$f^!:\Shv(\CY_1)\to \Shv(\CY_0).$$

\medskip

We denote by $\omega_\CY$ the dualizing sheaf on $\CY$, i.e., the pullback of 
$$\sfe\in \Vect\simeq \Shv(\on{pt})$$
along the tautological map $\CY\to \on{pt}$. 

\sssec{}

We let $X$ be a smooth, connected (but not necessarily proper) curve over $k$. Whenever we need $X$ to be proper,
we will explicitly say so. 

\sssec{}  \label{sss:dual geom}

This paper is closely related to the geometric Langlands theory, and the geometry of the Langlands dual group
$\cG$ makes it appearance. 

\medskip

By definition, $\cG$ is a reductive group over $\sfe$ and geometric objects constructed out of $\cG$ give
rise to $\sfe$-linear DG categories by considering quasi-coherent (resp., ind-coherent) sheaves on them.

\medskip

The most basic example of such a category is
$$\QCoh(\on{pt}/\cG)=:\Rep(\cG).$$

\ssec{Acknowledgements} The author would like to thank S.~Raskin for his suggestion to consider the formalism
of Drinfeld-Pl\"ucker structures, as well as numerous stimulating discussions. 

\medskip

The author is grateful to M.~Finkelberg
for igniting his interest in the semi-infinite IC sheaf, which has been on author's mind for some 20 years now.  

\medskip 

The author is grateful to J.~Campbell and L.~Chen for pointing out important mistakes in an earlier version of the
paper. 

\medskip

The author is supported by NSF grant DMS-1063470. He has also received support from ERC grant 669655.

\section{The Ran version of the semi-infinite category}  \label{s:semiinf}

In this section we extend the definition of the semi-infinite category given in \cite{Ga1} from the affine Grassmannian
$\Gr_{G,x}$ corresponding to a fixed point $x\in X$ to the Ran version, denoted $\Gr_{G,\Ran}$. 

\ssec{The Ran Grassmannian} 

\sssec{}

We recall that the Ran space of $X$, denoted $\Ran$, is the prestack that assigns to an affine test-scheme $Y$ the set of finite non-empty
subsets
$$\CI\subset \Hom(Y,X).$$

\medskip

One can explicitly exhibit $\Ran$ as a colimit (in $\on{PreStk}$) of schemes:
$$\Ran\simeq \underset{I}{\on{colim}}\, X^I,$$
where the colimit is taken over the category opposite to the category $\on{Fin}^{\on{surj}}$ of finite non-empty sets and surjective maps,
where to a map $\phi:I\twoheadrightarrow J$ we assign the corresponding diagonal embedding 
$$X^J \overset{\Delta_\phi}\hookrightarrow X^I.$$

This description implies, in particular, that if $X$ is proper, then $\Ran$ is pseudo-proper as a prestack
(see \secref{sss:pseudo-proper} for what it means). 

\medskip

Another key feature of $\Ran$ is that it is homologically contractible (see \secref{sss:UHC} for what this means). 

\sssec{}

We will consider the Ran version of the affine Grassmannian, denoted $\Gr_{G,\Ran}$, defined as follows.

\medskip

It assigns to an affine test-scheme $Y$, the set of triples $(\CI,\CP_G,\alpha)$, where $\CI$ is a $Y$-point of $\Ran$,
$\CP_G$ is a $G$-bundle on $Y\times X$, and $\alpha$ is a trivialization of $\CP_G$ on the open subset of $Y\times X$
equal to the complement of the union $\Gamma_\CI$ of the graphs of the maps $Y\to X$ that comprise $\CI$. 

\medskip

The projection $\Gr_{G,\Ran}\to \Ran$ is pseudo-proper. 

\medskip

We will also consider the Ran versions of the loop and arc groups (ind)-schemes, denoted $$\fL^+(G)_{\Ran}\subset \fL(G)_{\Ran}.$$ 
The Ran Grassmannian $\Gr_{G,\Ran}$ identifies with the \'etale (equivalently, fppf) sheafification of the prestack 
quotient $\fL(G)_{\Ran}/\fL^+(G)_{\Ran}$.

\sssec{}

For a fixed finite non-empty set $I$, we denote
$$\Gr_{G,I}:=X^I\underset{\Ran}\times \Gr_{G,\Ran},\quad 
\fL(G)_I:=X^I\underset{\Ran}\times \fL(G)_{\Ran},\quad \fL^+(G)_I:=X^I\underset{\Ran}\times \fL^+(G)_{\Ran}.$$

\medskip

For a map of finite sets $\phi:I\twoheadrightarrow J$, we will denote by $\Delta_\phi$ the corresponding map
$\Gr_{G,J}\to \Gr_{G,I}$, so that we have the Cartesian square:
$$
\CD
\Gr_{G,J}   @>{\Delta_\phi}>>   \Gr_{G,I}  \\
@VVV   @VVV   \\
X^J @>{\Delta_\phi}>>  X^I.
\endCD
$$

\sssec{}

We introduce also the following closed (resp., locally closed) subfunctors 
$$S^0_\Ran\subset \ol{S}{}^0_\Ran\subset \Gr_{G,\Ran}.$$

\medskip

Namely, for an affine test-scheme $Y$, a $Y$-point $(\CI,\CP_G,\alpha)$ belongs to $\ol{S}{}^0_\Ran$ if for every dominant weight $\check\lambda$, the composite 
meromorphic map of vector bundles on $Y\times X$
\begin{equation} \label{e:Plucker map}
\CO\to \CV^{\check\lambda}_{\CP^0_G}\overset{\alpha}\longrightarrow \CV^{\check\lambda}_{\CP_G}
\end{equation} 
is regular. In the above formula the notations are as follows:

\begin{itemize}

\item $\CV^{\check\lambda}$ denotes the Weyl module over $G$ with highest weight $\check\lambda$;

\smallskip 

\item $\CV^{\check\lambda}_{\CP_G}$ (resp., $\CV^{\check\lambda}_{\CP^0_G}$) denotes the vector bundles
associated with $\CV^{\check\lambda}$ and the $G$-bundle $\CP_G$ (resp., the trivial $G$-bundle $\CP^0_G$);

\smallskip 

\item $\CO\to \CV^{\check\lambda}_{\CP^0_G}$ is the map coming from the highest weight vector in $\CV^{\check\lambda}$. 

\end{itemize} 

\medskip

A point as above belongs to $S^0_\Ran$ if the above composite map is an injective bundle map (i.e., the cokernel is
flat over $Y\times X$). 

\medskip


\medskip




\ssec{The semi-infinite category}  \label{ss:SI}

\sssec{}

Since $\Gr_{G,\Ran}$ a prestack locally of finite type, we have a well-defined category
$$\Shv(\Gr_{G,\Ran}).$$

We have
$$\Shv(\Gr_{G,\Ran}):=\underset{I}{\on{lim}}\, \Shv(\Gr_{G,I}),$$
where the limit is formed using the !-pullback functors. 

\sssec{}

Although the group ind-scheme $\fL(N)_{\Ran}$ is \emph{not} locally of finite type, we have a well-defined full 
subcategory 
$$\SI_{\Ran}:=\Shv(\Gr_{G,\Ran})^{\fL(N)_{\Ran}}\subset \Shv(\Gr_{G,\Ran}).$$ 

\medskip

Namely, for every fixed finite non-empty set $I$, we consider the full subcategory
$$\SI_I:=\Shv(\Gr_{G,I})^{\fL(N)_I}\subset \Shv(\Gr_{G,I}),$$
defined as in \cite[Sect. 1.2]{Ga1}. 

\medskip

We say that the object of $\Shv(\Gr_{G,\Ran})$ belongs to $\Shv(\Gr_{G,\Ran})^{\fL(N)_{\Ran}}$
if its restriction to any $\Gr_{G,I}$ yields an object of $\Shv(\Gr_{G,I})^{\fL(N)_I}$. By construction,
we have an equivalence
$$\SI_{\Ran}:=\underset{I}{\on{lim}}\, \SI_I.$$

\sssec{}

Let $\SI_{\Ran}^{\leq 0}\subset \SI_{\Ran}$ be the full subcategory consisting of objects supported on 
$\ol{S}{}^0_\Ran$. I.e., 
$$\SI_{\Ran}^{\leq 0}=\SI_{\Ran}\cap \Shv(\ol{S}{}^0_\Ran),$$
while
$$\Shv(\ol{S}{}^0_\Ran)\simeq \underset{I}{\on{lim}}\, \Shv(\ol{S}{}^0_I).$$

\ssec{Stratification}

In order to study the structure of $\SI_{\Ran}^{\leq 0}$, we will now describe a certain natural 
stratification of $\ol{S}{}^0_\Ran$, whose open stratum will be $S^0_\Ran$.

\sssec{}
 
For $\lambda\in \Lambda^{\on{neg}}$, let $X^\lambda$ denote the corresponding partially symmetrized power of $X$.
I.e., if
$$\lambda=\underset{i}\Sigma\, (-n_i)\cdot \alpha_i,\quad n_i\geq 0$$
where $\alpha_i$ are simple positive coroots, then
$$X^\lambda=\underset{i}\Pi\, X^{(n_i)}.$$

\medskip

In other words, $Y$-points of $X^\lambda$ are effective $\Lambda^{\on{neg}}$-valued divisors on $X$. 

\medskip

For $\lambda=0$ we by definition have $X^0=\on{pt}$.

\sssec{}

Let 
$$(X^\lambda \times \Ran)^{\subset} \subset \Ran\times X^\lambda$$
be the ind-closed subfunctor, whose $S$-points are pairs $(\CI,D)$ for which the support of the divisor $D$
is \emph{set-theoretically} supported on the union of the graphs of the maps $S\to X$ that comprise
$\CI$. 

\medskip

In other words,
$$(X^\lambda \times \Ran)^{\subset}=\underset{I}{\on{colim}}\, (X^\lambda\times X^I)^{\subset},$$
where 
$$(X^\lambda\times X^I)^{\subset} \subset X^I\times X^\lambda$$ is 
the formal completion of the corresponding incidence subvariety. 

\medskip

For future use we note:

\begin{lem} \label{l:pr contr}
The map
$$\on{pr}^\lambda_\Ran:(X^\lambda \times \Ran)^{\subset}\to X^\lambda$$
is universally homologically contractible\footnote{The notion of universal
local contractibility is recalled in \secref{sss:UHC}}.
\end{lem} 

The proof in the case when $X$ is proper will be given in \secref{sss:proof of pr contr}.
For the proof in the general case see Remark \ref{r:UHC}.

\begin{cor} \label{c:pr contr}
The pullback functor 
$$(\on{pr}^\lambda_\Ran)^!:\Shv(X^\lambda)\to \Shv((X^\lambda \times \Ran)^{\subset})$$
is fully faithful.
\end{cor}

\sssec{}

We let
$$\ol{S}{}^\lambda_{\Ran}\subset (X^\lambda \times \Ran)^{\subset}\underset{\Ran}\times \Gr_{G,\Ran}$$
be the closed subfunctor defined by the following condition:

\medskip

An $S$-point $(\CI,D,\CP_G,\alpha)$ of the fiber product $(X^\lambda \times \Ran)^{\subset}\underset{\Ran}\times \Gr_{G,\Ran}$
belongs to $\ol{S}{}^\lambda_{\Ran}$ if for every dominant weight $\check\lambda$ the map \eqref{e:Plucker map} extends to a regular map
\begin{equation} \label{e:Plucker div}
\CO(-\check\lambda(D))\to \CV^{\check\lambda}_{\CP_G}.
\end{equation} 

We denote by $\ol\bi{}^\lambda$ the composite map
$$\ol{S}{}^\lambda_{\Ran}\to (X^\lambda \times \Ran)^{\subset}\underset{\Ran}\times \Gr_{G,\Ran}\to \Gr_{G,\Ran}.$$

This map is proper, and its image is contained in $\ol{S}{}^0_\Ran$. 

\medskip

Note that for $\lambda=0$, the map $\ol\bi{}^0$ is the identity map on $\ol{S}{}^0_\Ran$.

\medskip 

Let $\ol{p}^\lambda_\Ran$ denote the projection
$$\ol{S}{}^\lambda_{\Ran}\to (X^\lambda \times \Ran)^{\subset}.$$

\sssec{}

We define the open subfunctor
$$S^\lambda_{\Ran}\subset \ol{S}{}^\lambda_{\Ran}$$
to correspond to those quadruples $(\CI,D,\CP_G,\alpha)$ for which the map \eqref{e:Plucker div} is an injective bundle map
(i.e., the cokernel is flat over $Y\times X$). 

\medskip

The projection
\begin{equation} \label{e:stratum to config}
p^\lambda_{\Ran}:=\ol{p}{}^\lambda_{\Ran}|_{S^\lambda_{\Ran}}:S^\lambda_{\Ran}\to (X^\lambda \times \Ran)^{\subset}
\end{equation}
admits a canonically defined section
\begin{equation} \label{e:config to stratum}
s^\lambda_{\Ran}:(X^\lambda \times \Ran)^{\subset}\to S^\lambda_{\Ran}.
\end{equation}

Namely, it sends $(\CI,D)$ to the quadruple $(\CI,D,\CP_G,\alpha)$, where $\CP_G$ is the 
$G$-bundle induced from the $T$-bundle $\CP_T:=\CP^0_T(D)$, and $\alpha$ is the trivialization of
$\CP_G$ induced by the tautological trivialization of $\CP_T$ away from the support of $D$. 

\sssec{}

We let
$$\bj^\lambda:S^\lambda_{\Ran}\hookrightarrow \ol{S}{}^\lambda_{\Ran}, \quad 
\bi^\lambda=\ol\bi{}^\lambda\circ \bj^\lambda:S^\lambda_\Ran\to \Gr_{G,\Ran}$$
denote the resulting maps.

\medskip

For a fixed finite non-empty set $I$, we obtain the corresponding subfunctors
$$\ol{S}{}^\lambda_I \subset (X^\lambda\times X^I)^{\subset}\underset{X^I}\times \Gr_{G,I}$$
and
$$S^\lambda_I\subset \ol{S}{}^\lambda_I,$$
and maps, denoted by the same symbols $\bj^\lambda$, $\ol\bi{}^\lambda$, $\bi^\lambda$. Let
$p^\lambda_I$ (resp., $\ol{p}{}^\lambda_I$) denote the resulting map from $S^\lambda_I$
(resp., $\ol{S}{}^\lambda_I$) to $(X^\lambda\times X^I)^{\subset}$. 

\medskip

Let $s^\lambda_I$ denote the resulting section
$$s^\lambda_I:(X^\lambda\times X^I)^{\subset}\to S^\lambda_I.$$

\sssec{}

The following results easily from the definitions:

\begin{lem}
The maps
$$\bi^\lambda:S^\lambda_\Ran\to \ol{S}{}^0_\Ran \text{ and } S^\lambda_I\to \ol{S}{}^0_I$$
are locally closed embeddings. Every field-valued point of $\ol{S}{}^0_\Ran$ (resp., $\ol{S}{}^0_I$)
belongs to the image of exactly one such map.
\end{lem}

\sssec{}

In what follows we will denote by
\begin{equation} \label{e:semiinf strata}
\SI^{\leq \lambda}_\Ran\subset \Shv(\ol{S}{}^\lambda_{\Ran}) \text{ and } \SI^{=\lambda}_\Ran\subset \Shv(S^\lambda_{\Ran}),
\end{equation} 
and also
$$\SI^{\leq \lambda}_I\subset \Shv(\ol{S}{}^\lambda_I) \text{ and } \SI^{=\lambda}_I\subset \Shv(S^\lambda_I),$$
the corresponding full subcategories. 

\ssec{The category on a single stratum}

\sssec{}

We have the following explicit description of the category on each stratum separately:

\begin{prop} \label{p:on stratum}
Pullback along the map $p^\lambda_{\Ran}$ of \eqref{e:stratum to config}
defines an equivalence
$$\Shv((X^\lambda \times \Ran)^{\subset})\to \SI^{=\lambda}_\Ran.$$
The inverse equivalence is given by restriction to the section $s^\lambda_{\Ran}$ of \eqref{e:config to stratum}, 
and similarly for $\Ran$ replaced by $X^I$ for an individual $I$.
\end{prop}

\sssec{Proof of \propref{p:on stratum}}
Follows from the fact that the action of the group ind-scheme
$$(X^\lambda \times \Ran)^{\subset}\underset{\Ran}\times \fL(N)_{\Ran}$$
on $S^\lambda_\Ran$ is transitive along the fibers of the map \eqref{e:stratum to config},
with the stabilizer of the section $s^\lambda_{\Ran}$ being a pro-unipotent 
group-scheme over $(X^\lambda \times \Ran)^{\subset}$.

\qed 

\ssec{Interaction between the strata}

\sssec{}

Consider the subcategories \eqref{e:semiinf strata}. 
%
The maps $\bj^\lambda$, $\ol\bi{}^\lambda$ and $\bi^\lambda$ induce functors
$$(\ol\bi{}^\lambda)_!:=(\ol\bi{}^\lambda)_*:\SI^{\leq \lambda}_\Ran\to \SI^{\leq 0}_\Ran,$$
$$(\ol\bi{}^\lambda)^!:\SI^{\leq 0}_\Ran\to \SI^{\leq \lambda}_\Ran;$$
$$(\bj^\lambda)^*:=(\bj^\lambda)^!:\SI^{\leq \lambda}_\Ran\to \SI^{=\lambda}_\Ran;$$
$$(\bj^\lambda)_*:\SI^{=\lambda}_\Ran\to \SI^{\leq \lambda}_\Ran;$$
$$(\bi^\lambda)^!\simeq (\bj^\lambda)^!\circ (\ol\bi{}^\lambda)^!:\SI^{\leq 0}_\Ran\to \SI^{=\lambda}_\Ran;$$
$$(\bi^\lambda)_*\simeq (\ol\bi{}^\lambda)_*\circ (\bj^\lambda)_*:\SI^{=\lambda}_\Ran\to \SI^{\leq 0}_\Ran.$$

\medskip

The same applies to  $\Ran$ replaced by $X^I$ for a fixed finite non-empty set $I$.

\sssec{}

In \secref{ss:Braden} we will prove:

\begin{prop} \label{p:i* well defined} \hfill

\smallskip

\noindent{\em(a)} For a fixed finite set $I$, the left adjoint of
$$(\bi^\lambda)_*:\SI^{=\lambda}_I\to \SI^{\leq 0}_I$$
is defined as a functor
$$\SI^{\leq 0}_I\to \SI^{=\lambda}_I;$$
to be denoted by $(\bi^\lambda)^*$. 

\smallskip

\noindent{\em(b)} For $\CF\in \SI^{\leq 0}_I$ and $\CF'\in \Shv(X^I)$, the map\footnote{In the formula below $-|_{(X^\lambda\times X^I)^{\subset}}$
denotes the !-pullback along the projection $(X^\lambda\times X^I)^{\subset}\to X^I$.} 
$$(\bi^\lambda)^*((\ol{p}{}^0_I)^!(\CF')\sotimes \CF)\to (p_I^\lambda)^!(\CF'|_{(X^\lambda\times X^I)^{\subset}}
\sotimes (\bi^\lambda)^*(\CF))$$ is an isomorphism. 

\smallskip

\noindent{\em(c)} For a map of finite sets $\phi:I\twoheadrightarrow J$, the natural transformation
$$(\bi^\lambda)^*\circ (\Delta_\phi)^!\to (\Delta_\phi)^!\circ (\bi^\lambda)^*, \quad \SI^{\leq 0}_I \rightrightarrows \SI^{=\lambda}_J$$
is an isomorphism.

\end{prop}

\begin{rem}  \label{r:unambiguous}

Let $\CF\in \SI^{\leq 0}_I$, be such that the \emph{partially defined} left adjoint $(\bi^\lambda)^*$ of 
\begin{equation} \label{e:i lambda *}
(\bi^\lambda)_*:\Shv(S^\lambda_I)\to \Shv(\ol{S}{}^0_I)
\end{equation}
is defined on $\CF$, viewed as an object of $\Shv(\ol{S}{}^0_I)$. 

\medskip

(Note that the condition of point (a') of \propref{p:i* well defined} is always satisfied in the context of constructible sheaves. 
In the context of D-modules, it is satisfied if, for example, $\CF$ is ind-holonomic.)

\medskip

Then it follows formally that the resulting object of $\Shv(S^\lambda_I)$ equals 
$$(\bi^\lambda)^*(\CF)\in \SI^{=\lambda}_I\subset \Shv(S^\lambda_I),$$
where $(\bi^\lambda)^*$ is understood in the sense of point point (a) of \propref{p:i* well defined}. 

\medskip

In other words, for such $\CF$, the notation $(\bi^\lambda)^*(\CF)$ is unambiguous. 

\medskip

A similar remark applies to the functor $(\bi^\lambda)_!$ studied in \corref{c:i! well defined} below.

\end{rem}

\sssec{}

From \propref{p:i* well defined}, by a formal Cousin argument, we obtain:

\begin{cor} \label{c:i! well defined} \hfill

\smallskip

\noindent{\em(a)} For a fixed finite set $I$, the left adjoint of
$$(\bi^\lambda)^!:\SI^{\leq 0}_I \to \SI^{=\lambda}_I$$
is defined as a functor
$$\SI^{=\lambda}_I\to \SI^{\leq 0}_I;$$
to be denoted $(\bi^\lambda)_!$.

\smallskip

\noindent{\em(b)} For $\CF\in \SI^{=\lambda}_I$ and $\CF'\in \Shv(X^I)$, the map
$$(\bi^\lambda)_!((p^\lambda_I)^!(\CF'|_{(X^\lambda\times X^I)^{\subset}}))\sotimes \CF)\to (\ol{p}{}^0_I)^!(\CF')\sotimes (\bi^\lambda)_!(\CF)$$
is an isomorphism.

\smallskip

\noindent{\em(c)} For a map of finite sets $\phi:I\twoheadrightarrow J$, the natural transformation
$$(\bi^\lambda)_!\circ (\Delta_\phi)^!\to (\Delta_\phi)^!\circ (\bi^\lambda)_!, \quad \SI^{=\lambda}_I  \rightrightarrows \SI^{\leq 0}_J.$$

\end{cor} 

\sssec{}

Passing to the limit over $I\in \on{Fin}^{\on{surj}}$, we obtain: 

\begin{cor} \label{c:well defined Ran} \hfill 

\smallskip

\noindent{\em(a)} The left adjoint of 
$$(\bi^\lambda)_*:\SI^{=\lambda}_\Ran\to \SI^{\leq 0}_\Ran$$
is defined as a functor
$$\SI^{\leq 0}_\Ran\to \SI^{=\lambda}_\Ran;$$
to be denoted by $(\bi^\lambda)^*$. 

\smallskip

\noindent{\em(b)} The left adjoint of
$$(\bi^\lambda)^!:\SI^{\leq 0}_\Ran \to \SI^{=\lambda}_\Ran$$
is defined as a functor
$$\SI^{=\lambda}_\Ran\to \SI^{\leq 0}_\Ran,$$
to be denoted $(\bi^\lambda)_!$.

\smallskip

\noindent{\em(c)} For $\CF\in \SI^{\leq 0}_\Ran$ and $\CF'\in \Shv(\Ran)$, the map
$$(\bi^\lambda)^*((\ol{p}{}^0_\Ran)^!(\CF')\sotimes \CF)\to (p_\Ran^\lambda)^!(\CF'|_{(X^\lambda \times \Ran)^{\subset}}))
\sotimes (\bi^\lambda)^*(\CF)$$
is an isomorphism. 

\smallskip

\noindent{\em(d)} For $\CF\in \SI^{=\lambda}_\Ran$ and $\CF'\in \Shv(\Ran)$, the map
$$(\bi^\lambda)_!((p^\lambda_\Ran)^!(\CF'|_{(X^\lambda \times \Ran)^{\subset}}))\sotimes \CF)\to 
(\ol{p}{}^0_\Ran)^!(\CF')\sotimes (\bi^\lambda)_!(\CF)$$
is an isomorphism.

\end{cor}

\begin{rem}
A slight variation of the proof of \propref{p:i* well defined} shows that the assertions of \corref{c:well defined Ran}
remain valid for $\bi^\lambda$ replaced
by $\ol\bi{}^\lambda$. Similarly, the assertion of \corref{c:i! well defined} remains valid for $\bi^\lambda$ replaced by $\bj^\lambda$,
and the same is true for their Ran variants.
\end{rem}

\ssec{An aside: the ULA property}

Consider the object
$$(\bj^0)_!(\omega_{S^0_I})\in \SI^{\leq 0}_I \subset \Shv(\ol{S}{}^0_I).$$

Here $(\bj^0)_!$ is understood as the (partially defined) left adjoint of
$$(\bj^0)^!:\Shv(\ol{S}{}^0_I)\to (\bj^0)_!(\omega_{S^0_I});$$
it is always defined in constructible contexts; in the context of D-modules, it is defined
since $\omega_{S^0_I}$ is ind-holonomic.

\medskip

We will now formulate a certain strong acyclicity property of the above object that it enjoys with respect to the projection 
$$\ol{p}^0_I:\ol{S}{}^0_I\to X^I.$$

\sssec{}

Let $Y$ be a scheme, and let $\CC$ be a DG category equipped with an action of the $\Shv(Y)$, viewed as a monoidal
category with respect to $\sotimes$.

\medskip

We shall say that an object $c\in \CC$ is ULA with respect to $Y$ if for any compact $\CF\in \Shv(Y)^c$, and any $c'\in \CC$, the map
$$\Hom_\CC(\CF\sotimes c,c')\to \Hom(\BD(\CF)\sotimes \CF\sotimes c,\BD(\CF)\sotimes c')\to 
\Hom(\sfe_Y\sotimes c,\BD(\CF)\sotimes c')$$
is an isomorphism.  

\medskip

In the above formula, $\BD(-)$ denotes the Verdier duality anti-equivalence of $\Shv(Y)^c$, 
$$(\Shv(Y)^c)^{\on{op}}\to \Shv(Y)^c,$$
and $\sfe_Y$ is the ``constant sheaf"
on $Y$, i.e., $\sfe_Y:=\BD(\omega_Y)$. 
Note that when $Y$ is smooth of dimension $d$, we have $\sfe_Y\simeq \omega_Y[-2d]$. 
 
\sssec{}

We regard $\Shv(\ol{S}{}^0_I)$ as tensored over $\Shv(X^I)$ via 
$$\CF\in \Shv(X^I),\,\,\CF'\in \Shv(\ol{S}{}^0_I)\mapsto (\ol{p}{}^0_I)^!(\CF)\sotimes \CF'.$$

\medskip

We claim:

\begin{prop} \label{p:j ULA}
The object $(\bj^0)_!(\omega_{S^0_I})\in \Shv(\ol{S}{}^0_I)$ is ULA with respect to $X^I$.
\end{prop}

\begin{proof} 

For $\CF\in \Shv(X^I)$ and $\CF'\in \Shv(\ol{S}{}^0_I)$, we have a commutative square
$$
\CD
\Hom((p^0_I)^!(\CF),(\bj^0)^!(\CF'))    @>>>  \Hom((p^0_I)^!(\sfe_{X^I}),(p^0_I)^!(\BD(\CF))\sotimes (\bj^0)^!(\CF'))  \\
@V{\sim}VV      @VV{\sim}V  \\ 
\Hom((\bj^0)_!\circ (p^0_I)^!(\CF),\CF') & & 
\Hom((\bj^0)_!\circ (p^0_I)^!(\sfe_{X^I}),(\ol{p}{}^0_I)^!(\BD(\CF))\sotimes \CF')   \\
@AAA   @AAA  \\
\Hom((\ol{p}{}^0_I)^!(\CF)\sotimes (\bj^0)_!(\omega_{S^0_I}),\CF') @>>>  
\Hom((\ol{p}{}^0_I)^!(\sfe_{X^I})\sotimes (\bj^0)_!(\omega_{S^0_I}),(\ol{p}{}^0_I)^!(\BD(\CF))\sotimes \CF') .
\endCD
$$

In this diagram the lower vertical arrows are isomorphisms by \corref{c:i! well defined}(b). 
The top horizontal arrow is an isomorphism because $S^0_I$ can be exhibited as a union of closed subschemes, each being smooth 
over $X^I$. (Indeed, write $\fL(N)_I$ as a union of group sub-schemes $N^\alpha_I$ pro-smooth over $X^I$; then $S^0_I$
is a union of the quotients $N^\alpha_I/\fL^+(N)_I$.) 

\medskip

Hence, the bottom horozontal arrow is an isomorphism, as required.

\end{proof} 

\ssec{An application of Braden's theorem}  \label{ss:Braden}

In this subsection we will prove \propref{p:i* well defined}. 

\sssec{}  \label{sss:setting for Braden}

Let
$$S^{-,\lambda}_I \overset{\bj^{-,\lambda}}\hookrightarrow \ol{S}{}^{-,\lambda}_I \overset{\ol\bi{}^{-\lambda}}\to \Gr_{G,I}$$
be the objects defined in the same way as their counterparts 
$$S^{\lambda}_I \overset{\bj^{\lambda}}\hookrightarrow \ol{S}{}^{\lambda}_I \overset{\ol\bi{}^{\lambda}}\to \Gr_{G,I},$$
but where we replace $N$ by $N^-$.

\medskip

Choose a regular dominant coweight $\BG_m\to T$. It gives rise to an action of $\BG_m$ on $\ol{S}{}^0_I$
along the fibers of the projection $\ol{p}{}^0_I$. We have:

\begin{lem}  \label{l:attr locus}
The attracting (resp., repelling) locus of the above $\BG_m$ action identifies with
$$\underset{\lambda\in \Lambda^{\on{neg}}}\sqcup\, S^\lambda_I \text{ and }
\underset{\lambda\in \Lambda^{\on{neg}}}\sqcup\, S^{-,\lambda}_I,$$ respectively. The fixed point locus 
identifies with
$$\underset{\lambda\in \Lambda^{\on{neg}}}\sqcup\, s^\lambda_I((X^\lambda\times X^I)^{\subset}).$$
\end{lem} 

\sssec{}

Let us now prove point (a) of \propref{p:i* well defined}\footnote{We are grateful to Lin Chen for pointing out a mistake in the
statement of \propref{p:i* well defined} in the previous version of the paper. The corrected argument is due to him.}.

\medskip

By \propref{p:on stratum}, it suffices to show that the functor
$$\Shv((X^\lambda\times X^I)^{\subset}) \overset{(p^\lambda_I)^!}\longrightarrow \SI^{=\lambda}_I  \overset{(\bi^\lambda)_*}\longrightarrow
\SI^{\leq 0}_I,$$
admits a left adjoint. 

\medskip

For this, it suffices to show that the \emph{partially defined} left adjoint to 
$$\Shv((X^\lambda\times X^I)^{\subset}) \overset{(p^\lambda_I)^!}\longrightarrow \Shv(S^\lambda_I)  \overset{(\bi^\lambda)_*}\longrightarrow
\Shv(\ol{S}{}^0_I),$$
is defined on objects that belong to $\SI^{\leq 0}_I\subset \Shv(\ol{S}{}^0_I)$.

\medskip

It is easy to see that every object in $\Shv(\ol{S}{}^0_I)$ is $\BG_m$-monodromic. Now, the result follows from 
Braden's theorem\footnote{Braden's theorem extends from schemes to ind-schemes by an easy colimit argument.} 
(see \cite[Theorem 3.3.4]{DrGa}), combined with \lemref{l:attr locus}. 

\sssec{}

Note that Braden's theorem also implies the existence of a canonical isomorphism
\begin{equation} \label{e:Braden semiinf}
(s^\lambda_I)^!\circ (\bi^\lambda)^*\simeq (p^{-,\lambda}_I)_*\circ (\bi^{-,\lambda})^!.
\end{equation}

This implies point (b) of \propref{p:i* well defined} by base change.

\medskip

Point (c) is a formal corollary of point (b). 

\qed

\begin{rem} \label{r:*-restr Ran}
For future use, we note that \eqref{e:Braden semiinf} and \propref{p:i* well defined}(c) imply that an analogous formula 
holds over the Ran space:
\begin{equation} \label{e:Braden semiinf Ran}
(s^\lambda_\Ran)^!\circ (\bi^\lambda)^*(\CF)\simeq (p^{-,\lambda}_\Ran)_*\circ (\bi^{-,\lambda})^!(\CF),\quad \CF\in \SI_\Ran^{\leq 0}. 
\end{equation}
\end{rem} 

\section{The t-structure and the semi-infinite IC sheaf}  \label{s:t}

In this section we define a t-structure on $\SI^{\leq 0}_\Ran$, and define the main object of study in this paper--
the Ran version of the semi-infinite intersection cohomology sheaf, denoted $\ICs_\Ran$. 

\medskip

We will also give a presentation of $\ICs_\Ran$ as a colimit, and describe explicitly its *- and !-restrictions to the strata 
$S^\lambda_\Ran$. 

\ssec{The t-structure on the semi-infinite category}

\sssec{}

We introduce a t-structure on the category $\Shv((X^\lambda \times \Ran)^{\subset})$ as follows. 

%

\medskip

We declare an object $\CF\in \Shv((X^\lambda \times \Ran)^{\subset})$ to be \emph{connective} if 
$$\Hom(\CF,(\on{pr}_\Ran^\lambda)^!(\CF'))=0$$
for all $\CF'\in \Shv(X^\lambda)$ that are \emph{strictly coconnective}
(in the perverse t-structure). 

\begin{rem} \label{r:non-local 1}
The above t-structure on $\Shv((X^\lambda \times \Ran)^{\subset})$ is quite pathological
in that it is \emph{non-local}, see also Remark \ref{r:non-local}.
\end{rem} 

\sssec{}

By construction, the functor 
$$(\on{pr}_\Ran^\lambda)^!:\Shv(X^\lambda)\to \Shv((X^\lambda \times \Ran)^{\subset})$$
is left t-exact. 

\medskip

However, we claim:

\begin{prop} \label{p:pr is exact}
The functor $(\on{pr}_\Ran^\lambda)^!:\Shv(X^\lambda)\to \Shv((X^\lambda \times \Ran)^{\subset})$ is t-exact.
\end{prop}

\begin{proof} 
Follows from \corref{c:pr contr}. 
\end{proof}

\sssec{}

We define a t-structure on $\SI^{=\lambda}_\Ran$ as follows. We declare an object
$\CF\in \SI^{=\lambda}_\Ran$ 
to be connective/coconnective if
$$(s^\lambda_\Ran)^!(\CF)[\langle \lambda,2\check\rho\rangle]\in \Shv((X^\lambda \times \Ran)^{\subset})$$
is connective/coconnective.

\medskip

In other words, this t-structure is transferred from $\Shv((X^\lambda \times \Ran)^{\subset})$ via the equivalences
$$(s^\lambda_\Ran)^!:\SI^{=\lambda}_\Ran\to \Shv((X^\lambda \times \Ran)^{\subset}):(p^\lambda_\Ran)^!$$
of \propref{p:on stratum}, up to the cohomological shift $[\langle \lambda,2\check\rho\rangle]$.

\sssec{}

We define a t-structure on $\SI^{\leq 0}_\Ran$ by declaring that an object $\CF$ is 
coconnective if 
$$(\bi^\lambda)^!(\CF)\in \SI^{=\lambda}_\Ran$$
is coconnective in the above t-structure. 

\begin{rem}  \label{r:non-local}
The above t-structure on $\SI_\Ran^{\leq 0}$ is a somewhat artificial construct, since the t-structure on the 
individual strata
$$\SI^{=\lambda}_\Ran\simeq \Shv((X^\lambda \times \Ran)^{\subset})$$
was transferred from a pathological t-structure on $\Shv(X^\lambda \times \Ran)^{\subset}$,
see Remark \ref{r:non-local 1}. 

\medskip

This drawback will be cured in \secref{ss:t on unital}: we will single out a (naturally defined) full subcategory
$$\SI^{\leq 0}_{\Ran,\on{unital}}\subset \SI^{\leq 0}_\Ran,$$
such that for each stratum $S^\lambda_\Ran$, the functor 
$(\on{pr}_\Ran^\lambda\circ p^\lambda_\Ran)^!$ defines an \emph{equivalence}
$$\Shv(X^\lambda)\to \SI^{=\lambda}_{\Ran,\on{unital}}.$$
\end{rem}

\sssec{}

By construction, the subcategory of connective objects in $\SI^{\leq 0}_\Ran$ is generated under colimits by objects
of the form
\begin{equation} \label{e:gen conn}
(\bi^\lambda)_! \circ (p^\lambda_\Ran)^! (\CF)[-\langle \lambda,2\check\rho\rangle], \quad \lambda\in \Lambda^{\on{neg}}
\end{equation} 
where $\CF$ is a connective object of $\Shv((X^\lambda \times \Ran)^{\subset})$. 

\medskip

We claim:

\begin{lem}  \label{l:char of conn}
An object $\CF\in \SI^{\leq 0}_\Ran$ is connective if and only if $(\bi^\lambda)^*(\CF)\in \SI^{\leq \lambda}_\Ran$
is connective for every $\lambda\in \Lambda^{\on{neg}}$.
\end{lem}

\begin{proof}
It is clear that for objects of the form \eqref{e:gen conn}, their *-pullback to any $S^\lambda_\Ran$ is connective. Hence,
the same is true for any connective object of $\SI^{\leq 0}_\Ran$.

\medskip

Vice versa, let $0\neq \CF$ be a strictly coconnective object of $\SI^{\leq 0}_\Ran$. We need to show that if
all $(\bi^\lambda)^*(\CF)$ are connective, then $\CF=0$. Let $\lambda$ be the largest element such that
$(\bi^\lambda)^!(\CF)\neq  0$. On the one hand, since $\CF$ is strictly coconnective, $(\bi^\lambda)^!(\CF)$
is strictly coconnective. On the other hand, by the maximality of $\lambda$, we have
$$(\bi^\lambda)^!(\CF)\simeq (\bi^\lambda)^*(\CF),$$
and the assertion follows.

\end{proof}

\ssec{Definition of the semi-infinite IC sheaf}

When considering the semi-infinite IC sheaf, we will assume that $G$ is semi-simple and simply connected. 

\sssec{}

By construction, the object $(\bi^\lambda)_!(\omega_{S^\lambda_\Ran})[-\langle \lambda,2\check\rho\rangle]$
(resp., $(\bi^\lambda)_*(\omega_{S^\lambda_\Ran})[-\langle \lambda,2\check\rho\rangle]$) of $\SI^{\leq 0}_\Ran$ is connective (resp., coconnective). 

\medskip

However, in \secref{sss:loc to glob exact bis} we will prove:

\begin{prop}  \label{p:! and * perv}
The objects
$$(\bi^\lambda)_!(\omega_{S^\lambda_\Ran})[-\langle \lambda,2\check\rho\rangle] \text{ and }
(\bi^\lambda)_*(\omega_{S^\lambda_\Ran})[-\langle \lambda,2\check\rho\rangle]$$
both belong to $(\SI^{\leq 0}_\Ran)^\heartsuit$.
\end{prop}

\sssec{}

Consider the canonical map
$$(\bi^0)_!(\omega_{S^0_\Ran}) \to 
(\bi^0)_*(\omega_{S^0_\Ran}).$$

According to \propref{p:! and * perv} both sides belong to $(\SI^{\leq 0}_\Ran)^\heartsuit$. We let
$$\IC^\semiinf_\Ran\in (\SI^{\leq 0}_\Ran)^\heartsuit$$
denote the image of this map.

\medskip

The above object is the main object of study of this paper. 

\sssec{}

Our goal in this section and the next is to describe $\IC^\semiinf_\Ran$ as explicitly as possible. Specifically, we will do the following:

\begin{itemize}

\item We will describe the !- and *- restrictions of $\IC^\semiinf_\Ran$ to the strata $S^\lambda_\Ran$
(see \thmref{t:descr of restr}); 

\item We will exhibit the values of $\IC^\semiinf_\Ran$ in $\Shv(\Gr_{G,I})$ explicitly as colimits
(see \thmref{t:descr as colimit}); 

\item We will relate $\IC^\semiinf_\Ran$ to the intersection cohomology sheaf of Drinfeld's compactification $\BunNb$
(see \thmref{t:IC loc glob}). 

\end{itemize}

\ssec{Digression: from commutative algebras to factorization algebras}  \label{ss:factor alg}

Let $A$ be a commutative algebra, graded by $\Lambda^{\on{neg}}$ with $A(0)\simeq \sfe$. To $A$
we can attach an object $$\on{Fact}^{\on{alg}}(A)_{X^\lambda}\in \Shv(X^\lambda),$$
characterized by the property that its !-fiber at a divisor
$$\underset{k}\Sigma\, \lambda_k\cdot x_k\in X^\lambda, \quad 0\neq \lambda_k\in \Lambda^{\on{neg}}, \quad \underset{k}\Sigma\, \lambda_k=\lambda,
\quad k'\neq k''\,\Rightarrow\, x_{k'}\neq x_{k''}$$
equals $\underset{k}\bigotimes\, A(\lambda_k)$. 

\medskip

In the present subsection we recall this construction. 

\sssec{}  \label{sss:tw arrows lambda}

Consider the category $\on{TwArr}_\lambda$ whose objects are diagrams
\begin{equation} \label{e:obj tw arrow}
\Lambda^{\on{neg}}-0 \overset{\ul\lambda}\longleftarrow J \overset{\phi}\twoheadrightarrow K, \quad \underset{j\in J}\Sigma\, \ul\lambda(j)=\lambda,
\end{equation}
where $I$ and $J$ are finite non-empty sets. A morphism between two such objects 
is a diagram
$$
\CD
\Lambda^{\on{neg}}-0 @<{\ul\lambda_1}<< J_1 @>{\phi_1}>>  K_1 \\
@V{\on{id}}VV  @V{\psi_J}VV  @AA{\psi_K}A \\
\Lambda^{\on{neg}}-0 @<{\ul\lambda_2}<< J _2 @>{\phi_2}>>  K_2,
\endCD
$$
where:

\begin{itemize}

\item The right square commutes; 

\item The maps $\psi_J$ and $\psi_K$ are surjective;

\item $\ul\lambda_2(j_2)=\underset{j_1\in \psi_J^{-1}(j_2)}\Sigma\, \ul\lambda_1(j_1)$. 

\end{itemize}

\sssec{}

The algebra $A$ defines a functor 
$$\on{TwArr}(A):\on{TwArr}_\lambda\to \Shv(X^\lambda),$$ constructed as follows. 

\medskip

For an object \eqref{e:obj tw arrow}, let
$\Delta_{K,\lambda}$ denote the map $X^K\to X^\lambda$ that sends a point
$\{x_k,k\in K\} \in X^K$ to the divisor
$$\underset{k\in K}\Sigma\, (\underset{j\in \phi^{-1}(k)}\Sigma\, \ul\lambda(j))\cdot x_k\in X^\lambda.$$

\medskip

Then the value of $\on{TwArr}(A)$ on \eqref{e:obj tw arrow} is
$$(\Delta_{K,\lambda})_*(\omega_{X^K})\bigotimes \left(\underset{j\in J}\otimes\, A(\lambda_j)\right),$$
where $\lambda_j=\ul\lambda(j)$. 

\medskip

The structure of functor on $\on{TwArr}(A)$ is provided by the commutative algebra structure on $A$. 

\sssec{}

We set 
$$\on{Fact}^{\on{alg}}(A)_{X^\lambda}:=\underset{\on{TwArr}_\lambda}{\on{colim}}\, \on{TwArr}(A)\in \Shv(X^\lambda).$$

\sssec{An example} \label{sss:ex fact}

Let $V$ be a $\Lambda^{\on{neg}}$-graded vector space with $V(0)=0$. Let us take $A=\Sym(V)$. In this
case $\on{Fact}^{\on{alg}}(A)_{X^\lambda}$, can be explicitly described as follows:

\medskip

It is the direct sum over all ways to write $\lambda$ as a sum 
$$\lambda=\underset{k}\Sigma\, n_k\cdot \lambda_k, \quad n_k>0,\,\,\lambda_k\in \Lambda^{\on{neg}}-0$$
of the direct images of
$$\underset{k}\bigotimes\,  (\omega_X\otimes V(\lambda_k))^{(n_k)}$$
along the maps
$$\underset{k}\Pi\, X^{(n_k)}\to X^\lambda,$$
where $(-)^{(n)}$ denotes the $n$-th symmetric power of a given local system. 

\sssec{} \label{sss:coalg}

Dually, if $A$ is a co-commutative co-algebra, graded by $\Lambda^{\on{neg}}$ with $A(0)\simeq \sfe$, then to $A$ we
attach an object $\on{Fact}^{\on{coalg}}(A)_{X^\lambda}\in \Shv(X^\lambda)$
characterized by the property that its *-fiber at a divisor
$$\underset{k}\Sigma\, \lambda_k\cdot x_k\in X^\lambda, \quad 0\neq \lambda_k\in \Lambda^{\on{neg}}, \quad \underset{k}\Sigma\, \lambda_k=\lambda,
\quad k'\neq k''\,\Rightarrow\, x_{k'}\neq x_{k''}$$
equals $\underset{k}\bigotimes\, A(\lambda_k)$. 

\medskip

If all the graded components of $A$ are finite-dimensional, we can view the dual $A^\vee$ of $A$ as a $\Lambda^{\on{neg}}$-graded
algebra, and we have
\begin{equation} \label{e:alg coalg Verdier}
\BD(\on{Fact}^{\on{coalg}}(A)_{X^\lambda}) \simeq \on{Fact}^{\on{alg}}(A^\vee)_{X^\lambda},
\end{equation} 
where we remind that $\BD$ stands for the Verdier duality functor. 

\ssec{Restriction of $\IC^\semiinf_\Ran$ to strata}

\sssec{}

We apply the construction of \secref{ss:factor alg} to $A$ being the (classical) algebra $\CO(\cN)$ (resp., co-algebra 
$U(\cn^-)$). 

\medskip

Thus, we obtain the objects
$$\on{Fact}^{\on{alg}}(\CO(\cN))_{X^\lambda} \text{ and } \on{Fact}^{\on{coalg}}(U(\cn^-))_{X^\lambda} $$ in $\Shv(X^\lambda)$. 

\medskip

Note also that $U(\cn^-)$ is the graded dual of $\CO(\cN)$, and so the objects $\on{Fact}^{\on{alg}}(\CO(\cN))_{X^\lambda}$ and
$\on{Fact}^{\on{coalg}}(U(\cn^-))_{X^\lambda}$ are Verdier dual to each other, see \eqref{e:alg coalg Verdier}. 

\sssec{}

From the construction it follows that for $\lambda\neq 0$, 
$$\on{Fact}^{\on{alg}}(\CO(\cN))_{X^\lambda}\in \Shv(X^\lambda)^{<0},$$
and hence 
$$\on{Fact}^{\on{coalg}}(U(\cn^-))_{X^\lambda} \in \Shv(X^\lambda)^{>0}.$$

\begin{rem}
Note that by the PBW theorem, when viewed
as a co-commuatative co-algebra, $U(\cn^-)$ is canonically identified with $\Sym(\cn^-)$; in this paper we will not use the
algebra structure on $U(\cn^-)$, which allows to distinguish it from $\Sym(\cn^-)$. 

\medskip

Dually, when viewed just as a commutative 
algebra (i.e., ignoring the Hopf algebra structure), $\CO(\cN)$ is canonically identified with $\Sym(\cn^*)$.  
So $\on{Fact}^{\on{alg}}(\CO(\cN))_{X^\lambda}$ falls into the paradigm of Example \ref{sss:ex fact}. 
\end{rem} 

\sssec{}

In \secref{ss:proof calc IC} we will prove:

\begin{thm}  \label{t:descr of restr}
The objects
$$(\bi^\lambda)^!(\IC^\semiinf_\Ran)  \text{ and } (\bi^\lambda)^*(\IC^\semiinf_\Ran)$$
of $\Shv(S^\lambda_\Ran)$ identify with the !-pullback along
$$S^\lambda_\Ran \overset{p^\lambda_\Ran} \longrightarrow 
(X^\lambda \times \Ran)^{\subset}\overset{\on{pr}_\Ran^\lambda}\longrightarrow X^\lambda$$
of $\on{Fact}^{\on{coalg}}(U(\cn^-))_{X^\lambda} [-\langle \lambda,2\check\rho\rangle]$ and 
$\on{Fact}^{\on{alg}}(\CO(\cN))_{X^\lambda}[-\langle \lambda,2\check\rho\rangle]$, respectively,
\end{thm} 

\ssec{Digression: categories over the Ran space}  \label{ss:spread over Ran}

We will now discuss a variant of the construction in \secref{ss:factor alg} that attaches to a symmetric monoidal
category $\CA$ a category spread over the Ran space, denoted $\on{Fact}^{\on{alg}}(\CA)_{\Ran}$. 

\sssec{}

Consider the category $\on{TwArr}$ whose objects are 
\begin{equation}  \label{e:object usual tw arrow}
I \overset{\phi}\twoheadrightarrow J,
\end{equation} 
where $I$ and $J$ are finite non-empty sets. A morphism between two such objects is 
is a commutative diagram
\begin{equation}  \label{e:morphism usual tw arrow}
\CD
J_1 @>{\phi_1}>>  K_1 \\
@V{\psi_J}VV @AA{\psi_K}A \\
J_2 @>{\phi_2}>>  K_2,
\endCD
\end{equation} 
where the maps $\psi_J$ and $\psi_K$ are surjective. 

\sssec{}

To $\CA$ we attach a functor
$$\on{TwArr}(\CA):\on{TwArr}\to \DGCat$$
by sending an object \eqref{e:object usual tw arrow} to $\Shv(X^K)\otimes \CA^{\otimes J}$,
and a morphism \eqref{e:morphism usual tw arrow} to a functor
comprised of
$$\Shv(X^{K_1})\overset{(\Delta_{\psi_J})_*}\longrightarrow \Shv(X^{K_2})$$
and the functor 
$$\CA^{\otimes J_1}\to \CA^{\otimes J_2},$$
given by the symmetric monoidal structure on $\CA$. 

\sssec{}

We set
\begin{equation} \label{e:categ over Ran}
\on{Fact}^{\on{alg}}(\CA)_{\Ran}:=\underset{\on{TwArr}}{\on{colim}}\, \on{TwArr}(\CA)\in \DGCat.
\end{equation}

\sssec{}  \label{sss:ex fact categ}

Let us consider some examples. 

\medskip

\noindent(i) Let $\CA=\Vect$. Then $\on{Fact}^{\on{alg}}(\CA)_{\Ran}\simeq \Shv(\Ran)$. 

\medskip

\noindent(ii) Let $\CA$ be the (non-unital) symmetric monoidal category consisting of vector spaces
graded by the semi-group $\Lambda^{\on{neg}}-0$.  We have a canonical functor
\begin{equation} \label{e:fact graded}
\on{Fact}^{\on{alg}}(\CA)_{\Ran}\to \Shv(\underset{\lambda\in \Lambda^{\on{neg}}-0}\sqcup\, X^\lambda),
\end{equation}
and it follows from \cite[Lemma 7.4.11(d)]{Ga2} that this functor is an equivalence. 

\sssec{} \label{sss:comon over Ran}

Similarly, if $\CA$ is a symmetric co-monoidal category, we can form the limit of the corresponding functor
$$\on{TwArr}(\CA):\on{TwArr}^{\on{op}}\to \DGCat,$$
and obtain a category that we denote by $\on{Fact}^{\on{coalg}}(\CA)_{\Ran}$. 

\sssec{}  \label{sss:limit and colimit}

Recall that whenever we have a diagram of categories  
$$t\mapsto \CC_t$$
indexed by some category $T$, then 
$$\underset{t\in T}{\on{colim}}\, \CC_t$$
is canonically equivalent to
$$\underset{t\in T^{\on{op}}}{\on{lim}}\, \CC_t,$$
where the transition functors are given by the right adjoints of those in the original family. 

\sssec{}  \label{sss:comult cont}

Let $\CA$ be again a symmetric monoidal category. Applying the observation of \secref{sss:limit and colimit} to  
the colimit \eqref{e:categ over Ran}, we obtain that $\on{Fact}^{\on{alg}}(\CA)_{\Ran}$ can also be written as a limit. 

\medskip

Assume now that $\CA$ is such that the functor
$$\CA\to \CA\otimes \CA,$$
right adjoint to the tensor product operation, is continuous. In this case, the above tensor co-product operation 
makes $\CA$ into a symmetric co-monoidal category, and we can form $\on{Fact}^{\on{coalg}}(\CA)_{\Ran}$.

\medskip

We note however, that the \emph{limit} presentation of $\on{Fact}^{\on{alg}}(\CA)_{\Ran}$ tautologically coincides with the
limit defining $\on{Fact}^{\on{coalg}}(\CA)_{\Ran}$. I.e., we have a canonical equivalence: 
$$\on{Fact}^{\on{alg}}(\CA)_{\Ran}\simeq \on{Fact}^{\on{coalg}}(\CA)_{\Ran}.$$

\medskip

Hence, in what follows we will sometimes write simply $\on{Fact}(\CA)_{\Ran}$, having both of the above realizations in mind.  

\sssec{}

Let $I$ be a fixed finite non-empty set. The above constructions have a variant, where instead of $\on{TwArr}$
we use its variant, denoted $\on{TwArr}_{I/}$, whose objects are commutative diagrams
$$I\twoheadrightarrow J \overset{\phi}\twoheadrightarrow K,$$
and whose morphisms are commutative diagrams
$$
\CD
I @>>> J_1 @>{\phi_1}>>  K_1 \\
@V{\on{id}}VV  @V{\psi_J}VV @AA{\psi_K}A \\
I @>>> J_2 @>{\phi_2}>>  K_2,
\endCD
$$

We denote the resulting category by $\on{Fact}^{\on{alg}}(\CA)_I$, i.e.,
$$\on{Fact}^{\on{alg}}(\CA)_I:=\underset{\on{TwArr}_{I/}}{\on{colim}}\, \on{TwArr}(\CA)|_{\on{TwArr}_{I/}}.$$

\sssec{}

For a surjective morphism $\phi:I_1\twoheadrightarrow I_2$, we have the corresponding functor
$$\on{TwArr}_{I_2/}\to \on{TwArr}_{I_1/},$$
which induces a functor
\begin{equation} \label{e:Delta phi *}
(\Delta_\phi)_*:\on{Fact}^{\on{alg}}(\CA)_{I_2}\to \on{Fact}^{\on{alg}}(\CA)_{I_1}.
\end{equation}

\medskip

An easy cofinality argument shows that the resulting functor
\begin{equation} \label{e:I Ran}
\underset{I}{\on{colim}}\, \on{Fact}^{\on{alg}}(\CA)_I\to \on{Fact}^{\on{alg}}(\CA)_\Ran
\end{equation} 
is an equivalence. 

\sssec{}  \label{sss:Ran restr}

Note also that push-out defines a functor
$$\on{TwArr}_{I_1/}\to \on{TwArr}_{I_2/},$$
and we have a natural transformation from the composite
$$\on{TwArr}_{I_1/} \to \on{TwArr} \overset{\on{TwArr}(\CA)}\longrightarrow \on{DGCat}$$
to the composite
$$\on{TwArr}_{I_1/}\to \on{TwArr}_{I_2/} \to \on{TwArr} \overset{\on{TwArr}(\CA)}\longrightarrow \on{DGCat},$$
inducing a functor
$$\on{Fact}^{\on{alg}}(\CA)_{I_1}:=\on{Fact}^{\on{alg}}(\CA)_{I_1} \to \on{Fact}^{\on{alg}}(\CA)_{I_2}=:\on{Fact}^{\on{alg}}(\CA)_{I_2},$$
to be denoted $(\Delta_\phi)^!$.

\medskip

By unwinding the constructions, it follows that the above functor $(\Delta_\phi)^!$ is the right adjoint of the functor
$(\Delta_\phi)_*$ of \eqref{e:Delta phi *}. 

\medskip

In particular, by \secref{sss:limit and colimit}, we can also write
\begin{equation} \label{e:I Ran lim}
\on{Fact}^{\on{alg}}(\CA)_\Ran \simeq \underset{I}{\on{lim}}\, \on{Fact}^{\on{alg}}(\CA)_I,
\end{equation}
where the limit is taken with respect to the functors $(\Delta_\phi)^!$.

%
%
%
%
%
%

\ssec{Presentation of $\ICs$ as a colimit} \label{ss:pres as colimit}

Consider the symmetric monoial category $\Rep(\cG)$.

\sssec{}  \label{sss:V lambda I}

For a fixed finite non-empty set $I$ and
a map $\ul\lambda:I\to \Lambda^+$, we consider the following object of $\on{Fact}(\Rep(\cG))_I$, denoted $V^{\ul\lambda}$. 

\medskip

Informally, $V^{\ul\lambda}$ is designed so its !-fiber at a point 
$$I\to X, \quad I=\underset{k}\sqcup\, I_k,\,\, I_k\mapsto x_k,\,\,\quad x_{k'}\neq x_{k''}$$
is
$$\underset{k}\bigotimes\, V^{\lambda_k}\in \Rep(\cG)^{\otimes k},\quad \lambda_k=\underset{i\in I_k}\Sigma\, \ul\lambda(i),$$
where for $\lambda\in \Lambda^+$, we denote by $V^\lambda$ the corresponding irreducible highest weight 
representation of $\cG$, normalized so that its highest weight line is identified with $\sfe$.

\sssec{}

It terms of the presentation of $\on{Fact}(\Rep(\cG))_I$ as a colimit 
$$\on{Fact}(\Rep(\cG))_I=\underset{\on{TwArr}_{I/}}{\on{colim}}\, \on{TwArr}(\Rep(\cG)),$$
the object $V^{\ul\lambda}$ corresponds to the colimit over $\on{TwArr}_{I/}$ of the functor 
$$\on{TwArr}_{I/} \to \on{Fact}(\Rep(\cG))_I$$
that sends 
\begin{equation} \label{e:highest weight}
(I\twoheadrightarrow J \twoheadrightarrow K)\in \on{TwArr}_{I/} \, \rightsquigarrow 
V^{\ul\lambda}_{I\twoheadrightarrow J \twoheadrightarrow K}\in \Shv(X^K)\otimes \Rep(\cG)^{\otimes J}\to \on{Fact}(\Rep(\cG))_I,
\end{equation} 
where 
$$V^{\ul\lambda}_{I\twoheadrightarrow J \twoheadrightarrow K}=\omega_{X^K}\bigotimes \left(\underset{j\in J}\otimes\, V^{\lambda_j}\right), \quad 
\lambda_j=\underset{i\in I,i\mapsto j}\Sigma\, \ul\lambda(i).$$

The structure of a functor $\on{TwArr}_{I/}\to \on{Fact}(\Rep(\cG))_I$ on \eqref{e:highest weight}
is given by the Pl\"ucker maps
$$\underset{i}\otimes \, V^{\lambda_i}\to V^\lambda, \quad \lambda=\underset{i}\Sigma\, \lambda_i.$$

\sssec{}

Denote
$$\Sph_{G,I}:=\Shv(\fL^+(G)_I\backslash \Gr_{G,I}) \text{ and } \Sph_{G,\Ran}:=\Shv(\fL^+(G)_\Ran\backslash \Gr_{G,\Ran}).$$

\medskip

Consider the symmetric monoial category $\Rep(\cG)$. Geometric Satake defines functors
$$\Sat_{G,I}:\on{Fact}(\Rep(\cG))_I\to \Sph_{G,I}$$
that glue to a functor
$$\Sat_{G,\Ran}:\on{Fact}(\Rep(\cG))_{\Ran}\to \Sph_{G,\Ran}.$$

\sssec{}

Consider the object 
$$\Sat_{G,I}(V^{\ul\lambda}) \in \Sph_{G,I}.$$

\medskip

The element $\ul\lambda$ gives rise to a section
$$s^{-,\ul\lambda}_I:X^I \to \Gr_{T,I}\subset \Gr_{G,I}.$$

Denote 
$$\delta_{-\ul\lambda}:=(s^{-,\ul\lambda}_I)_!(\omega_{X^I})\in \Shv(\Gr_{G,I}).$$

Consider the object 
$$\delta_{-\ul\lambda}\star \Sat_{G,I}(V^{\ul\lambda})[\langle \lambda,2\check\rho\rangle]\in \Shv(\Gr_{G,I}).$$
In the above formula, $\lambda=\underset{i\in I}\Sigma\, \ul\lambda(i)$, and $\star$ denotes the (right) convolution
action of $\Sph_{G,I}$ on $\Shv(\Gr_{G,I})$.

\sssec{}  \label{sss:semiinf I}

Consider now the set $\Maps(I,\Lambda^+)$ of maps
$$\ul\lambda:I\to \Lambda^+.$$

We equip it with a partial order by declaring
$$\ul\lambda_1\leq \ul\lambda_2\, \Leftrightarrow\, \ul\lambda_2(i)-\ul\lambda_1(i)\in \Lambda^+,\, \forall i\in I.$$

The assignment
$$\ul\lambda\mapsto \delta_{-\ul\lambda}\star \Sat_{G,I}(V^{\ul\lambda})[\langle \lambda,2\check\rho\rangle]\in \Shv(\Gr_{G,I})$$
has a structure of a functor 
$$\Maps(I,\Lambda^+)\to \Shv(\Gr_{G,I}),$$
see Sects. \ref{sss:extend to functor} and \ref{sss:DrPl on IC}.

\medskip

Set
$$\IC^\semiinf_I:=\underset{\ul\lambda\in \Maps(I,\Lambda^+)}{\on{colim}}\, 
\delta_{-\ul\lambda}\star \Sat_{G,I}(V^{\ul\lambda})[\langle \lambda,2\check\rho\rangle]\in \Shv(\Gr_{G,I}).$$

\medskip

As in \cite[Proposition 2.3.7(a,b,c)]{Ga1} one shows:

\begin{lem} \label{l:elem ppties IC}
The object $\IC^\semiinf_I$ has the following properties:

\smallskip

\noindent{\em(a)} It is supported on $\ol{S}{}^0_I$;

\smallskip

\noindent{\em(b)} It belongs to $\SI^{\leq 0}_I=\Shv(\ol{S}{}^0_I)^{\fL(N)_I}\subset \Shv(\ol{S}{}^0_I)$;

\smallskip

\noindent{\em(c)} Its restriction to $S^0_I$ is identified with $\omega_{S^0_I}$. 

\end{lem}

\sssec{}  \label{sss:semiinf prime}

For a surjective map
$$\phi:I_2\twoheadrightarrow I_1$$
and the corresponding map
$$\Delta_\phi: \Gr_{G,I_1}\to \Gr_{G,I_2},$$
we have a canonical identification
$$(\Delta_\phi)^!(\IC^\semiinf_{I_2})\simeq \IC^\semiinf_{I_1}.$$

One endows this system of isomorphisms with a homotopy-coherent system of compatibilities,
thus making the assignment
$$I\mapsto \IC^\semiinf_I$$
into an object of $\SI^{\leq 0}_\Ran$, see \secref{sss:V lambda Ran}.

\medskip

We denote this object by $'\!\IC^\semiinf_\Ran$. By \lemref{l:elem ppties IC}(c), we have a canonical identification
\begin{equation} \label{e:prime open}
'\!\IC^\semiinf_\Ran|_{S^0_\Ran}\simeq \omega_{S^0_\Ran},
\end{equation}

\sssec{}  \label{sss:at a point}

Fix a point $x\in X$, and consider the restriction of $'\!\IC^\semiinf_\Ran$ along the map
$$\Gr_{G,x}\simeq \{x\}\underset{\Ran}\times \Gr_{G,\Ran}\to \Gr_{G,\Ran}.$$

It follows from the construction, that this restriction identifies canonically with the object
$$\IC^\semiinf_x\in \Shv(\Gr_{G,x}),$$
constructed in \cite[Sect. 2.3]{Ga1}.

\ssec{Presentation of $\ICs_\Ran$ as a colimit}

\sssec{}

The rest of this section will be devoted to the proof of the following result:

\begin{thm}  \label{t:descr as colimit}
There exists a unique isomorphism $'\!\IC^\semiinf_\Ran\simeq \IC^\semiinf_\Ran$,
extending the identification 
$$'\!\IC^\semiinf_\Ran|_{S^0_\Ran}\overset{\text{\eqref{e:prime open}}}\simeq \omega_{S^0_\Ran} \simeq
\IC^\semiinf_\Ran|_{S^0_\Ran}.$$
\end{thm} 

The proof of \thmref{t:descr as colimit} will amount to the combination of
the following two assertions: 

\begin{prop} \label{p:* retsr IC'}
For $\mu\in \Lambda^{\on{neg}}$, the object
$$(\bi^\mu)^*({}'\!\IC^\semiinf_\Ran)\in \SI^{=\mu}_\Ran$$
identifies canonically with the !-pullback along
$$S^\mu_\Ran \overset{p^\mu_\Ran} \longrightarrow (X^\mu \times \Ran)^{\subset}\overset{\on{pr}_\Ran^\mu}\longrightarrow X^\mu$$
of $\on{Fact}^{\on{alg}}(\CO(\cN))_{X^\mu}[-\langle \mu,2\check\rho\rangle]$.
\end{prop}

\begin{prop} \label{p:! retsr IC'}
For $0\neq \mu\in \Lambda^{\on{neg}}$, the object
$$(\bi^\mu)^!({}'\!\IC^\semiinf_\Ran)[\langle \mu,2\check\rho\rangle]\in \SI^{=\mu}_\Ran$$
is a pullback along $\on{pr}_\Ran^\mu\circ p^\mu_\Ran$ of an object of $\Shv(X^\mu)$ that is strictly coconnective.
\end{prop}

\begin{proof}[Proof of \thmref{t:descr as colimit} modulo the propositions]

By the definition of the t-structure on $\SI^{\leq 0}_\Ran$ and \lemref{l:char of conn}, it suffices to show that for $\mu\in \Lambda^{\on{neg}}-0$,
the !-restriction (resp., *-restriction) of $'\!\IC^\semiinf_\Ran$ to $S^\mu_\Ran$ is cohomologically $>0$ (resp., $<0$). 

\medskip

Now, \propref{p:! retsr IC'} (resp., \propref{p:* retsr IC'}) implies the required cohomological estimate by \propref{p:pr is exact}.

\end{proof}

\begin{rem} \label{r:*-fibers proved}
Note that \thmref{t:descr as colimit} and \propref{p:* retsr IC'} imply the assertion of \thmref{t:descr of restr} about the *-fibers. 

\medskip

We will use this observation in the sequel, for the proof of the assertion of \thmref{t:descr of restr} about the !-fibers. 
\end{rem}

\sssec{} 

Let us assume \thmref{t:descr as colimit} for a moment. As a corollary, and taking into account \secref{sss:at a point}, we obtain:

\begin{cor}
The restriction of $\IC^\semiinf_\Ran$ along the map
$$\Gr_{G,x}\simeq \{x\}\underset{\Ran}\times \Gr_{G,\Ran}\to \Gr_{G,\Ran}$$
identifies canonically with the object 
$\IC^\semiinf_x\in \Shv(\Gr_{G,x})$
of \cite[Sect. 2.3]{Ga1}.
\end{cor}

\sssec{}  \label{sss:factor prime}

Before we proceed with the proof of Propositions \ref{p:* retsr IC'} and \ref{p:! retsr IC'}, let us make the following
observation concerning the object $'\!\IC^\semiinf_\Ran$ (it will be used in the proof of \ref{p:! retsr IC'}): 

\medskip

By construction, $$'\!\IC^\semiinf_\Ran\in \Shv(\ol{S}{}^0_\Ran)$$ has the following \emph{factorization property} with respect to $\Ran$:

\medskip

Let $(\Ran\times \Ran)_{\on{disj}}$ denote the \emph{disjoint locus}. I.e., for an affine test-scheme $Y$, 
$$\Hom(Y,(\Ran\times \Ran)_{\on{disj}})\subset \Hom(Y,\Ran)\times \Hom(Y,\Ran)$$
consists of those pairs $\CI_1,\CI_2\in \Hom(Y,X)$, for which for every $i_1\in I_1$ and $i_2\in I_2$, the corresponding two maps
$Y\rightrightarrows X$ have non-intersecting images. 

\medskip

It is well-known that we have a canonical isimorphism
\begin{equation} \label{e:factor Gr}
(\Gr_{G,\Ran}\times \Gr_{G,\Ran})\underset{\Ran\times \Ran}\times (\Ran\times \Ran)_{\on{disj}}
\simeq \Gr_{G,\Ran}\underset{\Ran}\times (\Ran\times \Ran)_{\on{disj}},
\end{equation}
where 
$$(\Ran\times \Ran)_{\on{disj}}\to \Ran\times \Ran\to \Ran$$ is the map
$$\CI_1,\CI_2\mapsto \CI_1\cup \CI_2.$$

Then, in terms of the identification \eqref{e:factor Gr}, we have a canonical isomorphism
\begin{multline} \label{e:factor 'IC}
({}'\!\IC^\semiinf_\Ran\boxtimes {}'\!\IC^\semiinf_\Ran)
|_{(\Gr_{G,\Ran}\times \Gr_{G,\Ran})\underset{\Ran\times \Ran}\times (\Ran\times \Ran)_{\on{disj}}}\simeq \\
\simeq {}'\!\IC^\semiinf_\Ran|_{} \Gr_{G,\Ran}\underset{\Ran}\times (\Ran\times \Ran)_{\on{disj}}.
\end{multline} 

\ssec{Description of the *-restriction to strata}

The goal of this subsection is to prove \propref{p:* retsr IC'}.

\sssec{}

We will compute 
$$(\bi^\mu)^*({}'\!\IC^\semiinf_I)\in \SI^{\leq 0}_I$$
for each individual finite non-empty set $I$, and obtain the !-pullback of 
$$(\on{pr}^\mu_\Ran\circ p^\mu_\Ran)^!(\on{Fact}^{\on{alg}}(\CO(\cN))_{X^\mu})[-\langle \mu,2\check\rho\rangle]$$
along $S^\mu_I\to S^\mu_\Ran$. 

\medskip

Thus, by \propref{p:on stratum}, we need to construct an identification
\begin{equation} \label{e:to compute mu * 1}
(p^\mu_I)_!\circ (\bi^\mu)^*({}'\!\IC^\semiinf_I)[\langle \mu,2\check\rho\rangle]
\simeq (\on{pr}^\mu_I)^!(\on{Fact}^{\on{alg}}(\CO(\cN))_{X^\mu}),
\end{equation} 
where $\on{pr}^\mu_I$ denotes the map 
$$(X^\mu\times X^I)^{\subset}\to X^\mu.$$

\sssec{}

We will compute
\begin{equation} \label{e:to compute mu * 2}
(p^\mu_I)_!\circ (\bi^\mu)^*(\delta_{-\ul\lambda}\star \Sat_{G,I}(V^{\ul\lambda}))[\langle \lambda+\mu,2\check\rho\rangle]\in \Shv((X^\mu\times X^I)^{\subset})
\end{equation} 
for each individual $\ul\lambda:I\to \Lambda^+$ with $\lambda=\underset{i\in I}\Sigma\, \ul\lambda(i)$. 

\medskip

Namely, we will show that \eqref{e:to compute mu * 2} identifies with the following object, denoted
$$V^{\ul\lambda}(\ul\lambda+\mu)\in  \Shv((X^\mu\times X^I)^{\subset}),$$
described below. 

\sssec{}

Before we give the definition of $V^{\ul\lambda}(\ul\lambda+\mu)$, let us describe what its !-fibers are.
Fix a point of $(X^\mu\times X^I)^{\subset}$. By definition, a datum of such a point consists of:

\begin{itemize}

\item A partition $I=\underset{k}\sqcup\, I_k$; 

\item A collection of distinct points $x_k$ of $x$;

\item An assignment $x_k\mapsto \mu_k\in \Lambda^{\on{neg}}-0$, so that $\underset{k}\Sigma\, \mu_k=\mu$.

\end{itemize} 

Then the !-fiber of $V^{\ul\lambda}(\ul\lambda+\mu)$ at a such a point is
$$\underset{k}\otimes\, V^{\lambda_k}(\lambda_k+\mu_k),$$
where $\lambda_k=\underset{i\in I_k}\Sigma\, \ul\lambda(i)$, and where 
$V(\nu)$ denotes the $\nu$-weight space in a $\cG$-representation $V$. 

\sssec{}

Consider the category, denoted $\on{TwArr}_{\mu,I/}$, whose objects are commutative diagrams
$$
\CD
I @>{\upsilon}>> J @>{\psi}>> K \\
& & @V{\phi_J}VV @VV{\phi_K}V  \\
& & \wt{J} @>{\wt\psi}>> \wt{K}  \\
& & @A{\phi'_J}AA @AA{\phi'_K}A \\
\Lambda^{\on{neg}}  @<{\ul\mu}<< \wt{J}'   @>{\wt\psi'}>> \wt{K}',
\endCD
$$
where the maps $\upsilon,\psi,\wt\psi,\wt\psi',\phi_J,\phi_K$ are surjective (but $\phi'_J$ and $\phi'_K$ are not necessarily so),
and 
$$\underset{\wt{j}'\in  \wt{J}'}\Sigma\, \ul\mu(\wt{j}')=\mu.$$

Morphisms in this category are defined by the same principle as in $\on{TwArr}_{\mu}$ and $\on{TwArr}_{I/}$ 
introduced earlier, i.e., the sets $J$, $\wt{J}$, $\wt{J}'$ map forward and the sets $K$, $\wt{K}$, $\wt{K}'$ map
backwards. 

\medskip

Let $\Delta_{\wt{K},I,\lambda}$ denote the map
$$X^{\wt{K}}\to X^\mu\times X^I,$$
comprised of  
$$\Delta_{\phi_K\circ \psi\circ \upsilon}:X^{\wt{K}}\to X^I$$
and 
$$X^{\wt{K}} \overset{\Delta_{\phi'_K}}\longrightarrow X^{\wt{K}'} \overset{\Delta_{\wt{K}',\ul\mu}}\longrightarrow X^\mu.$$

\medskip

We let $V^{\ul\lambda}(\ul\lambda+\mu)$ be the colimit over  $\on{TwArr}_{\mu,I/}$ of the objects
$$(\Delta_{\wt{K},I,\lambda})_*(\omega_{X^{\wt{K}}})\bigotimes \left(\underset{\wt{j}\in \wt{J}}\otimes \, V^{\lambda_{\wt{j}}}(\lambda_{\wt{j}}+\mu_{\wt{j}})\right),$$
where 
$$\lambda_{\wt{j}}=\underset{i\in I,i\mapsto \wt{j}}\Sigma\, \ul\lambda(i) \text{ and }
\mu_{\wt{j}}=\underset{\wt{j}'\in \wt{J}',\wt{j}'\mapsto \wt{j}}\Sigma\, \ul\mu(\wt{j}').$$

\sssec{}

Applying Braden's theorem (see \secref{sss:setting for Braden}), we obtain a canonical isomorphism
$$(p^\mu_I)_!\circ (\bi^\mu)^*(\delta_{-\ul\lambda}\star \Sat_{G,I}(V^{\ul\lambda}))
\simeq (p^{-,\mu}_I)_*\circ (\bi^{-,\mu})^!(\delta_{-\ul\lambda}\star \Sat_{G,I}(V^{\ul\lambda})).$$

\medskip

The key property of the geometric Satake functor $\Sat_{G,I}$ for $I=\{1\}$ is that for $V'\in \Rep(\cG)$ and
$\mu'\in \Lambda$
$$(p_{\{1\}}^{-,\mu'})_*\circ (\bi^{-,\mu'})^!(\Sat_{G,\{1\}}(V'))[\langle \mu,2\check\rho\rangle]\simeq
\omega_X\otimes V'(\mu').$$

Unwinding the construction of its multi-point version $\Sat_{G,I}$, we obtain a canonical isomorphism
$$(p^{-,\mu}_I)_*\circ (\bi^{-,\mu})^!(\delta_{-\ul\lambda}\star \Sat_{G,I}(V^{\ul\lambda}))
[\langle \lambda+\mu,2\check\rho\rangle] \simeq V^{\ul\lambda}(\ul\lambda+\mu),$$
giving rise to the desired expression for \eqref{e:to compute mu * 2}. 

\sssec{}

To finish the proof of \propref{p:* retsr IC'}, we have to show that 
$$\underset{\ul\lambda\in \Maps(I,\Lambda^+)}{\on{colim}}\, V^{\ul\lambda}(\ul\lambda+\mu)$$
identifies canonically with $(\on{pr}^\mu_I)^!(\on{Fact}^{\on{alg}}(\CO(\cN))_{X^\lambda})$. 

\medskip

Indeed, this follows from the fact that we have a canonical identification
$$\underset{\lambda\in \Lambda^+}{\on{colim}}\, V^\lambda(\lambda+\mu)\simeq \CO(\cN)(\mu),$$
where $\Lambda^+$ is endowed with the order relation
$$\lambda_1\leq\lambda_2\, \Leftrightarrow\, \lambda_2-\lambda_1\in \Lambda^+.$$

\ssec{Proof of coconnectivity}

In this subsection we will prove \propref{p:! retsr IC'}, thereby completing the proof
of \thmref{t:descr as colimit}. 

\sssec{}

Consider the diagonal stratification of $X^\mu$. For each parameter $\beta$ of the stratification, let 
$X^\mu_\beta$ let denote the corresponding stratum, and denote by 
$$(X^\mu_\beta\times \Ran)^{\subset}:=X^\mu_\beta\underset{X^\mu}\times (X^\mu \times \Ran)^{\subset}) \overset{\iota_\beta}\hookrightarrow  
(X^\mu \times \Ran)^{\subset}$$
and 
$$(X^\mu_\beta\times \Ran)^{\subset} 
\overset{\on{pr}_{\Ran,\beta}^\mu}\longrightarrow X^\mu_\beta$$
the resulting maps.

\medskip

Let $\CF^\mu\in \Shv((X^\mu \times \Ran)^{\subset})$ be such that
$$(\bi^\mu)^!({}'\!\IC^\semiinf_\Ran)\simeq (p^\mu_\Ran)^!(\CF^\mu).$$

\medskip

Since the functor $(\on{pr}_{\Ran}^\mu)^!$ is left t-exact and fully faithful
(the latter by \corref{c:pr contr}), in order to prove \propref{p:! retsr IC'}, it suffices to show that each
$$(\iota_\beta)^!\circ \CF^\mu \in \Shv((X^\mu_\beta\times \Ran)^{\subset})$$
is of the form
$$(\on{pr}_{\Ran,\beta}^\mu)^!(\CF^\mu_\beta),$$ 
where $\CF^\mu_\beta\in \Shv(X^\mu_\beta)$ is such that $\CF^\mu_\beta[\langle \mu,2\check\rho\rangle]$
is strictly coconnective.

\sssec{}

By the factorization property of $'\!\IC^\semiinf_\Ran$ (see \eqref{e:factor 'IC}),
it suffices to prove the above assertion for $\beta$ corresponding to the main diagonal
$X \to X^\mu$. Denote the corresponding stratum in $(X^\mu \times \Ran)^{\subset}$ by $$(X\times \Ran)^{\subset}.$$ Denote the
corresponding map $\on{pr}_{\Ran,\beta}^\mu$ by 
$$\on{pr}_{(X\times \Ran)^{\subset}}^\mu:(X\times \Ran)^{\subset}\to X.$$
Denote the restriction of the section 
$$s^\mu_\Ran:(X^\mu \times \Ran)^{\subset}\to S^\mu_\Ran$$
to this stratum by $s^\mu_{(X\times \Ran)^{\subset}}$. 

\medskip

We claim that 
$$(s^\mu_{(X\times \Ran)^{\subset}})^!({}'\!\IC^\semiinf_\Ran)\simeq (\on{pr}_{(X\times \Ran)^{\subset}}^\mu)^!
(\omega_X)\otimes \sW_\mu[-\langle \mu,2\check\rho\rangle],$$
where $\sW_\mu\in \Vect$ lives in cohomological degrees $\geq 2$.

\begin{rem} 
One can show that there is a canonical identification
$$\sW_\mu\simeq \Sym(\cn^-[-2])(\mu),$$
where $\cn$ is the unipotent radical of the Langlands dual Lie algebra. In fact, such an identification would
follow once we prove \thmref{t:descr of restr} for the !-restrictions.
\end{rem} 

\sssec{}

In fact, we claim that for every $I$, we have:
$$(s^\mu_{(X\times X^I)^{\subset}})^!({}'\!\IC^\semiinf_I)\simeq (\on{pr}_{(X\times X^I)^{\subset}}^\mu)^!
(\omega_X)\otimes \sW_\mu[-\langle \mu,2\check\rho\rangle],$$
where 
$$\on{pr}_{(X\times X^I)^{\subset}}^\mu:=\on{pr}_{(X\times \Ran)^{\subset}}^\mu|_{(X\times X^I)^{\subset}}.$$

\sssec{}

Indeed, it follows from the definitions that for any $\ul\lambda:I\to \Lambda^+$, 
$$(s^\mu_{(X\times X^I)^{\subset}})^!(\delta_{-\ul\lambda}\star \Sat_{G,I}(V^{\ul\lambda}))[\langle \lambda,2\check\rho\rangle]
\simeq (\on{pr}_{(X\times X^I)^{\subset}}^\mu)^!
(\omega_X)\otimes W_{\lambda,\mu}[-\langle \mu,2\check\rho\rangle],$$
where $W_{\lambda,\mu}$ is the cohomogically graded vector space such that
$$\Sat(V^\lambda)|_{\Gr_G^{\lambda+\mu}}\simeq \IC_{\Gr_G^{\lambda+\mu}}\otimes W_{\lambda,\mu},\quad \sW_{\lambda,\mu}\in \Vect,$$
where $-|_{-}$ means !-restriction. By parity vanishing, $\sW_{\lambda,\mu}$ lives in cohomological degrees $\geq 2$.

\medskip

Finally, 
$$\sW_\mu=\underset{\lambda\in \Lambda^+}{\on{colim}}\, W_{\lambda,\mu},$$
and the cohomological estimate holds for $\sW_\mu$ because the poset $\Lambda^+$ is filtered (here we use the
assumption that $G$ is semi-simple and simply connected). 

\section{The semi-infinite IC sheaf and Drinfeld's compactification}  \label{s:glob}

In this section we will express $\IC^\semiinf_\Ran$ in terms of an \emph{actual intersection cohomology sheaf}, i.e., one 
arising in finite-dimensional algebraic geometry (technically, on an algebraic stack locally of finite type). 

\medskip

Throughout this section, the curve $X$ is assumed to be proper. 

\ssec{Drinfeld's compactification}

\sssec{}

Let $\BunBb$ Drinfeld's relative compactification of the stack $\Bun_B$ along the fibers of the map
$\Bun_B\to \Bun_G$.

\medskip

I.e., $\BunBb$ is the algebraic stack that classifies triples $(\CP_G,\CP_T,\kappa)$, where:

\medskip

\noindent(i) $\CP_G$ is a $G$-bundle on $X$;

\medskip

\noindent(ii) $\CP_T$ is a $T$-bundle on $X$;

\medskip

\noindent(iii) $\kappa$ is a \emph{Pl\"ucker} data, i.e., a system of non-zero maps 
$$\kappa^{\check\lambda}:\check\lambda(\CP_T)\to \CV^{\check\lambda}_{\CP_G},$$
(here $\CV^{\check\lambda}$ denotes the Weyl module with highest weight $\check\lambda\in \check\Lambda^+$)
that satisfy Pl\"ucker relations, i.e., for $\check\lambda_1$ and $\check\lambda_2$ the diagram
$$
\CD
\check\lambda_1(\CP_T)\otimes \check\lambda_2(\CP_T)  @>{\kappa^{\check\lambda_1}\otimes \kappa^{\check\lambda_2}}>> 
\CV^{\check\lambda_1}_{\CP_G}\otimes \CV^{\check\lambda_2}_{\CP_G}  \\
@A{\sim}AA   @AAA  \\
(\check\lambda_1+\check\lambda_2)(\CP_T)  @>{\kappa^{\check\lambda_1+\check\lambda_2}}>> \CV^{\check\lambda_1+\check\lambda_2}_{\CP_G} 
\endCD
$$
must commute. 

\medskip

The open substack
\begin{equation} \label{e:BunB}
\Bun_B \overset{\bj_{\on{glob}}}\hookrightarrow \BunBb
\end{equation}
corresponds to the condition that the maps $\kappa^{\check\lambda}$ be injective bundle maps. 

\medskip

We denote by $\ol\sfp$ (resp., $\ol\sfq$)
resulting map from $\BunBb$ to $\Bun_G$ (resp., $\Bun_T$), which sends $(\CP_G,\CP_T,\kappa)$ to
$\CP_G$ (resp., $\CP_T$). 

\medskip

Its restriction to $\Bun_B\subset \BunBb$ is the usual map
$\sfq:\Bun_B\to \Bun_G$ (resp., $\sfq:\Bun_B\to \Bun_T$) 
induced by the map of groups $B\to G$ (resp., $B\to T$).

\sssec{}

For $\lambda\in \Lambda^{\on{neg}}$ we let $\ol\bi{}^\lambda_{\on{glob}}$ denote the map 
$$\BunBb^{\leq \lambda}:=\BunBb\times X^\lambda \to \BunBb,$$
given by 
$$(\CP_G,\CP_T,\kappa,D)\mapsto (\CP'_G,\CP'_T,\kappa')$$ with
$\CP'_G=\CP_G$, $\CP'_T=\CP_T(D)$ and $\kappa'$ given by precomposing $\kappa$ with the natural maps
$$\check\lambda(\CP'_T)=\check\lambda(\CP_T)(\check\lambda(D))\hookrightarrow \check\lambda(\CP_T).$$

\medskip

It is known that $\ol\bi{}^\lambda_{\on{glob}}$ is a finite morphism. 

\sssec{}

Let $\bj^\lambda_{\on{glob}}$ denote the open embedding
$$\BunBb^{=\lambda}:=\Bun_B\times X^\lambda \hookrightarrow \BunBb\times X^\lambda=:\BunBb^{\leq \lambda}.$$

Denote 
$$\bi^\lambda_{\on{glob}}=\ol\bi{}^\lambda_{\on{glob}}\circ \bj^\lambda_{\on{glob}}.$$

Note that by definition $\bi^0_{\on{glob}}=\bj^0_{\on{glob}}=\bj_{\on{glob}}$ is the embedding \eqref{e:BunB}. 

\medskip

The following is known:

\begin{lem}
The maps $\bi^\lambda_{\on{glob}}$ are locally closed embeddings. Every field-valued point of $\BunBb$ belongs to 
the image of exactly one such map.
\end{lem} 

\ssec{The global semi-infinite category}

\sssec{}

Denote
$$\BunNb:=\BunBb\underset{\Bun_T}\times \on{pt},\,\, \BunNb^{\leq \lambda}:=\BunBb^{\leq \lambda}\underset{\Bun_T}\times \on{pt},\,\,
\BunNb^{=\lambda}:=\BunBb^{=\lambda}\underset{\Bun_T}\times \on{pt},$$
where $\on{pt}\to \Bun_T$ corresponds to the trivial bundle and the map $\BunBb^{\leq \lambda}\to \Bun_T$ is
$$\BunBb^{\leq \lambda}=\BunBb\times X^\lambda \overset{\ol\sfq\times \on{id}}\longrightarrow \Bun_T\times X^\lambda
\overset{\on{id}\times \on{AJ}}\longrightarrow \Bun_T\times \Bun_T \overset{\on{mult}}\longrightarrow \Bun_T,$$
where $\on{AJ}$ is the Abel-Jacobi map,
$$D\mapsto \CO(D).$$ 

\medskip

In particular,
$$\BunNb^{=\lambda}\simeq \Bun_B\underset{\Bun_T}\times X^\lambda,$$
where $X^\lambda\to \Bun_T$ is the composition of the map $\on{AJ}$ and inversion
on $\Bun_T$. 

\medskip

We will denote by the same symbols the corresponding maps
$$\ol\bi{}^\lambda_{\on{glob}}:\BunNb^{\leq \lambda}\to \BunNb,\,\,
\bj^\lambda_{\on{glob}}:\BunNb^{=\lambda}\to \BunNb^{\leq \lambda},\,\,
\bi^\lambda_{\on{glob}}:\BunNb^{=\lambda}\to \BunNb.$$

\medskip

Denote by $p^\lambda_{\on{glob}}$ (resp., $\ol{p}^\lambda_{\on{glob}}$) 
the projection from $\BunNb^{=\lambda}$ (resp., $\BunNb^{\leq\lambda}$) to 
$X^\lambda$.

\sssec{}

We define 
\begin{equation} \label{e:global SI}
\SI^{\leq 0}_{\on{glob}}\subset \Shv(\BunNb)
\end{equation}
to be the full subcategory defined by the following condition: 

\medskip

An object $\CF\in \Shv(\BunNb)$ belongs to $\SI^{\leq 0}_{\on{glob}}$ if and only if for every $\lambda\in \Lambda^{\on{neg}}$, the object 
$$(\bi^\lambda_{\on{glob}})^!(\CF)\in \Shv(\BunNb^{=\lambda})$$ 
belongs to the full subcategory
\begin{equation} \label{e:global SI lambda}
\SI^{=\lambda}_{\on{glob}} \subset \Shv(\BunNb^{=\lambda}),
\end{equation}
equal by definition to the essential image of the pullback functor
$$(p^\lambda_{\on{glob}})^!:\Shv(X^\lambda)\to \Shv(\BunNb^{=\lambda}).$$

We note that the above pullback functor is fully faithful, since the map
$p^\lambda_{\on{glob}}$, being a base change of $\Bun_B\to \Bun_T$, is smooth with homologically contractible fibers. 

\sssec{}  \label{sss:si Av}

Proceeding as in \cite[Sects. 4.5-4.7]{Ga5}, one shows that the full subcategory \eqref{e:global SI} 
(resp., \eqref{e:global SI lambda}) can also be defined by an 
equivariance condition with respect to a certain pro-unipotent groupoid. 

\medskip

In particular, the embedding  \eqref{e:global SI} (resp., \eqref{e:global SI lambda}) admits
a \emph{right} adjoint\footnote{The corresponding assertion 
would be \emph{false} for the corresponding embedding $\SI^{\leq 0}_\Ran\subset \Shv(\ol{S}{}^0_\Ran)$;
this is a geometric counterpart of the fact that the local field is not compact, while the quotient of adeles 
by principal adeles is compact.},
to be denoted $\on{Av}_*^{\on{SI}}$. 

\sssec{}

The functors
\begin{equation} \label{e:p^! glob}
(\bi^\lambda_{\on{glob}})^!: \Shv(\BunNb) \to  \Shv(\BunNb^{=\lambda})
\end{equation}
and
\begin{equation} \label{e:p_* glob}
(\bi^\lambda_{\on{glob}})_*: \Shv(\BunNb^{=\lambda})\to \Shv(\BunNb)
\end{equation}
induce (same-named) functors 
\begin{equation} \label{e:p^! glob si}
(\bi^\lambda_{\on{glob}})^!: \SI^{\leq 0}_{\on{glob}}\to \SI^{=\lambda}_{\on{glob}}
\end{equation}
and
\begin{equation} \label{e:p_* glob si}
(\bi^\lambda_{\on{glob}})_*: \SI^{=\lambda}_{\on{glob}}\to \SI^{\leq 0}_{\on{glob}}.
\end{equation}

Moreover, the diagram
\begin{equation}  \label{e:Av diag}
\CD
\SI^{\leq 0}_{\on{glob}}   @<{\on{Av}_*^{\on{SI}}}<<   \Shv(\BunNb)  \\
@V{(\bi^\lambda_{\on{glob}})^!}VV  @VV{(\bi^\lambda_{\on{glob}})^!}V   \\
\SI^{=\lambda}_{\on{glob}}  @<{\on{Av}_*^{\on{SI}}}<<  \Shv(\BunNb^{=\lambda}),
\endCD
\end{equation}
and similarly for $(\bi^\lambda_{\on{glob}})_*$.

\begin{prop} \label{p:i_! glob}  \hfill

\smallskip

\noindent{\em(a)}
The partially defined left adjoint $(\bi^\lambda_{\on{glob}})_!$ of \eqref{e:p^! glob}
$$(\bi^\lambda_{\on{glob}})^!: \Shv(\BunNb) \to  \Shv(\BunNb^{=\lambda})$$
is defined on 
$$\SI^{=\lambda}_{\on{glob}}\subset \Shv(\BunNb^{=\lambda}).$$

\smallskip

\noindent{\em(b)}
The resulting functor 
$$\SI^{=\lambda}_{\on{glob}}\to \Shv(\BunNb)$$ 
takes values in
$$\SI^{\leq 0}_{\on{glob}}\subset \Shv(\BunNb),$$ 
and thus provides a left adjoint to \eqref{e:p^! glob si}.

\end{prop}

\begin{proof}

To prove point (a), it suffices to do so for the map $\bj^\lambda_{\on{glob}}$. We claim that the objects
$$(\bj^\lambda_{\on{glob}})_!(\omega_{\BunNb^{=\lambda}})\in \Shv(\BunNb^{\leq \lambda}))$$
and 
$$\omega_{\BunNb^{=\lambda}}\in \Shv(\BunNb^{=\lambda}))$$
are ULA with respect to the maps $\ol{p}^\lambda_{\on{glob}}$ and $p^\lambda_{\on{glob}}$,
respectively. This would imply that for $\CF\in \Shv(X^\lambda)$, we have
$$(\bj^\lambda_{\on{glob}})_!\circ (p^\lambda_{\on{glob}})^!(\CF)\simeq 
(\ol{p}^\lambda_{\on{glob}})^!(\CF),$$
in particular, giving a formula for the left-hand side as an object of $\Shv(\BunNb^{\leq \lambda}))$.

\medskip

To prove the required ULA property, it suffices to do so for the embedding
$$\bj_{\on{glob}}:\Bun_B\hookrightarrow \BunBb,$$
in which case, this is the assertion of \cite[Theorem 5.1.5]{BG1}. 

\medskip

Point (b) follows from the commutativity of the diagram \eqref{e:Av diag} by passing to left adjoints. 

\end{proof}

By a Cousin argument, it follows formally from \propref{p:i_! glob}
that the partially defined functor $(\bi^\lambda_{\on{glob}})^*$, left adjoint to \eqref{e:p_* glob}, is defined on
$\SI^{\leq 0}_{\on{glob}}\subset \Shv(\BunNb)$ and takes values in $\SI^{=\lambda}_{\on{glob}} \subset \Shv(\BunNb^{=\lambda})$,
thus providing a left adjoint to \eqref{e:p_* glob si}.

\sssec{}

The embeddings
$$\SI^{=\lambda}_{\on{glob}} \hookrightarrow \Shv(\BunNb^{=\lambda}) \text{ and } 
\SI^{\leq 0}_{\on{glob}}\hookrightarrow \Shv(\BunNb)$$
are compatible with the t-structure on the target categories. This follows from the fact that the right adjoints $\on{Av}^{\on{SI}}_*$ 
(see \secref{sss:si Av}) are right t-exact. 

\medskip

Hence, the categories $\SI^{=\lambda}_{\on{glob}}$ and $\SI^{\leq 0}_{\on{glob}}$ acquire t-structures. By construction, an object
$\CF\in \SI^{\leq 0}_{\on{glob}}$ is connective (resp., coconnective) if and only if $(\bi^\lambda_{\on{glob}})^*(\CF)$
(resp., $(\bi^\lambda_{\on{glob}})^!(\CF)$ is connective (resp., coconnective) for every $\lambda\in \Lambda^{\on{neg}}$. 

\sssec{}

We will denote by 
$$\IC^\semiinf_{\on{glob}}\in (\SI^{\leq 0}_{\on{glob}})^\heartsuit$$
the minimal extension of $\IC_{\Bun_N}\in (\SI^{=0}_{\on{glob}})^\heartsuit$ along $\bj^0_{\on{glob}}$. 

\ssec{Local vs global compatibility for the semi-infinite IC sheaf}

\sssec{}

For every finite set $I$ we have a canonically defined map
$$\pi_I:\ol{S}{}^0_I\to \BunNb.$$

Together these maps combine to a map
$$\pi_\Ran:\ol{S}{}^0_\Ran\to \BunNb.$$

\sssec{}

Let $d=\dim(\Bun_N)=(g-1)\cdot \dim(N)$. The main result of this section is:

\begin{thm} \label{t:IC loc glob}
There exists a (unique) isomorphism
$$(\pi_\Ran)^!(\IC^\semiinf_{\on{glob}})[d]=\IC^\semiinf_\Ran,$$
extending the tautological identification over $\Bun_N$. 
\end{thm} 

\sssec{}

The next few subsections are devoted to the proof of this theorem. Modulo auxiliary assertions,
the proof will be given in \secref{sss:loc to glob exact}. 

\ssec{The local vs global compatibility for the semi-infinite category}

This subsection contains some preparatory material for the proof of \thmref{t:IC loc glob}. 

\sssec{}

First, we observe:

\begin{lem}
For every $\lambda$, we have a commutative diagram
$$
\CD
\ol{S}{}^{\leq \lambda}_\Ran  @>{\ol\bi{}^\lambda}>>  \ol{S}{}^0_\Ran  \\
@VVV   @VV{\pi_\Ran}V  \\
\BunNb^{\leq \lambda} @>{\ol\bi{}^\lambda_{\on{glob}}}>>   \BunNb.
\endCD
$$
The corresponding diagram
\begin{equation} \label{e:stratum cart}
\CD
S^{=\lambda}_\Ran  @>{\bi^\lambda}>>  \ol{S}{}^0_\Ran  \\
@V{\pi^\lambda_\Ran}VV   @VV{\pi_\Ran}V  \\
\BunNb^{=\lambda} @>{\bi^\lambda_{\on{glob}}}>>   \BunNb.
\endCD
\end{equation} 
is Cartesian, and we have a commutative diagram
$$
\CD
S^{=\lambda}_\Ran   @>{p^\lambda_\Ran}>>  (X^\lambda \times \Ran)^{\subset} \\
@V{\pi^\lambda_\Ran}VV   @VV{\on{pr}^\lambda_\Ran}V  \\
\BunNb^{=\lambda} @>{p^\lambda_{\on{glob}}}>>  X^\lambda.
\endCD
$$
\end{lem} 

The assertions parallel to those in the above lemma hold for $\Ran$ replaced by $X^I$ for an individual finite set $I$. 

\sssec{}

The following assertion is not necessary for the needs of this paper, but we will prove it for the sake of completeness
(see \secref{sss:proof of contr}):

\begin{thm} \label{t:contr}
The functor
$$(\pi_\Ran)^!:\Shv(\BunNb)\to \Shv(\ol{S}{}^0_\Ran)$$
is fully faithful.
\end{thm}

When working with an individual stratum, a stronger assertion is true (to be proved in \secref{ss:proof of contr}):
Consider the map
$$(p^\lambda_\Ran\times \pi^\lambda_\Ran):S^\lambda_\Ran\to (X^\lambda \times \Ran)^{\subset}\underset{X^\lambda}\times \BunNb^{=\lambda}.$$

\begin{prop} \label{p:contr}
The functor 
$$(p^\lambda_\Ran\times \pi^\lambda_\Ran)^!:
\Shv((X^\lambda \times \Ran)^{\subset}\underset{X^\lambda}\times \BunNb^{=\lambda})\to \Shv(S^\lambda_\Ran)$$
is fully faithful.
\end{prop}

Combining with \lemref{l:pr contr}, we obtain: 

\begin{cor} \label{c:contr}
The functor
$$(\pi^\lambda_\Ran)^!:\Shv(\BunNb^{=\lambda})\to \Shv(S^\lambda_\Ran)$$
is fully faithful.
\end{cor}

\sssec{}

Next we claim:

\begin{prop}  \label{p:pullback pres N}
For every finite set $I$, the functor 
$$(\pi_I)^!:\Shv(\BunNb)\to \Shv(\ol{S}{}^0_I)$$
sends $\SI^{\leq 0}_{\on{glob}}$ to $\SI^{\leq 0}_I$. 
\end{prop}

\begin{proof}

Note that an object $\CF\in  \Shv(\ol{S}{}^0_I)$ belongs to $\SI^{\leq 0}_I$ if and only if $(\bi^\lambda)^!(\CF)$ belongs to 
$\SI^{=\lambda}_I$ for every $\lambda$. Now the result follows from the identification 
$$\on{pr}^\lambda_I\circ p^\lambda_I=p^\lambda_{\on{glob}}\circ \pi^\lambda_I.$$

\end{proof} 

We will now deduce: 

\begin{cor}  \label{c:global vs global}
An object of $\Shv(\BunNb)$ belongs to $\SI^{\leq 0}_{\on{glob}}$ \emph{if and only if} its pullback 
under $(\pi_\Ran)^!$ belongs to $\SI^{\leq 0}_\Ran\subset \Shv(\ol{S}{}^0_\Ran)$.
\end{cor}

\begin{proof}
The ``only if" direction is the content of \propref{p:pullback pres N}. 

\medskip

For the ``if" direction, we need to show that if an object $\CF\in \Shv(\BunNb^{=\lambda})$ is such that 
$$(\pi^\lambda_\Ran)^!(\CF)\simeq (p^\lambda_\Ran)^!(\CF')$$
for some $\CF'\in \Shv((X^\lambda \times \Ran)^{\subset})$, then $\CF$ is the pullback 
of an object in $\Shv(X^\lambda)$ along $p^\lambda_{\on{glob}}$.

\medskip

By \propref{p:contr}, in the diagram
$$
\CD
S^\lambda_\Ran \\
@V{p^\lambda_\Ran\times \pi^\lambda_\Ran}VV \\
(X^\lambda \times \Ran)^{\subset}\underset{X^\lambda}\times \BunNb^{=\lambda} 
@>{\on{id}_{(X^\lambda \times \Ran)^{\subset}}\times p^\lambda_{\on{glob}}}>>  (X^\lambda \times \Ran)^{\subset} \\
@V{\on{pr}^\lambda_\Ran\times \on{id}_{\BunNb^{=\lambda}}}VV  @VV{\on{pr}^\lambda_\Ran}V  \\
\BunNb^{=\lambda} @>{p^\lambda_{\on{glob}}}>> X^\lambda
\endCD
$$
we have
\begin{multline*} 
\CF\overset{\text{\lemref{l:pr contr}}}\simeq 
(\on{pr}^\lambda_\Ran\times \on{id}_{\BunNb^{=\lambda}})_!\circ (\on{pr}^\lambda_\Ran\times \on{id}_{\BunNb^{=\lambda}})^!(\CF)
\overset{\text{\propref{p:contr}}}\simeq \\
\simeq 
(\on{pr}^\lambda_\Ran\times \on{id}_{\BunNb^{=\lambda}})_!\circ
(p^\lambda_\Ran\times \pi^\lambda_\Ran)_!\circ (p^\lambda_\Ran\times \pi^\lambda_\Ran)^! \circ 
(\on{pr}^\lambda_\Ran\times \on{id}_{\BunNb^{=\lambda}})^!(\CF) \simeq \\
\simeq (\on{pr}^\lambda_\Ran\times \on{id}_{\BunNb^{=\lambda}})_!\circ
(p^\lambda_\Ran\times \pi^\lambda_\Ran)_!\circ (\pi^\lambda_\Ran)^!(\CF) \overset{\text{assumption}}\simeq \\
\simeq (\on{pr}^\lambda_\Ran\times \on{id}_{\BunNb^{=\lambda}})_!\circ
(p^\lambda_\Ran\times \pi^\lambda_\Ran)_!\circ (p^\lambda_\Ran)^!(\CF') \simeq \\
\simeq 
(\on{pr}^\lambda_\Ran\times \on{id}_{\BunNb^{=\lambda}})_!\circ
(p^\lambda_\Ran\times \pi^\lambda_\Ran)_!\circ (p^\lambda_\Ran\times \pi^\lambda_\Ran)^!\circ 
(\on{id}_{(X^\lambda \times \Ran)^{\subset}}\times p^\lambda_{\on{glob}})^!(\CF')\simeq \\
\overset{\text{\propref{p:contr}}}\simeq 
(\on{pr}^\lambda_\Ran\times \on{id}_{\BunNb^{=\lambda}})_!\circ 
(\on{id}_{(X^\lambda \times \Ran)^{\subset}}\times p^\lambda_{\on{glob}})^!(\CF')\simeq 
(p^\lambda_{\on{glob}})^!\circ (\on{pr}^\lambda_\Ran)_!(\CF'),
\end{multline*}
as required (the last isomorphism is base change, which holds due to the fact that the map $\on{pr}^\lambda_\Ran$
is pseudo-proper\footnote{See \secref{sss:pseudo-proper}, where this notion is recalled.}).

\end{proof} 

\ssec{Proof of \propref{p:contr}} \label{ss:proof of contr}

\sssec{}

Consider the morphism 
$$(p^\lambda_\Ran\times \pi^\lambda_\Ran):S^\lambda_\Ran\to 
(X^\lambda \times \Ran)^{\subset}\underset{X^\lambda}\times \BunNb^{=\lambda}.$$

\medskip

A point of $S^\lambda_\Ran$ is the following data:

\medskip

\noindent{(i)} A $B$-bundle $\CP_B$ on $X$ (denote by $\CP_T$ the induced $T$-bundle);

\medskip

\noindent{(ii)} A $\Lambda^{\on{neg}}$-valued divisor $D$ on $X$ (we denote by $\CO(D)$ the corresponding $T$-bundle); 

\medskip

\noindent{(iii)} An identification $\CP_T\simeq \CO(D)$;

\medskip

\noindent{(iv)} A finite non-empty set $\CI$ of points of $X$ that contains the support of $D$;

\medskip

\noindent{(v)} A trivialization $\alpha$ of $\CP_B$ away from $\CI$, such that the induced
trivialization of $\CP_T|_{X-\CI}$ agrees with the tautological trivialization of $\CO(D)|_{X-\CI}$. 

\sssec{}

The map $(p^\lambda_\Ran\times \pi^\lambda_\Ran)$ amounts to forgetting the data of (v) above. It is clear
that for an affine test-scheme $Y$ and a $Y$-point of 
$$(X^\lambda \times \Ran)^{\subset}\underset{X^\lambda}\times \BunNb^{=\lambda},$$
the set of its lifts to a $Y$-point of $S^\lambda_\Ran$ is non-empty and is a torsor for the group
$$\Maps(Y\times X-\Gamma_\CI,N).$$ 

\medskip

For a given $Y$ and $\CI\subset \Maps(Y,X)$, let $\bMaps_Y(X-\CI,N)$ be the prestack over $Y$ that
assigns to $Y'\to Y$ the set of maps
$$\Maps(Y'\times X-(Y'\underset{Y}\times \Gamma_\CI),N).$$

Thus, it suffices to show that the projection $\bMaps_Y(X-\CI,N)\to Y$ is \emph{universally homologically contractible},
see \secref{sss:UHC} for what this means.

\sssec{}

Since $N$ is unipotent, it is isomorphic to $\BA^m$, where $m=\dim(N)$. Hence, it suffices to show that the map
$$\bMaps_Y(X-\CI,\BA^1)\to Y$$
is universally homologically contractible. 

\medskip

However, the latter is clear: the prestack $\bMaps_Y(X-\CI,\BA^1)$ is isomorphic to the ind-scheme 
$\BA^\infty\times Y$, where 
$$\BA^\infty\simeq \underset{n}{\on{colim}}\, \BA^n.$$

\qed

\ssec{The key isomorphism}

\sssec{}

The base change isomorphism
$$(\pi_I)^!\circ (\bi^\lambda_{\on{glob}})_*\simeq (\bi^\lambda)_*\circ (\pi_I)^!$$
in the diagram \eqref{e:stratum cart} gives rise to a natural transformation
\begin{equation} \label{e:restr nat trans}
(\bi^\lambda)^*\circ (\pi_I)^!\to (\pi^\lambda_I)^!\circ (\bi^\lambda_{\on{glob}})^*
\end{equation}
as functors 
$$\SI_{\on{glob}}^{\leq 0} \rightrightarrows \SI^{=\lambda}_I,$$
see \propref{p:i* well defined}(a) for the notation $(\bi^\lambda)^*$. 

\sssec{}

In \secref{ss:proof of key}, we will prove:

\begin{prop} \label{p:restr isom} 
The natural transformation \eqref{e:restr nat trans} is an isomorphism.
\end{prop}

We will now deduce some corollaries of \propref{p:restr isom}; these will easily imply \thmref{t:IC loc glob},
see \secref{sss:loc to glob exact}.

\medskip

First, combining \propref{p:restr isom} with \propref{p:i* well defined}(c), we obtain:

\begin{cor}  \label{c:restr isom}
The natural transformation
$$(\bi^\lambda)^*\circ (\pi_\Ran)^!\to (\pi^\lambda_\Ran)^!\circ (\bi^\lambda_{\on{glob}})^*$$
as functors
$$\SI_{\on{glob}}^{\leq 0} \rightrightarrows \SI^{=\lambda}_\Ran$$
is an isomorphism.
\end{cor}

Next, by a Cousin argument, from \propref{p:restr isom} we formally obtain:

\begin{cor} \label{c:ext isom}
The natural transformation
$$(\bi^\lambda)_!\circ (\pi^\lambda_I)^!\to (\pi_I)^!\circ (\bi^\lambda_{\on{glob}})_!,$$
arising by adjunction from
$$(\pi^\lambda_I)^!\circ (\bi^\lambda_{\on{glob}})^!\simeq (\bi^\lambda)^!\circ (\pi_I)^!,$$
is an isomorphism of functors
$$\SI^{=\lambda}_{\on{glob}}  \rightrightarrows \SI_I^{\leq 0}.$$
\end{cor}

Combining \corref{c:ext isom} with \corref{c:i! well defined}(c), we obtain:

\begin{cor} \label{c:ext isom Ran}
The natural transformation
$$(\bi^\lambda)_!\circ (\pi^\lambda_\Ran)^!\to (\pi_\Ran)^!\circ (\bi^\lambda_{\on{glob}})_!$$
is an isomorphism of functors
$$\SI^{=\lambda}_{\on{glob}}  \rightrightarrows \SI_\Ran^{\leq 0}.$$
\end{cor}

\medskip

Finally, we claim:

\begin{cor}  \label{c:log glob exact}
The functor
$$\pi^![d]:\SI_{\on{glob}}^{\leq 0}\to \SI_{\Ran}^{\leq 0}$$
is t-exact.
\end{cor} 

\begin{proof}

This follows from \corref{c:restr isom}, combined with the (tautological) isomorphism
$$(\bi^\lambda)^!\circ (\pi_\Ran)^!\simeq (\pi^\lambda_\Ran)^!\circ (\bi^\lambda_{\on{glob}})^!.$$

\end{proof}

\sssec{}  \label{sss:loc to glob exact}

Note that \corref{c:log glob exact} immediately implies \thmref{t:IC loc glob}. 

\begin{rem}
In \secref{ss:proof of Hecke compat}
we will present another construction of the map in one direction 
$$\IC^\semiinf_\Ran\to \pi^!(\IC^\semiinf_{\on{glob}})[d],$$
where we will realize $\IC^\semiinf_\Ran$ as $'\!\IC^\semiinf_\Ran$.
\end{rem} 

\sssec{}  \label{sss:loc to glob exact bis}

Let us now prove \propref{p:! and * perv}. 

\begin{proof} 
By \corref{c:log glob exact}, it suffices to show that the objects
$$(\bi^\lambda_{\on{glob}})_!(\IC_{\BunNb^{=\lambda}}) \text{ and } (\bi^\lambda_{\on{glob}})_*(\IC_{\BunNb^{=\lambda}})$$
belong to the heart of the t-structure (i.e., are perverse sheaves on $\BunNb$). 

\medskip

We claim that the morphism $\bi^\lambda_{\on{glob}}$ is affine, which would imply that the functors $(\bi^\lambda_{\on{glob}})_!$
and $(\bi^\lambda_{\on{glob}})_*$ are t-exact.

\medskip

Indeed, $\bi^\lambda_{\on{glob}}$ is the base change
of the morphism 
$$\bi^\lambda_{\on{glob}}:\Bun_B\times X^\lambda\to \BunBb,$$
which we claim to be affine. 

\medskip

Indeed, $\bi^\lambda_{\on{glob}}=\ol\bi{}^\lambda_{\on{glob}}\circ \bj^\lambda_{\on{glob}}$,
where $\ol\bi{}^\lambda_{\on{glob}}$ is a finite morphism, and $\bj^\lambda_{\on{glob}}$ 
is known to be an affine open embedding (see \cite[Proposition 3.3.1]{FGV}).
\end{proof}

\ssec{Proof of \propref{p:restr isom}}  \label{ss:proof of key}

\sssec{}

Let $\CF$ be an object of $\SI^{\leq 0}_{\on{glob}}$. We need to show that the map 
\begin{equation} \label{e:LSH & RHS}
(s^\lambda_I)^!\circ (\bi^\lambda)^*\circ (\pi_I)^!(\CF)\simeq (s^\lambda_I)^!\circ (\pi^\lambda_I)^! \circ (\bi^\lambda_{\on{glob}})^*(\CF)
\end{equation} 
is an isomorphism. 

\sssec{}

We first rewrite the left-hand side in \eqref{e:LSH & RHS}. 

\medskip

As a first step, we note that by \eqref{e:Braden semiinf}, we have
\begin{equation} \label{e:LSH 1}
(s^\lambda_I)^!\circ (\bi^\lambda)^*\circ (\pi_I)^!(\CF)\simeq (p^{-,\lambda}_I)_*\circ (\bi^{-,\lambda})^! \circ (\pi_I)^!(\CF).
\end{equation} 

\sssec{}   \label{sss:Zastava}

For $\lambda\in \Lambda^{\on{neg}}$, let $\CZ^\lambda$ be the Zastava space, i.e., this is the open substack of
$$\BunNb\underset{\Bun_G}\times \Bun_{B^-}^{-\lambda},$$
corresponding to the condition that the $B^-$-reduction and the generalized $N$-reduction of
a given $G$-bundle are generically transversal. 

\medskip

Let $\fq$ denote the forgetful map $\CZ^\lambda\to \BunNb$. Let $\fp$ denote the projection
$$\CZ^\lambda\to X^\lambda,$$
and let $\fs$ denote its section
$$X^\lambda\to \CZ^\lambda.$$

\sssec{}

Note that we have a canonical identification
\begin{equation} \label{e:ident Zast}
(X^\lambda\times X^I)^{\subset}\underset{X^\lambda}\times \CZ^\lambda\simeq
\ol{S}{}^0_I\cap S^{-,\lambda}_I,
\end{equation}
so that the projection
$$(\on{id}_{(X^\lambda\times X^I)^{\subset}}\times \fp):(X^\lambda\times X^I)^{\subset}\underset{X^\lambda}\times \CZ^\lambda\to
(X^\lambda\times X^I)^{\subset}$$ 
identifies with
$$\ol{S}{}^0_I\cap S^{-,\lambda}_I\to S^{-,\lambda}_I\overset{p^{-,\lambda}_I}\longrightarrow (X^\lambda\times X^I)^{\subset},$$ 

\sssec{}

Hence, the right-hand side in \eqref{e:LSH 1} can be rewritten as
\begin{equation} \label{e:rewrite RHS}
(\on{id}_{(X^\lambda\times X^I)^{\subset}}\times \fp)_*\circ (\on{pr}^\lambda_I\times \on{id}_{\CZ^\lambda})^! \circ \fq^!(\CF).
\end{equation} 
where the maps are as shown in the diagram
$$
\CD
(X^\lambda\times X^I)^{\subset}\underset{X^\lambda}\times \CZ^\lambda @>{\on{pr}^\lambda_I\times \on{id}_{\CZ^\lambda}}>>  
\CZ^\lambda @>{\fq}>>  \BunNb \\
@V{\on{id}_{(X^\lambda\times X^I)^{\subset}}\times \fp}VV   @VV{\fp}V  \\
(X^\lambda\times X^I)^{\subset} @>{\on{pr}^\lambda_I}>> X^\lambda.
\endCD
$$

By base change, we rewrite \eqref{e:rewrite RHS} as 
\begin{equation} \label{e:rewrite rewrite RHS}
(\on{pr}^\lambda_I)^!\circ \fp_*\circ \fq^!(\CF).
\end{equation} 

\sssec{}

The adjoint action of $T$ on $N$ defines an an action of $T$ on $\BunNb$. It is easy to see that every object
of $\SI^{\leq 0}$ is monodromic for this action. Hence, the same is true for its pullback to $\CZ^\lambda$. 

\medskip

Choose a dominant coweight as in \secref{sss:setting for Braden}. 
Applying the contraction principle for the action of $\BG_m$ along the fibers of $\fp$ (see \cite[Proposition 3.2.2]{DrGa}), 
we rewrite \eqref{e:rewrite rewrite RHS} as 
\begin{equation} \label{e:summary LHS}
(\on{pr}^\lambda_I)^!\circ \fs^* \circ \fq^!(\CF).
\end{equation}

To summarize, we have rewritten the left-hand side in \eqref{e:LSH & RHS} as \eqref{e:summary LHS}. 

\sssec{}

We now rewrite the right-hand side in \eqref{e:LSH & RHS} .

\medskip

Note that we have a Cartesian diagram
\begin{equation} \label{e:stratum on Zast}
\CD
X^\lambda  @>{\fs}>>  \CZ^\lambda \\
@V{\fq^\lambda}VV   @VV{\fq}V  \\
\BunNb^{=\lambda}  @>{\bi^\lambda_{\on{glob}}}>>  \BunNb,
\endCD
\end{equation} 
where the map $\fq^\lambda$ is given by
$$X^\lambda \simeq X^\lambda\underset{\Bun_T}\times \Bun_T \to
X^\lambda\underset{\Bun_T}\times \Bun_B\simeq \BunNb^{=\lambda}.$$

\medskip

Note also that the map
$$(X^\lambda\times X^I)^{\subset} \overset{s^\lambda_I}\longrightarrow S^\lambda_I\overset{\pi^\lambda_I}\longrightarrow \BunNb^{=\lambda}$$
identifies with
$$(X^\lambda\times X^I)^{\subset} \overset{\on{pr}^\lambda_I}\longrightarrow X^\lambda \overset{\fq^\lambda}\longrightarrow \BunNb^{=\lambda}.$$

Hence, the right-hand side in \eqref{e:LSH & RHS} identifies with
\begin{equation} \label{e:summary RHS}
(\on{pr}^\lambda_I)^!\circ (\fq^\lambda)^!\circ (\bi^\lambda_{\on{glob}})^*(\CF).
\end{equation} 

\sssec{}

Unwinding the identifications, we obtain that the map in \eqref{e:LSH & RHS} is induced by the natural transformation
\begin{equation} \label{e:!* pullbacks}
\fs^* \circ \fq^!\to (\fq^\lambda)^!\circ (\bi^\lambda_{\on{glob}})^*,
\end{equation} 
coming from the Cartesian square \eqref{e:stratum on Zast}.

\medskip

Thus, it suffices to show that the natural transformation \eqref{e:!* pullbacks}
is an isomorphism, when evaluated on objects from
$\SI^{\leq 0}_{\on{glob}}$. 

\medskip

However, the latter is done by repeating the argument of \cite[Sect. 3.9]{Ga1}: 

\medskip

We first consider the case when
$-\lambda$ is \emph{sufficiently dominant}, in which case the morphism $\fq$ is smooth, being the base change
of $\Bun^{-\lambda}_{B^-}\to \Bun_G$. In this case, the fact that \eqref{e:!* pullbacks} is an isomorphism follows
by smoothness. 

\medskip

Then we reduce the case of a general $\lambda$ to one above using the factorization property of $\CZ^\lambda$. 

\qed

\sssec{}

Thus, we have completed the proof of \propref{p:restr isom} and hence also of \thmref{t:IC loc glob}. 

\ssec{Relation to the IC sheaf on Zastava spaces}

\sssec{}

Recall the Zastava spaces 
$$\CZ^\lambda\subset \BunNb\underset{\Bun_G}\times \Bun_{B^-}^{-\lambda},$$
introduced in \secref{sss:Zastava}.

\medskip

Let $\oCZ{}^\lambda\subset \CZ^\lambda$ denote the open subscheme equal to
$$\Bun_N\underset{\BunNb}\times \CZ^\lambda.$$

\sssec{}

Note now that the identification \eqref{e:ident Zast} gives rise to a map
\begin{equation} \label{e:q'}
\fq':(X^\lambda\times \Ran)^{\subset}\underset{X^\lambda}\times \CZ^\lambda \to \ol{S}^0_{\Ran}.
\end{equation}

\medskip

Let $\on{pr}_\Ran^\lambda \times \on{id}_{\CZ^\lambda}$ denote the projection
$$(X^\lambda\times \Ran)^{\subset}\underset{X^\lambda}\times \CZ^\lambda\to \CZ^\lambda.$$

We claim:

\begin{prop}  \label{p:compare with IC Zast}
There exists a canonical isomorphism
$$(\on{pr}_\Ran^\lambda \times \on{id}_{\CZ^\lambda})^!(\IC_{Z^\lambda})\simeq 
(\fq')^!(\ICs)[\langle \lambda,2\check\rho\rangle],$$
extending the tautological identification of the restriction of either side to 
$$(X^\lambda\times \Ran)^{\subset}\underset{X^\lambda}\times \oCZ{}^\lambda$$
with $\omega_{(X^\lambda\times \Ran)^{\subset}\underset{X^\lambda}\times \oCZ{}^\lambda}[\langle \lambda,2\check\rho\rangle]$.
\end{prop} 

\sssec{Proof of \propref{p:compare with IC Zast}}

We have a commutative diagram
$$
\CD
(X^\lambda\times \Ran)^{\subset}\underset{X^\lambda}\times \CZ^\lambda    @>{\fq'}>>   \ol{S}^0_{\Ran}  \\
@V{\on{pr}_\Ran^\lambda \times \on{id}_{\CZ^\lambda}}VV    @VV{\pi_\Ran}V  \\
\CZ^\lambda  @>{\fq}>>  \BunNb.
\endCD
$$

According to \cite[Prop. 3.6.5(a)]{Ga1}, we have a canonical isomorphism
$$\fq^!(\ICs_{\on{glob}})[(g-1)\cdot \dim(N)+\langle \lambda,2\check\rho\rangle]\simeq \IC_{\CZ^\lambda}.$$

Now the assertion follows from \thmref{t:IC loc glob}.

\qed

\ssec{Computation of fibers}  \label{ss:proof calc IC}

In this subsection we will prove \thmref{t:descr of restr}. One possible proof follows from the description of
the objects
$$(\bi^\lambda_{\on{glob}})^!(\IC^\semiinf_{\on{glob}})  \text{ and } (\bi^\lambda_{\on{glob}})^*(\IC^\semiinf_{\on{glob}})$$
in \cite[Proposition 4.4]{BG2}, combined with \thmref{t:IC loc glob}. 

\medskip

Instead, we will actually reprove \cite[Proposition 4.4]{BG2}, see \thmref{t:descr of restr glob} below, using
our \thmref{t:IC loc glob}. 

\begin{rem}

Let us add a clarification on the order of the argument proving Theorems \ref{t:descr of restr} and \ref{t:descr of restr glob}:

\begin{enumerate}

\item In \secref{ss:pres as colimit}, we defined the object $'\!\IC^\semiinf_\Ran$;

\item In \propref{p:! retsr IC'} we showed that the !-restrictions of $'\!\IC^\semiinf_\Ran$ to the stata $S^\lambda_\Ran$ are strictly coconnective;

\item In \propref{p:* retsr IC'} we calculated the *-restrictions of $'\!\IC^\semiinf_\Ran$ to the stata $S^\lambda_\Ran$ and showed that they
are isomorphic to (the pullbacks) of $\on{Fact}^{\on{alg}}(\CO(\cN))_{X^\lambda}[-\langle \lambda,2\check\rho\rangle]$;
in particular, they are strictly connective;

\item Points (2) and (3) imply that $'\!\IC^\semiinf_\Ran$ is isomorphic to $\IC^\semiinf_\Ran$;

\item Points (3) and (4) imply that the *-restrictions of $\IC^\semiinf_\Ran$ to the stata $S^\lambda_\Ran$ 
are isomorphic to (the pullbacks) of $\on{Fact}^{\on{alg}}(\CO(\cN))_{X^\lambda}[-\langle \lambda,2\check\rho\rangle]$,
thus proving the part of \thmref{t:descr of restr} about *-restrictions;

\item Point (5) above, combined with \thmref{t:IC loc glob} and \corref{c:restr isom}, will imply \thmref{t:descr of restr glob}(a) (see below);

\smallskip

\item Point (a) of \thmref{t:descr of restr glob} will imply point (b) by a duality argument (see below);

\smallskip

\item Point (b) of \thmref{t:descr of restr glob} will imply the assertion of \thmref{t:descr of restr} about !-restrictions (see below).

\end{enumerate}

\end{rem} 

\sssec{}

We first prove:

\begin{thm} \label{t:descr of restr glob} \hfill 

\smallskip

\noindent{\em(a)} $(\bi^\lambda_{\on{glob}})^*(\IC^\semiinf_{\on{glob}})\simeq 
(p^\lambda_{\on{glob}})^!(\on{Fact}^{\on{alg}}(\CO(\cN))_{X^\lambda})[-d-\langle \lambda,2\check\rho\rangle]$.

\smallskip

\noindent{\em(b)} $(\bi^\lambda_{\on{glob}})^!(\IC^\semiinf_{\on{glob}})\simeq
(p^\lambda_{\on{glob}})^!(\on{Fact}^{\on{coalg}}(U(\cn^-))_{X^\lambda})[-d-\langle \lambda,2\check\rho\rangle]$. 

\end{thm} 

\begin{proof}

We first prove point (a). Let $\CF^\lambda\in \Shv(X^\lambda)$ be such that
\begin{equation} \label{e:eq0}
(\bi^\lambda_{\on{glob}})^*(\IC^\semiinf_{\on{glob}})\simeq (p^\lambda_{\on{glob}})^!(\CF^\lambda)[-d-\langle \lambda,2\check\rho\rangle].
\end{equation} 

We will show that 
$$\CF^\lambda\simeq \on{Fact}^{\on{alg}}(\CO(\cN))_{X^\lambda}.$$

\medskip

Applying $(\pi^\lambda_\Ran)^!$ to both sides in \eqref{e:eq0}, we obtain:

\begin{multline} \label{e:eq2}
(\pi^\lambda_\Ran)^!\circ (\bi^\lambda_{\on{glob}})^*(\IC^\semiinf_{\on{glob}})\simeq \\
\simeq (\pi^\lambda_\Ran)^!\circ (p^\lambda_{\on{glob}})^!(\CF^\lambda)[-d-\langle \lambda,2\check\rho\rangle]\simeq 
(p^\lambda_\Ran)^!\circ (\on{pr}^\lambda_\Ran)^!(\CF^\lambda)[-d-\langle \lambda,2\check\rho\rangle]
\end{multline} 

By \corref{c:restr isom} and \thmref{t:IC loc glob}, we have:
$$(\pi^\lambda_\Ran)^!\circ (\bi^\lambda_{\on{glob}})^*(\IC^\semiinf_{\on{glob}})\simeq
(\bi^\lambda)^*(\IC^\semiinf_\Ran)[-d].$$

Further, by Remark \ref{r:*-fibers proved}, we have: 
$$(\bi^\lambda)^*(\IC^\semiinf_\Ran)\simeq (p^\lambda_\Ran)^!\circ 
(\on{pr}^\lambda_\Ran)^!(\on{Fact}^{\on{alg}}(\CO(\cN))_{X^\lambda})[-\langle \lambda,2\check\rho\rangle].$$

Combining with \eqref{e:eq2}, we obtain
$$(p^\lambda_\Ran)^!\circ (\on{pr}^\lambda_\Ran)^!(\CF^\lambda)\simeq 
(p^\lambda_\Ran)^!\circ 
(\on{pr}^\lambda_\Ran)^!(\on{Fact}^{\on{alg}}(\CO(\cN))_{X^\lambda}).$$

Since the functor $(p^\lambda_\Ran)^!\circ (\on{pr}^\lambda_\Ran)^!$ is fully faithful, 
we obtain the desired
$$\CF^\lambda \simeq \on{Fact}^{\on{alg}}(\CO(\cN))_{X^\lambda},$$
proving point (a). 

\medskip

Since $\IC^\semiinf_{\on{glob}}$ is Verdier self-dual, and using the fact that
$$\BD(\on{Fact}^{\on{coalg}}(U(\cn^-))_{X^\lambda})\simeq \on{Fact}^{\on{alg}}(\CO(\cN))_{X^\lambda},$$
from the isomorphism of point (a), we obtain
\begin{multline*} 
(\bi^\lambda_{\on{glob}})^!(\IC^\semiinf_{\on{glob}})\simeq 
(p^\lambda_{\on{glob}})^*(\on{Fact}^{\on{coalg}}(U(\cn^-))_{X^\lambda})[d+\langle \lambda,2\check\rho\rangle]\simeq \\
\simeq (p^\lambda_{\on{glob}})^!(\on{Fact}^{\on{coalg}}(U(\cn^-))_{X^\lambda})[-d-\langle \lambda,2\check\rho\rangle],
\end{multline*} 
the latter isomorphism because $p^\lambda_{\on{glob}}$ is smooth of relative dimension $d+\langle \lambda,2\check\rho\rangle$. 
This proves point (b). 

\end{proof}

\sssec{}

Let us now prove \thmref{t:descr of restr}. 

\begin{proof}

By Remark \ref{r:*-fibers proved}, it remains to prove the assertion regarding $(\bi^\lambda)^!(\IC^\semiinf_\Ran)$.

\medskip

Let $\CG^\lambda\in \Shv((X^\lambda \times \Ran)^{\subset})$ be such that
$$(\bi^\lambda)^!(\IC^\semiinf_\Ran)\simeq (p^\lambda_\Ran)^!(\CG^\lambda)[-\langle \lambda,2\check\rho\rangle].$$

Let us show that
$$\CG^\lambda\simeq (\on{pr}^\lambda_\Ran)^!(\on{Fact}^{\on{coalg}}(U(\cn^-))_{X^\lambda}).$$

\medskip

Indeed, by \thmref{t:IC loc glob} and \thmref{t:descr of restr glob}(b), we have:
\begin{multline*}
(p^\lambda_\Ran)^!(\CG^\lambda)[-\langle \lambda,2\check\rho\rangle]=(\bi^\lambda)^!(\IC^\semiinf_\Ran)\simeq
(\bi^\lambda)^!\circ (\pi_\Ran)^!(\IC^\semiinf_{\on{glob}})[d]\simeq \\
\simeq  (\pi^\lambda_\Ran)^! \circ (\bi^\lambda_{\on{glob}})^!(\IC^\semiinf_{\on{glob}})[d]\simeq
(\pi^\lambda_\Ran)^! \circ (p^\lambda_{\on{glob}})^!(\on{Fact}^{\on{coalg}}(U(\cn^-))_{X^\lambda})[-\langle \lambda,2\check\rho\rangle]\simeq \\
\simeq (p^\lambda_\Ran)^!\circ (\on{pr}^\lambda_\Ran)^!(\on{Fact}^{\on{coalg}}(U(\cn^-))_{X^\lambda})[-\langle \lambda,2\check\rho\rangle].
\end{multline*}

Since $(p^\lambda_\Ran)^!$ is fully faithful, this gives the desired isomorphism. 

\end{proof}

\section{Unital structure and factorization}  \label{s:unital}

The goal of this section is to explore an additional property of $\ICs$, which we will refer to as 
\emph{unitality}. It has to do with the following additional structure on $\Gr_{G,\Ran}$: 
one can ``throw in" more points in $\Ran$ without altering the $G$-bundle.  

\medskip

The unital property of $\ICs$ will allow us to construct on it a \emph{factorization structure}. 

\ssec{Unital structure on the affine Grassmannian}  \label{ss:unital on Gr}

In this subsection we introduce the geometric structure on $\Gr_{G,\Ran}$ that would allow us to 
talk about \emph{unitality}. 

\sssec{}

Let $(\Ran\times \Ran)^{\subset}$ be the following sufunctor of $\Ran\times \Ran$: for an affine test-scheme $Y$, the set
$\Hom(Y,(\Ran\times \Ran)^{\subset})$ consists of those
$$\CI,\CI'\subset \Hom(Y,X)$$
for which 
\begin{equation} \label{e:Ran subset}
\CI\subseteq \CI' \subset \Hom(Y,X).
\end{equation} 

\medskip

The diagonal map 
$$\Delta_{\Ran}:\Ran\to \Ran\times \Ran$$
factors through a map $\Ran\to (\Ran\times \Ran)^{\subset}$, which, by a slight abuse of notation, we denote by the
same symbol $\Delta_{\Ran}$. 

\medskip

There are two obvious projections
$$\sfob_{\on{small}},\sfob_{\on{big}}:(\Ran\times \Ran)^{\subset}\to \Ran$$
that send a point \eqref{e:Ran subset} to 
$$\CI\subset \Hom(Y,X) \text{ and } \CI'\subset \Hom(Y,X),$$
respectively.

\medskip

We have
$$\sfob_{\on{small}}\circ \Delta_{\Ran}=\on{id} \text{ and } \sfob_{\on{big}}\circ \Delta_{\Ran}=\on{id}.$$

\medskip

For future use we note:

\begin{lem}  \label{l:q UHC}
The map $\sfob_{\on{small}}$ is universally homologically contractible.
\end{lem} 

\begin{rem}  \label{r:UHC}
One proof of \lemref{l:q UHC} can be obtained by mimicking the argument in \secref{sss:proof of pr contr}.
We will now give a different argument, which does \emph{not} use the properness of $X$
(we note that the argument below can also be used to give an alternative proof of \lemref{l:pr contr},
see \propref{p:unital on config} below). 
\end{rem} 

\begin{proof}

Let $Y$ be an affine scheme and let is be given a $Y$-point $\CJ\subset \Hom(Y,X)$ of $\Ran$.
We need to show that the pullback functor
$$\Shv(Y)\to \Shv(Y\underset{\Ran}\times (\Ran\times \Ran)^{\subset})$$
is fully faithful, where the map $(\Ran\times \Ran)^{\subset}\to \Ran$ is $\sfob_{\on{small}}$.

\medskip

To show this, it suffices to show that the map $\sfob_{\on{small}}$ can be obtained as a \emph{retract} 
of a map which is universally homologically contractible. We let this other map be the projection
$$\Ran\times \Ran \to \Ran, \quad (\CI_1,\CI_2)\mapsto \CI_1.$$
It is universally homologically contractible because the Ran space is homologically contractible
(i.e., universally homologically contractible over $\on{pt}$).

\medskip

We realize $(\Ran\times \Ran)^{\subset}\to \Ran$ as a retract of $\Ran\times \Ran\to \Ran$ as follows. 
The map $$(\Ran\times \Ran)^{\subset}\to \Ran\times \Ran$$ sends 
$$(\CI\subset \CI')\mapsto (\CI,\CI').$$
The retraction $\Ran\times \Ran\to (\Ran\times \Ran)^{\subset}$ sends
$$(\CI_1,\CI_2)\mapsto (\CI_1\subseteq \CI_1\cup \CI_2).$$

\end{proof} 

\sssec{}

Consider the fiber product
$$\Gr_{G,(\Ran\times \Ran)^{\subset}}:=\Gr_{G,\Ran}\underset{\Ran}\times (\Ran\times \Ran)^{\subset},$$
where the map $(\Ran\times \Ran)^{\subset}\to \Ran$ is $\sfob_{\on{small}}$. By a slight abuse of notation, 
we will denote by the same symbol $\sfob_{\on{small}}$ the projection 
$$\Gr_{G,(\Ran\times \Ran)^{\subset}}\to \Gr_{G,\Ran}.$$

\medskip

Note, however, that we have another map, denoted
$$\sfob_{\on{big}}: \Gr_{G,(\Ran\times \Ran)^{\subset}}\to  \Gr_{G,\Ran}$$
that makes the following diagram commute:
\begin{equation}   \label{e:Gr Ran Ran}
\CD
\Gr_{G,(\Ran\times \Ran)^{\subset}}   @>{\sfob_{\on{big}}}>>  \Gr_{G,\Ran}  \\
@VVV   @VVV   \\
(\Ran\times \Ran)^{\subset} @>{\sfob_{\on{big}}}>>  \Ran.
\endCD
\end{equation} 

Namely, it sends a quadruple $(\CI\subseteq \CI',\CP_G,\alpha)$ to $(\CI',\CP_G,\alpha')$, where 
$\alpha$ is a trivialization of $\CP_G$ on the complement of $\Gamma_\CI$ and $\alpha'$ is the restriction of $\alpha$ to the
complement of $\Gamma_{\CI'}$. 

\medskip

\noindent{\it Warning:} Note, however, that the diagram \eqref{e:Gr Ran Ran} is \emph{not} Cartesian.

\medskip

Denote by $\Delta_{\Ran}$ the natural map
$$\Gr_{G,\Ran}\to \Gr_{G,(\Ran\times \Ran)^{\subset} }.$$

We have
$$\sfob_{\on{small}}\circ \Delta_{\Ran}\simeq \on{id} \text{ and }
\sfob_{\on{big}}\circ \Delta_{\Ran}\simeq \on{id}.$$

\sssec{}

We shall say that an object 
$$\CF\in \Shv(\Gr_{G,\Ran})$$
is \emph{unital} if \emph{there exists} an isomorphism
$$\sfob_{\on{small}}^!(\CF)\simeq \sfob_{\on{big}}^!(\CF)$$
for which the composition
$$\CF\simeq \Delta_{\Ran}^! \circ \sfob_{\on{small}}^!(\CF)\simeq 
\Delta_{\Ran}^! \circ \sfob_{\on{big}}^!(\CF)\simeq \CF$$
is the identity map. 

\medskip

Note that it follows from \lemref{l:q UHC} that if such an isomorphism exists, then it is
unique. 

\sssec{}

Let 
$$\Shv(\Gr_{G,\Ran})_{\on{unital}}\subset \Shv(\Gr_{G,\Ran})$$
be the full subcategory formed by unital objects. 

\medskip

From \lemref{l:q UHC} we obtain:

\begin{cor} The subcategory $\Shv(\Gr_{G,\Ran})_{\on{unital}}\subset \Shv(\Gr_{G,\Ran})$ is closed
under colimits.
\end{cor}

In particular, we obtain that $\Shv(\Gr_{G,\Ran})_{\on{unital}}$ is a (cocomplete) DG subcategory 
of $\Shv(\Gr_{G,\Ran})$.

\sssec{}

Set:
$$\SI_{\Ran,\on{unital}}:=\SI_{\Ran}\cap \Shv(\Gr_{G,\Ran})_{\on{unital}}\subset \Shv(\Gr_{G,\Ran}).$$

\medskip

Our next goal is to characterize $\SI_{\Ran,\on{unital}}$ more explicitly as a full subcategory of $\SI_{\Ran}$.

\ssec{Unital structure on the strata}  \label{ss:strata unital}

In this subsection we will extend the discussion of \secref{ss:unital on Gr} from $\Gr_\Ran$ to the prestacks 
$S^\lambda_\Ran$ and $(X^\lambda\times \Ran)^{\subset}$.

\medskip

We will see that the unital subcategory of $\Shv((X^\lambda\times \Ran)^{\subset})$ is actually equivalent to 
$\Shv(X^\lambda)$. 

\sssec{}

For a fixed $\lambda$, consider the functors 
$$S^\lambda_\Ran \hookrightarrow \ol{S}^\lambda_\Ran\to \Gr_{G,\Ran},$$
and consider the corresponding diagram of prestacks 
$$
\CD
S^\lambda_{(\Ran\times \Ran)^{\subset}}  @>{\bj^\lambda}>>  \ol{S}^\lambda_{(\Ran\times \Ran)^{\subset}}   
@>{\ol\bi^\lambda}>>  
\Gr_{G,(\Ran\times \Ran)^{\subset} }  \\
@V{\sfob_{\on{big}}}VV   @V{\sfob_{\on{big}}}VV  @VV{\sfob_{\on{big}}}V  \\
S^\lambda_\Ran @>{\bj^\lambda}>>  \ol{S}^\lambda_\Ran
@>{\ol\bi^\lambda}>>  
\Gr_{G,\Ran}. 
\endCD
$$

The discussion in \secref{ss:unital on Gr} applies to the present situation as well. In particular, we obtain the \emph{full} subcategories
$$\Shv(\ol{S}^\lambda_\Ran)_{\on{unital}}\subset \Shv(\ol{S}^\lambda_\Ran) \text{ and }
\Shv(S^\lambda_\Ran)_{\on{unital}}\subset \Shv(S^\lambda_\Ran)$$
as well as 
$$\SI^{\leq \lambda}_{\Ran,\on{unital}}\subset \SI^{\leq \lambda} \text{ and }
\SI^{=\lambda}_{\Ran,\on{unital}}\subset \SI^{=\lambda}.$$

\medskip

It is clear that the functors $(\ol\bi^\lambda)^!$, $(\bj^\lambda)^!$ and $(\ol\bi^\lambda)_*$, $(\bj^\lambda)_*$
send the corresponding unital subcategories to one another. In particular, from \lemref{l:q UHC}, we obtain: 

\begin{cor}   
%
%
An object $\CF\in \SI^{\leq 0}_\Ran$ belongs to $\SI^{\leq 0}_{\Ran,\on{unital}}$ if and only if
$(\bi^\lambda)^!(\CF)$ belongs to $\SI^{=\lambda}_{\Ran,\on{unital}}$ for all $\lambda$. 
%
%
\end{cor} 

Finally, from \eqref{e:Braden semiinf Ran} one obtains: 

\begin{cor} \label{c:well defined Ran unital}  \hfill

\smallskip

\noindent{\em(a)} 
The functor $(\bi^\lambda)^*:\SI^{\leq 0}_\Ran\to \SI^{=\lambda}_\Ran$ sends
$\SI^{\leq 0}_{\Ran,\on{unital}}$ to $\SI^{=\lambda}_{\Ran,\on{unital}}$. 

\smallskip

\noindent{\em(b)} 
The functor $(\bi^\lambda)_!:\SI^{=\lambda}_\Ran\to \SI^{\leq 0}_\Ran$ sends
$\SI^{=\lambda}_{\Ran,\on{unital}}$ to $\SI^{\leq 0}_{\Ran,\on{unital}}$. 

\end{cor} 

\sssec{}

For a fixed $\lambda$ consider the prestack 
$$(X^\lambda\times \Ran\times \Ran)^{\subset}:=
(X^\lambda\times \Ran)^{\subset}\underset{\Ran}\times (\Ran\times \Ran)^{\subset},$$
where the map $(\Ran\times \Ran)^{\subset}\to \Ran$ is $\sfob_{\on{small}}$. 
By a slight abuse of notation, 
we will denote by the same symbol $\sfob_{\on{small}}$ the projection 
$$(X^\lambda \times \Ran\times \Ran)^{\subset}\to (X^\lambda \times \Ran)^{\subset}, \quad
(D,\CJ,\CJ')\mapsto (D,\CJ).$$

\medskip

Let us denote by $\sfob_{\on{big}}$ the map
$$(X^\lambda\times \Ran\times \Ran)^{\subset}\to (X^\lambda\times \Ran)^{\subset}, \quad
(D,\CJ,\CJ')\mapsto (D,\CJ').$$

\medskip

Using this map, we define a full subcategory
$$\Shv((X^\lambda\times \Ran)^{\subset})_{\on{unital}}\subset \Shv((X^\lambda\times \Ran)^{\subset}).$$

From \propref{p:on stratum}, we obtain:

\begin{cor} \label{c:on stratum unital}
The equivalence 
$$(p^\lambda_\Ran)^!:\Shv((X^\lambda\times \Ran)^{\subset})\to \SI_\Ran^{=\lambda}$$
restricts to an equivalence
$$\Shv((X^\lambda\times \Ran)^{\subset})_{\on{unital}}\to \SI_{\Ran,\on{unital}}^{=\lambda}.$$
\end{cor} 

\sssec{}

We now claim:

\begin{prop}  \label{p:unital on config}
The pullback functor
$$(\on{pr}^\lambda_\Ran)^!:\Shv(X^\lambda)\to \Shv((X^\lambda\times \Ran)^{\subset})$$
defines an equivalence
$$\Shv(X^\lambda)\overset{\sim}\to \Shv((X^\lambda\times \Ran)^{\subset})_{\on{unital}}.$$
\end{prop}

\begin{proof}
The fact that the functor $(\on{pr}^\lambda_\Ran)^!$ sends $\Shv(X^\lambda)$ to $\Shv((X^\lambda\times \Ran)^{\subset})_{\on{unital}}$
is immediate from the definition\footnote{Note also that the fully faithfulness of $(\on{pr}^\lambda_\Ran)^!$ has been already stated in 
\lemref{l:pr contr}; however, the argument given below will give an alternative proof of this fact.}. 

\medskip

Choose a finite set $I$ so that we have a surjective symmetrization map $\on{sym}^{I\to \lambda}:X^I\to X^\lambda$. Since the map
$\on{sym}^{I\to \lambda}$ is finite and surjective, it satisfies descent for $\Shv(-)$. So it is sufficient to prove the assertion of the
proposition when the original map 
$$\on{pr}^\lambda_\Ran: (X^\lambda\times \Ran)^{\subset}\to X^\lambda$$
is base-changed by the \v{C}ech nerve of the map $X^I\to X^\lambda$.

\medskip

We will prove that the pullback functor 
$$(\on{pr}^I_\Ran)^!:\Shv(X^I) \to \Shv((X^I\times \Ran)^{\subset})$$
defines an equivalence onto $\Shv((X^I\times \Ran)^{\subset})_{\on{unital}}$. 
I.e., we will prove the assertion for the $0$-simplicies of the \v{C}ech nerve; for higher simplices the proof is the same. 

\medskip

Note that the map 
$$\on{pr}^I_\Ran: (X^I\times \Ran)^{\subset}\to X^I$$
admits a section, denoted $r^I$. Namely, for an affine test-scheme $Y$ and a $Y$-point of $X^I$, which is a map 
$I\to \Hom(Y,X)$, we assign its image, denoted $\CI$ in $\Hom(Y,X)$. 

\medskip

Pullback with respect to $r^I$ defines a functor $\Shv((X^I\times \Ran)^{\subset})\to \Shv(X^I)$. We claim that the restriction 
of $(r^I)^!$ to $\Shv((X^I\times \Ran)^{\subset})_{\on{unital}}$ is a functor inverse to $(\on{pr}^I_\Ran)^!$. 

\medskip

Indeed, the fact that $(r^I)^!\circ (\on{pr}^I_\Ran)^!\simeq \on{Id}$ is obvious. To construct an isomorphism 
$$(\on{pr}^I_\Ran)^!\circ (r^I)^!|_{\Shv((X^I\times \Ran)^{\subset})_{\on{unital}}}\simeq \on{Id},$$
we note that there exist canonically defined maps 
$$''\!r^I;{}'\!r^I:(X^I\times \Ran)^{\subset}\to (X^I\times \Ran\times \Ran)^{\subset}$$
such that 
$$\sfob_{\on{big}}\circ {}'\!r^I=\sfob_{\on{big}}\circ {}''\!r^I,$$
while 
$$\sfob_{\on{small}}\circ {}'\!r^I=\on{id} \text{ and } \sfob_{\on{small}}\circ {}''\!r^I=r^I\circ \on{pr}^I_\Ran.$$

The maps $'\!r^I;{}''\!r^I$ are given by sending a pair $(x,\CI')$ to 
$$(x,\CI',\CI\cup \CI') \text{ and } (x,\CI,\CI\cup \CI'),$$
respectively, 
where $x\in \Hom(Y,X^I)$, and $\CI$ denotes the image of the resulting map $I\to \Hom(Y,X)$.

\end{proof} 

\ssec{Local-to-global comparison, revisited}

Once we have defined the category $\SI^{\leq 0}_{\Ran,\on{unital}}$, we can sharpen the assertion of
\thmref{t:contr}, by directly comparing the global semi-infinite category and the unital Ran version of the
local one. 

\sssec{}

We claim:

\begin{thm} \label{t:local-to-global, unital}  
The pullback functor
$$(\pi_\Ran)^!:\SI^{\leq 0}_{\on{glob}}\to \SI^{\leq 0}_{\Ran}$$
defines an equivalence onto $\SI^{\leq 0}_{\Ran,\on{unital}}$.
\end{thm} 

The rest of this subsection is devote to the proof of this theorem.

\sssec{}

First off, it is clear that the essential image of the functor 
$$(\pi_\Ran)^!:\Shv(\BunNb)\to \Shv(\ol{S}^{\leq 0}_{\Ran})$$
belongs to the full subcategory 
$$\Shv(\ol{S}^{\leq 0}_{\Ran})_{\on{unital}}\subset \Shv(\ol{S}^{\leq 0}_{\Ran}).$$

Indeed, this follows from the fact that the following diagram commutes:
$$
\CD
\Gr_{G,(\Ran\times \Ran)^{\subset}}   @>{\sfob_{\on{small}}}>>  \Gr_{G,\Ran}  \\
@V{\sfob_{\on{big}}}VV   @VV{\pi_\Ran}V   \\
\Gr_{G,\Ran}   @>{\pi_\Ran}>> \BunNb.
\endCD
$$

\sssec{}

Second, the fact that the functor in question is fully faithful follows from \thmref{t:contr}.

\medskip

Thus, it remains to show that the functor
$$(\pi_\Ran)^!:\SI^{\leq 0}_{\on{glob}}\to \SI^{\leq 0}_{\Ran,\on{unital}}$$
is essentially surjective.

\medskip

Taking into account \corref{c:ext isom Ran}, it suffices to show that the functor
$$(\pi^\lambda_\Ran)^!:\SI^{=\lambda}_{\on{glob}}\to \SI^{=\lambda}_{\Ran}$$
defines an equivalence onto $\SI^{=\lambda}_{\Ran,\on{unital}}\subset \SI^{=\lambda}_{\Ran}$. 

\medskip

However, this follows from \corref{c:on stratum unital} and \propref{p:unital on config}
using the commutative diagram
$$
\CD
S^\lambda_{\Ran} @>{p^\lambda_\Ran}>>   (X^\lambda\times \Ran)^{\subset}   \\
@V{\pi_\Ran^\lambda}VV   @VV{\on{pr}^\lambda_\Ran}V   \\
\BunNb^{=\lambda}   @>>>   X^\lambda. 
\endCD
$$

\ssec{The t-structure on the unital category}  \label{ss:t on unital}

In this subsection we will show that the t-structure on $\SI^{\leq 0}_\Ran$ restricts to a t-structure on 
$\SI^{\leq 0}_{\Ran,\on{unital}}$.

\sssec{}

We define a t-structure on $\SI^{=\lambda}_{\Ran,\on{unital}}$
by transferring the (perverse) t-structure on $\Shv(X^\lambda)$ via the 
equivalences
$$\Shv(X^\lambda) \overset{(\on{pr}_\Ran^\lambda)^!}\longrightarrow 
\Shv((X^\lambda\times \Ran)^{\subset})_{\on{unital}} 
\overset{(p^\lambda_\Ran)^!}\longrightarrow \SI_{\Ran,\on{unital}}^{=\lambda},$$
and applying the shift $[\langle \lambda,2\check\rho\rangle]$. 

\sssec{}

We define a t-structure on $\SI^{\leq 0}_{\Ran,\on{unital}}$ by declaring that an object $\CF$ is \emph{coconnective} 
if $$(\bi^\lambda)^!(\CF)\in \SI^{=\lambda}_{\Ran,\on{unital}}$$ is \emph{coconnective} for each $\lambda$. 

\medskip

As in \lemref{l:char of conn} one show that an object $\CF\in \SI^{\leq 0}_{\Ran,\on{unital}}$ is \emph{connective} 
if and only if $$(\bi^\lambda)^*(\CF)\in \SI^{=\lambda}_{\Ran,\on{unital}}$$ is \emph{connective} for each $\lambda$. 

\medskip

From here, we obtain:

\begin{cor}  \label{c:unital exact}
The inclusion $\SI^{\leq 0}_{\Ran,\on{unital}}\hookrightarrow \SI^{\leq 0}_{\Ran}$
is compatible with t-structures (i.e., is t-exact).
\end{cor} 

\sssec{}

We now claim:

\begin{prop}  \label{p:IC unital}
The object $\ICs\in (\SI^{\leq 0}_{\Ran})^\heartsuit$ belongs to $(\SI^{\leq 0}_{\Ran,\on{unital}})^\heartsuit$.
\end{prop} 

\begin{proof}
The assertion follows from \corref{c:unital exact} and the fact that both
$$(\bi^0)_!(\omega_{S^0_\Ran}) \text{ and } (\bi^0)_*(\omega_{S^0_\Ran})$$
belong to $\SI^{\leq 0}_{\Ran,\on{unital}}$.
\end{proof} 

\ssec{Comparison with IC on Zastava spaces, continued}

Recall the isomorphism 
$$(\on{pr}_\Ran^\lambda \times \on{id}_{\CZ^\lambda})^!(\IC_{Z^\lambda})\simeq 
(\fq')^!(\ICs)[\langle \lambda,2\check\rho\rangle]$$
established in \propref{p:compare with IC Zast}. 

\medskip

In this subsection we will sharpen this assertion by showing that it is 
\emph{uniquely} characterized by the property that its restriction to the
open substack 
$$(X^\lambda\times \Ran)^{\subset}\underset{X^\lambda}\times \oCZ{}^\lambda\subset
(X^\lambda\times \Ran)^{\subset}\underset{X^\lambda}\times \CZ^\lambda$$
is the tautological identification of both sides with the dualizing sheaf. 

\sssec{}

First off, we note that the recipe in \secref{ss:strata unital} allows to introduce a full subcategory
$$\Shv((X^\lambda\times \Ran)^{\subset}\underset{X^\lambda}\times \CZ^\lambda)_{\on{unital}}
\subset \Shv((X^\lambda\times \Ran)^{\subset}\underset{X^\lambda}\times \CZ^\lambda),$$
and the functor $(\fq')^!$ (see \eqref{e:q'}) sends 
$$\SI^{\leq 0}_{\on{unital}}\to \Shv((X^\lambda\times \Ran)^{\subset}\underset{X^\lambda}\times \CZ^\lambda)_{\on{unital}}.$$

\medskip

Moreover, an analog of \propref{p:unital on config} applies, and the functor $(\on{pr}_\Ran^\lambda \times \on{id}_{\CZ^\lambda})^!$
defines an equivalence
\begin{equation} \label{e:Zastava base changed}
\Shv(\CZ^\lambda)\overset{\sim}\to \Shv((X^\lambda\times \Ran)^{\subset}\underset{X^\lambda}\times \CZ^\lambda)_{\on{unital}}.
\end{equation}

\sssec{}

We define a t-structure on $\Shv((X^\lambda\times \Ran)^{\subset}\underset{X^\lambda}\times \CZ^\lambda)_{\on{unital}}$
by transfering the t-structure on $\Shv(\CZ^\lambda)$ via the equivalence of \eqref{e:Zastava base changed}. 

\medskip 

In particular, we obtain that the object
$$(\on{pr}_\Ran^\lambda \times \on{id}_{\CZ^\lambda})^!(\IC_{Z^\lambda})\in 
\Shv((X^\lambda\times \Ran)^{\subset}\underset{X^\lambda}\times \CZ^\lambda)_{\on{unital}}$$
lies in the heart of the t-structure, and is the minimal extension of 
$$(\on{pr}_\Ran^\lambda \times \on{id}_{\CZ^\lambda})^!(\IC_{Z^\lambda})|_{(X^\lambda\times \Ran)^{\subset}\underset{X^\lambda}\times \oCZ{}^\lambda}\in
\Shv((X^\lambda\times \Ran)^{\subset}\underset{X^\lambda}\times \oCZ{}^\lambda)_{\on{unital}}.$$

\sssec{}

Hence, we obtain:

\begin{cor}  \label{c:compare with IC Zast}
The isomorphism 
$$(\on{pr}_\Ran^\lambda \times \on{id}_{\CZ^\lambda})^!(\IC_{Z^\lambda})\simeq 
(\fq')^!(\ICs)[\langle \lambda,2\check\rho\rangle]$$
of \propref{p:compare with IC Zast} is uniquely characterized by the property that it extends the tautological isomorphism over
$(X^\lambda\times \Ran)^{\subset}\underset{X^\lambda}\times \oCZ{}^\lambda$. 
\end{cor} 

\ssec{Factorization structure on $\ICs$}

We now arrive to the key point of this section. We will show that \emph{unitalilty} allows 
one to construct the factorization structure on the semi-infinite cohomology sheaf $\ICs$.

\sssec{}

Recall that identification
\begin{equation} \label{e:factor Gr again}
(\Gr_{G,\Ran}\times \Gr_{G,\Ran})\underset{\Ran\times \Ran}\times (\Ran\times \Ran)_{\on{disj}}
\simeq \Gr_{G,\Ran}\underset{\Ran}\times (\Ran\times \Ran)_{\on{disj}}
\end{equation}
of \eqref{e:factor Gr}. 

\medskip

Our current goal is to show that, with respect to this identification, we have a canonical isomorphism
\begin{equation} \label{e:factor IC}
(\ICs\boxtimes \ICs)_{(\Gr_{G,\Ran}\times \Gr_{G,\Ran})\underset{\Ran\times \Ran}\times (\Ran\times \Ran)_{\on{disj}}}
\simeq \ICs_{\Gr_{G,\Ran}\underset{\Ran}\times (\Ran\times \Ran)_{\on{disj}}}.
\end{equation}

Note that we already know that such an isomorphism takes place, due to the identification
$$({}'\!\ICs\boxtimes {}'\!\ICs)_{(\Gr_{G,\Ran}\times \Gr_{G,\Ran})\underset{\Ran\times \Ran}\times (\Ran\times \Ran)_{\on{disj}}}
\simeq {}'\!\ICs_{\Gr_{G,\Ran}\underset{\Ran}\times (\Ran\times \Ran)_{\on{disj}}}$$
of \eqref{e:factor 'IC} and the isomorphism
\begin{equation} \label{e:IC vs IC'}
\ICs\simeq {}'\!\ICs.
\end{equation}

However, we would like to present a different construction of the isomorphism \eqref{e:factor IC}. It will be based on ``abstract"
t-structure considerations rather the identification of $\ICs$ with the 
(explicitly constructed) object $'\!\ICs$. 

\sssec{}

Let $\Ran^{\subset,\bullet}$ be the simplicial prestack whose prestack of $n$-simplices $\Ran^{\subset,n}$ attaches to an affine
test-scheme $Y$ the set of
$$\CI_0\subseteq...\subseteq \CI_n\subset \Hom(Y,X).$$

\medskip

Let 
$$(\Ran^{\subset,\bullet}\times \Ran^{\subset,\bullet})_{\on{disj}}\subset \Ran^{\subset,\bullet}\times \Ran^{\subset,\bullet}$$
be an open simplicial sub-prestack equal to
$$(\Ran^{\subset,\bullet}\times \Ran^{\subset,\bullet})\underset{\Ran\times \Ran}\times (\Ran\times \Ran)_{\on{disj}},$$
where the map
$$\Ran^{\subset,\bullet}\times \Ran^{\subset,\bullet}\to \Ran\times \Ran$$ sends 
$$(\CI'_0\subseteq...\subseteq \CI'_n,\CI''_0\subseteq...\subseteq \CI''_n)\mapsto (\CI'_n,\CI''_n).$$

\medskip

Consider the simplicial prestack
$$(\Gr_{G,\Ran}\times \Gr_{G,\Ran})\underset{\Ran\times \Ran}\times (\Ran^{\subset,\bullet}\times \Ran^{\subset,\bullet})_{\on{disj}},$$
where the map
$$(\Ran^{\subset,\bullet}\times \Ran^{\subset,\bullet})_{\on{disj}}\to (\Ran\times \Ran)$$ sends 
\begin{equation} \label{e:small projection}
(\CI'_0\subseteq...\subseteq \CI'_n,\CI''_0\subseteq...\subseteq \CI''_n)\mapsto (\CI'_0,\CI''_0).
\end{equation}

Note also that the identification \eqref{e:factor Gr again} extends to an identification of simplicial prestacks
\begin{equation} \label{e:factor Gr simplicial}
(\Gr_{G,\Ran}\times \Gr_{G,\Ran})\underset{\Ran\times \Ran}\times (\Ran^{\subset,\bullet}\times \Ran^{\subset,\bullet})_{\on{disj}}\simeq
\Gr_{G,\Ran}\underset{\Ran}\times (\Ran^{\subset,\bullet}\times \Ran^{\subset,\bullet})_{\on{disj}},
\end{equation}
where the map
$$(\Ran^{\subset,\bullet}\times \Ran^{\subset,\bullet})_{\on{disj}}\to (\Ran\times \Ran)$$  
is again \eqref{e:small projection}, and the map
$$(\Ran^{\subset,\bullet}\times \Ran^{\subset,\bullet})_{\on{disj}}\to \Ran$$
is
$$(\CI'_0\subseteq...\subseteq \CI'_n,\CI''_0\subseteq...\subseteq \CI''_n)\mapsto (\CI'_0\cup\CI''_0).$$

\medskip

We define 
\begin{multline*}
\Shv((\Gr_{G,\Ran}\times \Gr_{G,\Ran})\underset{\Ran\times \Ran}\times (\Ran\times \Ran)_{\on{disj}})_{\on{unital}}:=\\
=\on{Tot}\left(\Shv((\Gr_{G,\Ran}\times \Gr_{G,\Ran})\underset{\Ran\times \Ran}\times (\Ran^{\subset,\bullet}\times \Ran^{\subset,\bullet})_{\on{disj}})\right).
\end{multline*} 

\noindent{\it Warning.} Unlike the case of the functor $\Shv(\Gr_{G,\Ran})_{\on{unital}}\to \Shv(\Gr_{G,\Ran})$, 
it is \emph{no longer} true that the functor of restriction to 0-simplices 
\begin{multline*}
\Shv((\Gr_{G,\Ran}\times \Gr_{G,\Ran})\underset{\Ran\times \Ran}\times (\Ran\times \Ran)_{\on{disj}})_{\on{unital}} \to \\
\to \Shv((\Gr_{G,\Ran}\times \Gr_{G,\Ran})\underset{\Ran\times \Ran}\times (\Ran\times \Ran)_{\on{disj}})
\end{multline*} 
is fully faithful. 

\sssec{}  

Proceeding as in \secref{ss:SI}, we define a full subcategory
$$\SI_{(\Ran\times \Ran)_{\on{disj}}}\subset \Shv((\Gr_{G,\Ran}\times \Gr_{G,\Ran})\underset{\Ran\times \Ran}\times (\Ran\times \Ran)_{\on{disj}})$$
and a full subcategory
$$\SI_{(\Ran\times \Ran)_{\on{disj}},\on{unital}}\subset 
\Shv((\Gr_{G,\Ran}\times \Gr_{G,\Ran})\underset{\Ran\times \Ran}\times (\Ran\times \Ran)_{\on{disj}})_{\on{unital}}.$$

\sssec{}  \label{sss:factor unital}

It is clear that if $\CF_1,\CF_2$ are objects in $\SI_{\Ran,\on{unital}}$, then
$$(\CF_1\boxtimes \CF_2)|_{(\Gr_{G,\Ran}\times \Gr_{G,\Ran})\underset{\Ran\times \Ran}\times (\Ran\times \Ran)_{\on{disj}}}\in \SI_{(\Ran\times \Ran)_{\on{disj}}}$$
naturally upgrades to an object of $\SI_{(\Ran\times \Ran)_{\on{disj}},\on{unital}}$. 

\medskip

Similarly, it is clear that if $\CF$ is an object of $\SI_{\Ran,\on{unital}}$, then
$$\CF|_{\Gr_{G,\Ran}\underset{\Ran}\times (\Ran^{\subset}\times \Ran^{\subset})_{\on{disj}}}\in \SI_{(\Ran\times \Ran)_{\on{disj}}}$$
naturally upgrades to an object of $\SI_{(\Ran\times \Ran)_{\on{disj}},\on{unital}}$. 

\medskip

In particular, we obtain that both sides in \eqref{e:factor IC} are naturally objects of $\SI_{(\Ran\times \Ran)_{\on{disj}},\on{unital}}$.

\sssec{}  

Similar definitions apply to 
$\Gr_{G,\Ran}\times \Gr_{G,\Ran}$ replaced by
$\ol{S}^0_\Ran\times \ol{S}^0_\Ran$
and also by  
$$S^\lambda_\Ran\times S^\mu_\Ran$$
for a pair of elements $\lambda,\mu\in \Lambda$. Denote the resulting categories by 
$$\SI^{\leq 0}_{(\Ran\times \Ran)_{\on{disj}},\on{unital}} \text{ and }\SI^{=\lambda,\mu}_{(\Ran\times \Ran)_{\on{disj}},\on{unital}}$$
respectively. 

\medskip

As in \corref{c:on stratum unital}, we have:

\begin{cor}  \label{c:disj strata upstairs unital}
The pullback functor along the map 
$p^{\lambda,\mu}_{(\Ran\times \Ran)_{\on{disj}}}$
$$(S^\lambda_\Ran\times S^\mu_\Ran)\underset{\Ran\times \Ran}\times (\Ran\times \Ran)_{\on{disj}}\to
((X^\lambda\times \Ran)^{\subset}\times (X^\mu\times \Ran)^{\subset}) \underset{\Ran\times \Ran}\times (\Ran\times \Ran)_{\on{disj}}$$
defines equivalences
$$\Shv(((X^\lambda\times \Ran)^{\subset}\times (X^\mu\times \Ran)^{\subset}) \underset{\Ran\times \Ran}\times (\Ran\times \Ran)_{\on{disj}})\to
\SI^{=\lambda,\mu}_{(\Ran\times \Ran)_{\on{disj}}}$$
and
$$\Shv(((X^\lambda\times \Ran)^{\subset}\times (X^\mu\times \Ran)^{\subset}) \underset{\Ran\times \Ran}\times (\Ran\times \Ran)_{\on{disj}})_{\on{unital}}\to 
\SI^{=\lambda,\mu}_{(\Ran\times \Ran)_{\on{disj}},\on{unital}}.$$
\end{cor} 

In addition, by repeating the argument of \propref{p:unital on config}, one shows: 

\begin{prop} \label{p:disj strata unital}
The pullback functor along the map $\on{pr}_{(\Ran\times \Ran)_{\on{disj}}}^{\lambda,\mu}$
$$((X^\lambda\times \Ran)^{\subset}\times (X^\mu\times \Ran)^{\subset}) \underset{\Ran\times \Ran}\times (\Ran\times \Ran)_{\on{disj}}\to
(X^\lambda\times X^\mu)_{\on{disj}}$$
defines an equivalence
$$\Shv((X^\lambda\times X^\mu)_{\on{disj}})\to 
\Shv(((X^\lambda\times \Ran)^{\subset}\times (X^\mu\times \Ran)^{\subset}) \underset{\Ran\times \Ran}\times (\Ran\times \Ran)_{\on{disj}})_{\on{unital}}.$$
\end{prop}

\sssec{}  \label{sss:t-structure fact}

Using \corref{c:disj strata upstairs unital} and \propref{p:disj strata unital}, 
proceeding as in \secref{ss:t on unital}, we define a t-structure on the categories $\SI^{=\lambda,\mu}_{(\Ran\times \Ran)_{\on{disj}},\on{unital}}$
and $\SI^{\leq 0}_{(\Ran\times \Ran)_{\on{disj}},\on{unital}}$. 

\medskip

It is clear that in the situation of \secref{sss:factor unital}, if $\CF_1,\CF_2\in \SI^{\leq 0}_{\Ran,\on{unital}}$ 
(resp., $\CF\in \SI^{\leq 0}_{\Ran,\on{unital}}$) are connective/coconnective, then so are
the corresponding objects 
$$(\CF_1\boxtimes \CF_2)|_{(\Gr_{G,\Ran}\times \Gr_{G,\Ran})\underset{\Ran\times \Ran}\times (\Ran\times \Ran)_{\on{disj}}}\in 
\SI^{\leq 0}_{(\Ran\times \Ran)_{\on{disj}},\on{unital}}$$
and 
$$\CF|_{\Gr_{G,\Ran}\underset{\Ran}\times (\Ran^{\subset}\times \Ran^{\subset})_{\on{disj}}}\in \SI^{\leq 0}_{(\Ran\times \Ran)_{\on{disj}},\on{unital}}.$$
are connective/coconnective.

\medskip

This implies that both sides in \eqref{e:factor IC} are minimal extensions of the object
$$\omega_{(S^0_\Ran\times S^0_\Ran)\underset{\Ran\times \Ran}\times (\Ran\times \Ran)_{\on{disj}}}\in
\SI^{=0,0}_{(\Ran\times \Ran)_{\on{disj}},\on{unital}}.$$

This implies the sought-for canonical isomorphism \eqref{e:factor IC}.

\ssec{Factorization and Zastava spaces}

In this subsection we will establish the compatibility of the factorization structure on $\ICs$ given by 
\eqref{e:factor IC} and the factorization property of the IC sheaf on Zastava spaces. 

\sssec{}

Recall again the Zastava spaces $\CZ^\lambda$.

\medskip

According to \cite[Prop. 2.4]{BFGM}, we have canonical isomorphisms
\begin{equation} \label{e:factor Zastava}
(\CZ^\lambda\times \CZ^\mu)\underset{X^\lambda\times X^\mu}\times 
(X^\lambda\times X^\mu)_{\on{disj}}\simeq \CZ^{\lambda+\mu}\underset{X^{\lambda+\mu}}\times (X^\lambda\times X^\mu)_{\on{disj}}.
\end{equation} 

Since the composite map 
$$(X^\lambda\times X^\mu)_{\on{disj}}\to X^\lambda\times X^\mu\to X^{\lambda+\mu}$$
is \'etale, we have a canonical isomorphism
\begin{equation} \label{e:factor IC Zastava}
(\IC_{\CZ^\lambda}\boxtimes \IC_{\CZ^\mu})|_{(\CZ^\lambda\times \CZ^\mu)\underset{X^\lambda\times X^\mu}\times 
(X^\lambda\times X^\mu)_{\on{disj}}}\simeq
\IC_{\CZ^{\lambda+\mu}}|_{\CZ^{\lambda+\mu}\underset{X^{\lambda+\mu}}\times (X^\lambda\times X^\mu)_{\on{disj}}}.
\end{equation} 

\sssec{}

Note that we have an identification

\begin{multline}   \label{e:factor base changed Zastava}
\left(((X^\lambda\times \Ran)^{\subset}\underset{X^\lambda}\times \CZ^\lambda) \times
((X^\mu\times \Ran)^{\subset}\underset{X^\mu}\times \CZ^\mu)\right) \underset{\Ran\times \Ran}\times (\Ran\times \Ran)_{\on{disj}} \simeq \\
\simeq \left((X^{\lambda+\mu}\times \Ran)^{\subset}\underset{X^{\lambda+\mu}}\times \CZ^{\lambda+\mu})\right)
\underset{(X^{\lambda+\mu}\times \Ran)^{\subset}}\times 
\left((X^\lambda\times \Ran)^{\subset}\times (X^\mu\times \Ran)^{\subset}\right)_{\on{disj}}.
\end{multline} 
where
\begin{multline*} 
\left((X^\lambda\times \Ran)^{\subset}\times (X^\mu\times \Ran)^{\subset})\right)_{\on{disj}}:= \\
=\left(((X^\lambda\times \Ran)^{\subset}\times (X^\mu\times \Ran)^{\subset})\underset{\Ran\times \Ran}\times (\Ran\times \Ran)_{\on{disj}}\right).
\end{multline*} 

Consider the maps
\begin{multline*} 
\left(((X^\lambda\times \Ran)^{\subset}\underset{X^\lambda}\times \CZ^\lambda) \times
((X^\mu\times \Ran)^{\subset}\underset{X^\mu}\times \CZ^\mu)\right) \underset{\Ran\times \Ran}\times (\Ran\times \Ran)_{\on{disj}} \to \\
\to (\ol{S}^0_{\Ran}\times  \ol{S}^0_{\Ran})\underset{\Ran\times \Ran}\times (\Ran\times \Ran)_{\on{disj}}.
\end{multline*} 
and
\begin{multline*}  
\left((X^{\lambda+\mu}\times \Ran)^{\subset}\underset{X^{\lambda+\mu}}\times \CZ^{\lambda+\mu})\right)
\underset{(X^{\lambda+\mu}\times \Ran)^{\subset}}\times 
\left((X^\lambda\times \Ran)^{\subset}\times (X^\mu\times \Ran)^{\subset}\right)_{\on{disj}} \to \\
\to \ol{S}^0_{\Ran} \underset{\Ran}\times (\Ran\times \Ran)_{\on{disj}}.
\end{multline*} 

They are compatible with respect to the identifications \eqref{e:factor base changed Zastava} and 
\begin{equation} \label{e:factor semiinf}
(\ol{S}^0_{\Ran}\times  \ol{S}^0_{\Ran})\underset{\Ran\times \Ran}\times (\Ran\times \Ran)_{\on{disj}}\simeq 
\ol{S}^0_{\Ran} \underset{\Ran}\times (\Ran\times \Ran)_{\on{disj}}.
\end{equation}

Hence, from \eqref{e:factor Zastava} we obtain an isomorphism:
\begin{multline} \label{e:factor IC Zast 1}
\ICs\boxtimes \ICs|_{\left(((X^\lambda\times \Ran)^{\subset}\underset{X^\lambda}\times \CZ^\lambda) \times
((X^\mu\times \Ran)^{\subset}\underset{X^\mu}\times \CZ^\mu)\right) \underset{\Ran\times \Ran}\times (\Ran\times \Ran)_{\on{disj}}}\simeq \\
\simeq \\
\simeq \ICs|_{\left((X^{\lambda+\mu}\times \Ran)^{\subset}\underset{X^{\lambda+\mu}}\times \CZ^{\lambda+\mu})\right)
\underset{(X^{\lambda+\mu}\times \Ran)^{\subset}}\times 
\left((X^\lambda\times \Ran)^{\subset}\times (X^\mu\times \Ran)^{\subset}\right)_{\on{disj}}}
\end{multline}

\sssec{}

Consider now the maps 

\begin{multline*} 
\left(((X^\lambda\times \Ran)^{\subset}\underset{X^\lambda}\times \CZ^\lambda) \times
((X^\mu\times \Ran)^{\subset}\underset{X^\mu}\times \CZ^\mu)\right) \underset{\Ran\times \Ran}\times (\Ran\times \Ran)_{\on{disj}} \to \\
\to (\CZ^\lambda\times \CZ^\mu) \underset{X^\lambda\times X^\mu}\times (X^\lambda\times X^\mu)_{\on{disj}}
\end{multline*} 
and
\begin{multline*} 
\left((X^{\lambda+\mu}\times \Ran)^{\subset}\underset{X^{\lambda+\mu}}\times \CZ^{\lambda+\mu})\right)
\underset{(X^{\lambda+\mu}\times \Ran)^{\subset}}\times 
\left((X^\lambda\times \Ran)^{\subset}\times (X^\mu\times \Ran)^{\subset}\right)_{\on{disj}} \to \\
\to \CZ^{\lambda+\mu}\underset{X^{\lambda+\mu}}\times (X^\lambda\times X^\mu)_{\on{disj}}.
\end{multline*} 

They are compatible with respect to the identifications \eqref{e:factor base changed Zastava} and \eqref{e:factor Zastava}. 
Hence, from \eqref{e:factor IC Zastava} we obtain the isomorphism
\begin{multline} \label{e:factor IC Zast 2}
\IC_{\CZ^\lambda}\boxtimes \IC_{\CZ^\mu}|_{\left(((X^\lambda\times \Ran)^{\subset}\underset{X^\lambda}\times \CZ^\lambda) \times
((X^\mu\times \Ran)^{\subset}\underset{X^\mu}\times \CZ^\mu)\right) \underset{\Ran\times \Ran}\times (\Ran\times \Ran)_{\on{disj}}}\simeq \\
\simeq 
\IC_{\CZ^{\lambda+\mu}}|_{\left((X^{\lambda+\mu}\times \Ran)^{\subset}\underset{X^{\lambda+\mu}}\times \CZ^{\lambda+\mu})\right)
\underset{(X^{\lambda+\mu}\times \Ran)^{\subset}}\times 
\left((X^\lambda\times \Ran)^{\subset}\times (X^\mu\times \Ran)^{\subset}\right)_{\on{disj}}}.
\end{multline}

\sssec{}

We claim:

\begin{prop}  \label{p:compat of factor}
The isomorphisms \eqref{e:factor IC Zast 1} and \eqref{e:factor IC Zast 2} are compatible with respect to the isomorphisms
$$(\on{pr}_\Ran^\lambda \times \on{id}_{\CZ^\lambda})^!(\IC_{Z^\lambda})\simeq 
(\fq')^!(\ICs)[\langle \lambda,2\check\rho\rangle],\,\, (\on{pr}_\Ran^\mu \times \on{id}_{\CZ^\mu})^!(\IC_{Z^\mu})\simeq 
(\fq')^!(\ICs)[\langle \mu,2\check\rho\rangle]$$
and
$$(\on{pr}_\Ran^{\lambda+\mu} \times \on{id}_{\CZ^{\lambda+\mu}})^!(\IC_{Z^{\lambda+\mu}})\simeq 
(\fq')^!(\ICs)[\langle \lambda+\mu,2\check\rho\rangle]$$
of \propref{p:compare with IC Zast}.
\end{prop} 

\begin{proof}

By mimicking the procedure in \secref{sss:factor unital}, we introduce the category
\begin{equation} \label{e:factor unital categ}
\Shv\left(\left(((X^\lambda\times \Ran)^{\subset}\underset{X^\lambda}\times \CZ^\lambda) \times
((X^\mu\times \Ran)^{\subset}\underset{X^\mu}\times \CZ^\mu)\right) \underset{\Ran\times \Ran}\times (\Ran\times \Ran)_{\on{disj}}\right)_{\on{unital}},
\end{equation} 
and we show that the object
\begin{multline*} 
\IC_{\CZ^\lambda}\boxtimes \IC_{\CZ^\mu}|_{\left(((X^\lambda\times \Ran)^{\subset}\underset{X^\lambda}\times \CZ^\lambda) \times
((X^\mu\times \Ran)^{\subset}\underset{X^\mu}\times \CZ^\mu)\right) \underset{\Ran\times \Ran}\times (\Ran\times \Ran)_{\on{disj}}}\in \\
\Shv\left(\left(((X^\lambda\times \Ran)^{\subset}\underset{X^\lambda}\times \CZ^\lambda) \times
((X^\mu\times \Ran)^{\subset}\underset{X^\mu}\times \CZ^\mu)\right) \underset{\Ran\times \Ran}\times (\Ran\times \Ran)_{\on{disj}}\right)
\end{multline*}
naturally upgrades to an object of \eqref{e:factor unital categ}.

\medskip

Furthermore, by mimicking the procedure in \secref{sss:t-structure fact}, we introduce a t-structure on \eqref{e:factor unital categ} and we show
that the above object 
\begin{multline*} 
\IC_{\CZ^\lambda}\boxtimes \IC_{\CZ^\mu}|_{\left(((X^\lambda\times \Ran)^{\subset}\underset{X^\lambda}\times \CZ^\lambda) \times
((X^\mu\times \Ran)^{\subset}\underset{X^\mu}\times \CZ^\mu)\right) \underset{\Ran\times \Ran}\times (\Ran\times \Ran)_{\on{disj}}}\in \\
\Shv\left(\left(((X^\lambda\times \Ran)^{\subset}\underset{X^\lambda}\times \CZ^\lambda) \times
((X^\mu\times \Ran)^{\subset}\underset{X^\mu}\times \CZ^\mu)\right) \underset{\Ran\times \Ran}\times (\Ran\times \Ran)_{\on{disj}}\right)_{\on{unital}}
\end{multline*}
is the minimal extension of its restriction to
\begin{equation} \label{e:factor Zastava open}
\left(((X^\lambda\times \Ran)^{\subset}\underset{X^\lambda}\times \oCZ{}^\lambda) \times
((X^\mu\times \Ran)^{\subset}\underset{X^\mu}\times \oCZ{}^\mu)\right) \underset{\Ran\times \Ran}\times (\Ran\times \Ran)_{\on{disj}}.
\end{equation} 

Now the compatibility stated in \secref{p:compat of factor} follows from the fact that it does so after restriction to \eqref{e:factor Zastava open}. 

\end{proof}

\section{The Hecke property of the semi-infinite IC sheaf} \label{s:H}

The goal of this section is to show that the object $\ICs_\Ran$ that we have constructed satisfies 
the (appropriately formulated) Hecke eigen-property. 

\ssec{Pointwise Hecke property}

\sssec{}

Consider the category $\Shv(\fL^+(T)_\Ran \backslash \Gr_{G,\Ran})$, i.e.,
we impose the structure of equivariance 
with respect to group-scheme of arcs into $T$ over the base prestack $\Ran$. 

\medskip

The action of $\fL(T)_\Ran$ on $\Gr_{G,\Ran}$ by left multiplication defines an action of 
$\Sph_{T,\Ran}$ on $\Shv(\fL^+(T)_\Ran\backslash \Gr_{G,\Ran})$.




\medskip

We consider $\Shv(\fL^+(T)_\Ran \backslash \Gr_{G,\Ran})$ 
as acted on by the monoidal category $\Sph_{G,\Ran}$ 
on the right by convolutions. 

\medskip

This action commutes with the left action of $\Sph_{T,\Ran}$.

\sssec{}

Since $\fL(T)_\Ran$ normalizes $\fL(N)_\Ran$, the category 
$$(\SI_\Ran)^{\fL^+(T)_\Ran}:=\Shv(\Gr_{G,\Ran})^{\fL^+(T)_\Ran\cdot \fL(N)_\Ran}$$
inherits an action of $\Sph_{T,\Ran}$ and a commuting $\Sph_{G,\Ran}$-action.

\medskip

Working with this version of the semi-infinite category, we can define a t-structure on it in the same way 
as for 
$$\SI_\Ran:=\Shv(\Gr_{G,\Ran})^{\fL(N)_\Ran},$$
so that the forgetful functor
$$(\SI_\Ran)^{\fL^+(T)_\Ran}\to \SI_\Ran$$
is t-exact.

\medskip

Thus, we obtain that the object $\IC^\infty_\Ran\in \SI_\Ran \subset \Shv(\Gr_{G,\Ran})$ 
naturally lifts to an object of 
$$(\SI_\Ran)^{\fL^+(T)_\Ran}:= \Shv(\Gr_{G,\Ran})^{\fL(N)_\Ran\cdot \fL^+(T)_\Ran}
\subset \Shv(\fL^+(T)_\Ran \backslash \Gr_{G,\Ran});$$
by a slight abuse of notation we denote it by the symbol $\IC^\infty_\Ran$. 

\sssec{}

Fix a point $x$. Let $\Ran_x$ be the version of the Ran space with $x$ as a marked point. By definition,
for an affine test-scheme $Y$, the set $\Hom(Y,\Ran_x)$ consists of finite subsets 
$$\CI\subset \Hom(Y,X)$$
equipped with distinguished element $*\in \CI$ such that the corresponding map $Y\to X$ is
$$X\to \on{pt}\overset{x}\to X.$$

\sssec{}

We have the natural forgetful map $\Ran_x\to \Ran$, and we can use it to base change all the objects
considered above. 

\medskip

In particular, we consider the prestack 
$$\Gr_{G,\Ran_x}:=\Gr_{G,\Ran}\underset{\Ran}\times \Ran_x,$$ the category
$$\Shv(\fL^+(T)_{\Ran_x}\backslash\Gr_{G,\Ran_x}),$$
acted on by
$$\Sph_{G,\Ran_x}:=\Shv(\fL^+(G)_{\Ran_x}\backslash \Gr_{G,\Ran_x}) \text{ and }
\Sph_{T,\Ran_x}:=\Shv(\fL^+(T)_{\Ran_x}\backslash \Gr_{T,\Ran_x}),$$
etc.

\medskip

We can consider the corresponding object
$$\IC^\infty_{\Ran_x}\in \Shv(\Gr_{G,\Ran_x})^{\fL(N)_{\Ran_x}\cdot \fL^+(T)_{\Ran_x}}\subset 
\Shv(\fL^+(T)_{\Ran_x}\backslash\Gr_{G,\Ran_x}),$$
equal to the !-pullback of $\IC^\infty_{\Ran}$ along the projection $\Gr_{G,\Ran_x}\to \Gr_{G,\Ran}$. 

\sssec{} 

Note that we have a tautologically defined map
\begin{equation} \label{e:insert x}
\Ran_x\times \Gr_{G,x}\to \Gr_{G,\Ran_x}.
\end{equation} 

\medskip

From \eqref{e:insert x} we obtain a canonically defined monoidal functor
$$\Sph_{G,x}\to  \Sph_{G,\Ran_x}.$$

Composing with the geometric Satake functor
$$\Sat_{G,x}:\Rep(\cG)\to \Sph_{G,x},$$ 
we obtain a monoidal functor
$$\Sph_{G,x}\to  \Sph_{G,\Ran_x}.$$

\medskip

We modify the geometric Satake functor for $T$ by applying the cohomological shift by
$[-\langle \lambda,2\check\rho\rangle]$ on $\sfe^\lambda\in \Rep(\cT)$. Denote the resulting functor
by
$$\Sat'_{T,x}:\Rep(\cT)\to \Sph_{T,x}.$$

Pre-composing with
$$\Sph_{T,x}\to  \Sph_{T,\Ran_x},$$
we obtain a monoidal functor
$$\Rep(\cT)\to  \Sph_{T,\Ran_x}.$$

\sssec{}

Thus, we obtain that $\Shv(\fL^+(T)_{\Ran_x}\backslash \Gr_{G,\Ran_x})$ is a bimodule category for $(\Rep(\cT),\Rep(\cG))$. In this case, we can talk
about the category of \emph{graded Hecke objects} in $\Shv(\fL^+(T)_{\Ran_x}\backslash \Gr_{G,\Ran_x})$, denoted
$$\bHecke(\Shv(\fL^+(T)_{\Ran_x}\backslash \Gr_{G,\Ran_x})),$$
see \cite[Sect. 4.3.5]{Ga1}, and also \secref{sss:graded Hecke} below.

\medskip

These are objects $\CF\in \Shv(\fL^+(T)_{\Ran_x}\backslash \Gr_{G,\Ran_x})$, equipped with a system of isomorphisms  
$$\CF\star \Sat_{G,x}(V)\overset{\phi(V,\CF)}\longrightarrow \Sat'_{T,x}(\Res^\cG_\cT(V))\star \CF,\quad V\in \Rep(\cG)$$
that are compatible with the monoidal structure on $\Rep(\cG)$ in the sense that the diagrams
$$
\CD
\CF\star \Sat_{G,x}(V_1) \star \Sat_{G,x}(V_2)  @>{\phi(V_1,\CF)}>> \Sat'_{T,x}(\Res^\cG_\cT(V_1))\star \CF \star \Sat_{G,x}(V_2)  \\
@V{\sim}VV   @VV{\phi(V_2,\CF)}V  \\
\CF \star  \Sat_{G,x}(V_1\otimes V_2)  @>>>  \Sat'_{T,x}(\Res^\cG_\cT(V_1))\star  \Sat'_{T,x}(\Res^\cG_\cT(V_2))\star \CF  \\
@V{\phi(V_1\otimes V_2,\CF)}VV    @VV{\sim}V   \\
\Sat'_{T,x}(\Res^\cG_\cT(V_1\otimes V_2)) \star \CF @>{\sim}>>  \Sat'_{T,x}(\Res^\cG_\cT(V_1)\otimes \Res^\cG_\cT(V_2))\star \CF,
\endCD
$$
along with a coherent system of higher compatibilities. 

\sssec{}

We will prove:

\begin{thmconstr} \label{t:Hecke}
The object $\IC^\infty_{\Ran_x}\in \Shv(\fL^+(T)_{\Ran_x}\backslash \Gr_{G,\Ran_x})$ naturally lifts to an 
object of $\bHecke(\Shv(\fL^+(T)_{\Ran_x}\backslash \Gr_{G,\Ran_x}))$.
\end{thmconstr}

Several remarks are in order.

\begin{rem} 
In the proof of \thmref{t:Hecke}, the object $\IC^\infty_\Ran$ will come in its incarnation as $'\!\IC^\infty_\Ran$, 
constructed in \secref{ss:pres as colimit}.
\end{rem}

\begin{rem}
Consider the restriction
$$\IC^\infty_x\simeq \IC^\infty_{\Ran_x}|_{\Gr_{G,x}}.$$

The Hecke structure on $\IC^\infty_{\Ran_x}$ induces one on $\IC^\infty_x$. It will follow from the construction
and \cite[Sect. 6.2.5]{Ga1} that the resulting Hecke structure on $\IC^\infty_x$ coincides with one constructed in \cite[Sect. 5.1]{Ga1}.
\end{rem}

\begin{rem}
In order to prove \thmref{t:Hecke} we will need to consider the Hecke action of $\Rep(\cG)$ on $\Shv(\fL^+(T)_\Ran\backslash\Gr_{G,\Ran})$
over the entire Ran space. The next few subsections are devoted to setting up the corresponding formalism.
\end{rem} 

\ssec{Categories over the Ran space, continued} \label{ss:over Ran,ctd}

\sssec{}

Recall the construction 
\begin{equation} \label{e:spread constr}
\CA\rightsquigarrow \on{Fact}^{\on{alg}}(\CA)_I
\end{equation} 
of \secref{ss:spread over Ran}, viewed as a functor
$\DGCat^{\on{SymMon}}\to \Shv(X^I)\mmod$.

\medskip

Note that the functor \eqref{e:spread constr} has a natural right-lax symmetric monoidal structure, i.e.,
we have the natural transformation
$$\on{Fact}^{\on{alg}}(\CA')_I\underset{\Shv(X^I)}\otimes \on{Fact}^{\on{alg}}(\CA'')_I\to \on{Fact}^{\on{alg}}(\CA'\otimes \CA'')_I.$$

\medskip

In particular, since any $\CA\in \DGCat^{\on{SymMon}}$ can be viewed as an object in $\on{ComAlg}(\DGCat^{\on{SymMon}})$, we obtain that 
$\on{Fact}^{\on{alg}}(\CA)_I$ itself acquires a structure of symmetric monoidal category.



\sssec{}

For a surjection of finite sets $\phi:I_1\twoheadrightarrow I_2$, the corresponding functor
\begin{equation} \label{e:restr factor}
(\Delta_\phi)^!:\on{Fact}^{\on{alg}}(\CA)_{I_1}\to \on{Fact}^{\on{alg}}(\CA)_{I_2}
\end{equation}
(see \secref{sss:Ran restr}) is naturally symmetric monoidal. In particular, we obtain that 
$$\on{Fact}^{\on{alg}}(\CA)_\Ran\simeq \underset{I}{\on{lim}}\, \on{Fact}(\CA)_I$$
(see \eqref{e:I Ran lim}) acquires a natural symmetric monoidal structure. 

\sssec{}

Let $\CA'\to \CA''$ be a right-lax symmetric monoidal functor. The functor \eqref{e:spread constr} 
gives rise to a right-lax symmetric monoidal functor
$$\on{Fact}^{\on{alg}}(\CA')_I\to \on{Fact}^{\on{alg}}(\CA'')_I,$$
compatible with the restriction 
functors \eqref{e:restr factor}. Varying $I$, we obtain a right-lax symmetric monoidal functor
$$\on{Fact}^{\on{alg}}(\CA')_\Ran\to \on{Fact}^{\on{alg}}(\CA'')_\Ran.$$

\medskip

In particular, a commutative algebra object $A$ in $\CA$, viewed as a right-lax symmetric monoidal functor
$\Vect\to \CA$, gives rise to a commutative algebra 
$$\on{Fact}^{\on{alg}}(A)_I\in \on{Fact}^{\on{alg}}(\CA)_I.$$

These algebra objects are compatible under the restriction functors \eqref{e:restr factor}. Varying $I$, we obtain 
a commutative algebra object 
$$\on{Fact}^{\on{alg}}(A)_\Ran\in \on{Fact}(\CA)_\Ran.$$

\sssec{Examples} Let us consider the two examples of $\CA$ from \secref{sss:ex fact categ}. 

\medskip

\noindent(i) Let $\CA=\Vect$. We obtain that to $A\in \on{ComAlg}(\Vect)$ we can canonically assign 
an object $\on{Fact}^{\on{alg}}(A)_\Ran\in \Shv(\Ran)$.

\medskip

\noindent(ii) Let $\CA$ be the category of $\Lambda^{\on{neg}}-0$ graded vector spaces. Note that a commutative
algebra $A$ in $\CA$ is the same as a commutative $\Lambda^{\on{neg}}$-algebra with $A(0)=k$. On the one hand,
the construction of \secref{ss:factor alg} assigns to such an $A$ a collection of objects
$$\on{Fact}^{\on{alg}}(A)_{X^\lambda}\in \Shv(X^\lambda), \quad \lambda\in \Lambda^{\on{neg}}-0.$$

On the other hand, we have the above object 
$$\on{Fact}^{\on{alg}}(A)_\Ran\in \on{Fact}^{\on{alg}}(\CA)_\Ran.$$

By unwinding the constructions we obtain that these two objects match up under the equivalence \eqref{e:fact graded}.

\ssec{Digression: right-lax central structures}

\sssec{}

Let $\CA$ and $\CA'$ be symmetric monoidal categories, and let $\CC$ be a $(\CA',\CA)$-bimodule category.
Let $F:\CA\to \CA'$ be a right-lax symmetric monoidal functor. 

\medskip

A right-lax central structure on an object $c\in \CC$ with respect to $F$ is a system of maps
$$F(a)\otimes c\overset{\phi(a,c)}\longrightarrow c\otimes a, \quad a\in \CA$$
that make the diagrams
$$
\CD
F(a_1)\otimes (F(a_2)\otimes c) @>{\phi(a_2,c)}>>  F(a_1)\otimes (c\otimes a_2) \\
@V{\sim}VV   @VV{\sim}V  \\
(F(a_1)\otimes F(a_2))\otimes c & & (F(a_1)\otimes c)\otimes a_2 \\
@VVV    @VV{\phi(a_1,c)}V   \\
F(a_1\otimes a_2) \otimes c & & (c\otimes a_1)\otimes a_2 \\
@V{\phi(a_1\otimes a_2,c)}VV   @VV{\sim}V  \\
c\otimes (a_1\otimes a_2)  @>{\on{id}}>> c\otimes (a_1\otimes a_2),
\endCD
$$
commute, along with a coherent system of higher compatibilities. 

\medskip

Denote the category of objects of $\CC$ equipped with a right-lax central structure on an object with respect to $F$ 
by $Z_F(\CC)$.

\sssec{}

From now on we will assume that $\CA$ is rigid (see \cite[Chapter 1, Sect. 9.1]{GR} for what this means). 

\medskip

If $\CA$ is compactly generated, this
condition is equivalent to requiring that the class of compact objects in $\CA$ coincides with the class of objects that are dualizable
with respect to the symmetric monoidal structure on $\CA$. 

\sssec{}

Assume for a moment that $F$ is strict (i.e., is a genuine symmetric monoidal functor). We have:

\begin{lem} \label{l:central genuine}
If $c\in Z_F(\CC)$, then the morphisms $\phi(a,c)$ are isomorphisms.
\end{lem}

In other words, this lemma says that if $F$ is genuine, then any right-lax central structure is a genuine central structure
(under the assumption that $\CA$ is rigid).

\sssec{}

Let $R_\CA\in \CA\otimes \CA$ be the (commutative) algebra object,
obtained by applying the right adjoint 
$$\CA\to \CA\otimes \CA$$
of the monoidal operation $\CA\otimes \CA\to \CA$, to the unit object ${\bf 1}_\CA\in \CA$.

\medskip

Consider the (commutative) algebra object
$$R^F_\CA:=(F\otimes \on{id})(R_\CA)\in \CA'\otimes \CA.$$

We have:

\begin{lem} \label{l:central structure}
A datum of right-lax central structure on an object $c\in \CC$ is equivalent to upgrading $c$ to an 
object of $R^F_\CA\mod(\CC)$.
\end{lem}

\sssec{}

Let $F'$ be another right-lax symmetric monoidal functor, and let $F\to F'$ be a right-lax symmetric monoidal natural transformation. 
Restriction defines a functor
\begin{equation} \label{e:restr lax central}
Z_{F'}(\CC)\to Z_F(\CC).
\end{equation} 

In addition, we have a homomorphism of commutative algebra objects in $\CA'\otimes \CA$
$$R^F_\CA\to R^{F'}_\CA.$$

It easy to see that with respect to the equivalence of \lemref{l:central structure}, the diagram
$$
\CD
Z_{F'}(\CC)   @>>>   Z_F(\CC)    \\
@V{\sim}VV   @VV{\sim}V   \\
R^{F'}_\CA\mod(\CC)   @>>> R^F_\CA\mod(\CC),
\endCD
$$
commutes, where the bottom arrow is given by restriction.

\medskip

In particular, we obtain that the functor \eqref{e:restr lax central} admits a left adjoint, given by
$$R^{F'}_\CA\underset{R^F_\CA}\otimes -.$$

\sssec{}  \label{sss:lax central, factor}

We now modify our context, and we let $\CC$ be a module category for
$$\on{Fact}^{\on{alg}}(\CA'\otimes \CA)_I.$$

We have the corresponding category of right-lax central objects, denoted by the same symbol $Z_F(\CC)$, which can be identified with
$$\on{Fact}^{\on{alg}}(R^F_\CA)_I\mod(\CC).$$

\medskip

For a right-lax symmetric monoidal natural transformation $F\to F'$, the left adjoint to the restriction functor
$Z_{F'}(\CC)\to Z_F(\CC)$ is given by 
\begin{equation} \label{e:tensor up fact}
\on{Fact}^{\on{alg}}(R^{F'}_\CA)_I\underset{\on{Fact}^{\on{alg}}(R^F_\CA)_I}\otimes -.
\end{equation} 

\sssec{} \label{sss:lax central, Ran}

Let 
$$I\rightsquigarrow \CC_I, \quad I\in \on{Fin}^{\on{surj}}$$
be a compatible family of module categories over $\on{Fact}(\CA'\otimes \CA)_I$.

\medskip

Set
$$\CC_\Ran:=\underset{I\in \on{Fin}^{\on{surj}}}{\on{lim}}\, \CC_I.$$

\medskip

We can thus talk about an object $c\in \CC_\Ran$ being equipped with a right-lax central structure with respect to $F$.
Denote the corresponding category of right-lax central objects by $Z_F(\CC_\Ran)$. 

\medskip

The functors \eqref{e:tensor up fact} provide a left adjoint to the forgetful functor
$$Z_{F'}(\CC_\Ran)\to Z_F(\CC_\Ran).$$

This follows from the fact that for a surjective map of finite sets $\phi:I_1\twoheadrightarrow I_2$, the natural transformation
in the diagram
$$
\CD
Z_F(\CC_{I_1})  @>{\Delta_\phi^!}>> Z_F(\CC_{I_2}) \\
@V{\on{Fact}^{\on{alg}}(R^{F'}_\CA)_{I_1}\underset{\on{Fact}^{\on{alg}}(R^F_\CA)_{I_1}}\otimes-}VV   
@VV{\on{Fact}^{\on{alg}}(R^{F'}_\CA)_{I_2}\underset{\on{Fact}^{\on{alg}}(R^F_\CA)_{I_2}}\otimes-}V  \\
Z_{F'}(\CC_{I_1})  @>{\Delta_\phi^!}>> Z_{F'}(\CC_{I_2}) 
\endCD
$$
is an isomorphism. 

\ssec{Hecke and Drinfeld-Pl\"ucker structures}  \label{ss:DrPl}

We will be interested in the following particular cases of the 
above situation\footnote{The formalism described in this subsection (as well as the term) was suggested by S.~Raskin.}. 

\sssec{}  \label{sss:graded Hecke}

Take $\CA=\Rep(\cG)$ and $\CA'=\Rep(\cT)$
with $F'$ being given by restriction along $\cT\to \cG$. We denote the corresponding category
$Z_{F'}(\CC)$ by 
$$\bHecke(\CC).$$

\medskip

By \lemref{l:central genuine}, its objects are $c\in \CC$, equipped with a system of \emph{iso}morphisms
$$\Res^\cG_\cT(V)\otimes c\simeq c\otimes V, \quad V\in \Rep(\cG),$$
compatible with tensor products of the $V$'s.

\medskip

For this reason, we call a (right-lax) central structure on an object of $\CC$ in this case
a \emph{graded Hecke structure}.

\medskip

Equivalently, these are objects of $\CC$ equipped with an action of the algebra
$$R^{F'}_\CA:=(\on{Res}(\cG_\cT)\otimes \on{id})(R_\cG),$$
where $R_\cG\in \Rep(\cG)\otimes \Rep(\cG)$ is the regular representation. 

\sssec{}

Let us now take $\CA=\Rep(\cG)$ and $\CA'=\Rep(\cT)$, but the functor $F$
is given by the \emph{non-derived} functor of $\cN$-invariants
$$V^\lambda\mapsto V^\lambda(\lambda)=\sfe^\lambda.$$

\medskip

The corresponding algebra object
$$R^F_\CA\in \Rep(\cT)\otimes \Rep(\cG)$$
is $\CO(\ol{\cN\backslash \cG})$, where $\ol{\cN\backslash \cG}$ is the base affine space of $\cG$, viewed
as acted on on the left by $\cT$ and on the right by $\cG$.

\medskip

We denote the corresponding category
$Z_{F}(\CC)$ by 
$$\on{DrPl}(\CC).$$

\medskip

By definition, its objects are $c\in \CC$, equipped with a collection of maps
$$\sfe^\lambda\otimes c\overset{\phi(\lambda,c)}\longrightarrow c\otimes V^\lambda$$
that make the diagrams 
$$
\CD
\sfe^\lambda\otimes (\sfe^\mu\otimes c)  @>{\phi(\mu,c)}>>   \sfe^\lambda\otimes (c\otimes V^\mu) \\
@V{\sim}VV   @VV{\sim}V  \\
(\sfe^\lambda\otimes \sfe^\mu)\otimes c & & (\sfe^\lambda\otimes c)\otimes V^\mu \\
@V{\sim}VV  @VV{\phi(\lambda,c)}V  \\
\sfe^{\lambda+\mu}\otimes c  & & (c\otimes V^\lambda)\otimes V^\mu \\
@V{\phi(\lambda+\mu,c)}VV  @VV{\sim}V   \\
c\otimes V^{\lambda+\mu}  @>>>  c\otimes (V^\lambda\otimes V^\mu)
\endCD
$$
commute, along with a coherent system of higher compatibilities. 

\medskip

We will call a right-lax central structure on an object of $\CC$ in this case a \emph{Drinfeld-Pl\"ucker} structure. 

\sssec{}

We have a right-lax symmetric monoidal natural transformation $F\to F'$,
$$\sfe^\lambda \to \Res^\cG_\cT(V^\lambda).$$

The corresponding morphism of commutative algebra objects in $\Rep(\cT)\otimes \Rep(\cG)$ is given by pullback
along the projection map
$$\cG\to \ol{\cN\backslash \cG}.$$

\medskip

Consider the forgetful functor
$$\Res^{\bHecke}_{\on{DrPl}}:\bHecke(\CC)\to \on{DrPl}(\CC),$$
and its left adjoint
$$\Ind^{\bHecke}_{\on{DrPl}}:\on{DrPl}(\CC)\to \bHecke(\CC).$$

\sssec{}

Let us now recall the statement of \cite[Proposition 6.2.4]{Ga1} that describes the composition
\begin{equation} \label{e:induce forget}
\on{DrPl}(\CC) \overset{\Ind^{\bHecke}_{\on{DrPl}}}\longrightarrow \bHecke(\CC) \to \CC,
\end{equation}
where the second arrow is the forgetful functor. 

\medskip

Given an object $c\in \on{DrPl}(\CC)$, the construction of \cite[Sect. 2.7]{Ga1} defines a functor
$\Lambda^+\to \CC$, which at the level of objects sends $\lambda\in \Lambda^+$ to
$$\sfe^{-\lambda}\otimes c\otimes V^\lambda.$$ 

\medskip

The assertion \cite[Proposition 6.2.4]{Ga1} says that the value of \eqref{e:induce forget} on the above $c$
is canonically identified with
$$\underset{\lambda\in \Lambda^+}{\on{colim}}\, \sfe^{-\lambda}\otimes c\otimes V^\lambda.$$ 

\sssec{}

We now place ourselves in the context of \secref{sss:lax central, factor}. Let $\CC$ be a module category for
$$\on{Fact}^{\on{alg}}(\Rep(\cT)\otimes \Rep(\cG))_I.$$

We denote corresponding categories $Z_{F'}(\CC)$ and $Z_{F}(\CC)$ by $\bHecke(\CC)$
and $\on{DrPl}(\CC)$, respectively. 

\medskip

Let $c\in \CC$ be an object of $Z_{F}(\CC)$. We wish to describe the value on $c$ of the composite functor
\begin{equation} \label{e:induce forget factor}
\on{DrPl}(\CC) \overset{\Ind^{\bHecke}_{\on{DrPl}}}\longrightarrow \bHecke(\CC) \to \CC
\end{equation} 

\sssec{}  \label{sss:extend to functor}

For $\ul\lambda\in \Maps(I,\Lambda^+)$, recall the object $V^{\ul\lambda}\in \on{Fact}(\Rep(\cG))_I$, see \secref{sss:V lambda I}.
Similarly, we have the object
$$\sfe^{\ul\lambda}\in \on{Fact}(\Rep(\cT))_I.$$

\medskip

The construction of \cite[Sect.2.7]{Ga1} defines on the assignment 
$$\ul\lambda\mapsto \sfe^{-\ul\lambda}\otimes c\otimes V^{\ul\lambda}$$
a structure of a functor
$$\Maps(I,\Lambda^+)\to \CC.$$

Generalizing \cite[Proposition 6.2.4]{Ga1} one shows:

\begin{prop}  \label{p:Sam's}
The value of the composite functor \eqref{e:induce forget factor} on $c\in \on{DrPl}(\CC)$ identifies
canonically with
$$\underset{\ul\lambda\in \Maps(I,\Lambda^+)}{\on{colim}}\, \ul\lambda\mapsto \sfe^{-\ul\lambda}\otimes c\otimes V^{\ul\lambda}.$$
\end{prop}

\sssec{}  \label{sss:V lambda Ran}

Let $I\rightsquigarrow \CC_I$ be as in \secref{sss:lax central, Ran}. Consider the corresponding categories
$\on{DrPl}(\CC_\Ran)$ and $\bHecke(\CC_\Ran)$. 

\medskip

The compatibility of the functors $\Ind^{\bHecke}_{\on{DrPl}}$ for surjections of finite sets gives rise to a well-defined functor
$$\Ind^{\bHecke}_{\on{DrPl}}:\on{DrPl}(\CC_\Ran)\to \bHecke(\CC_\Ran),$$
left adjoint to the restriction functor. 

\medskip

For $c\in \on{DrPl}(\CC_\Ran)$, the value of the composite functor
$$\on{DrPl}(\CC) \overset{\Ind^{\bHecke}_{\on{DrPl}}}\longrightarrow \bHecke(\CC) \to \CC\to \CC_I$$
is given by
$$\underset{\ul\lambda\in \Maps(I,\Lambda^+)}{\on{colim}}\, \sfe^{-\ul\lambda}\otimes c_I\otimes V^{\ul\lambda},$$
where $c_I$ is the value of $c$ in $\CC_I$. 

\ssec{The Hecke property--enhanced statement}  \label{ss:DrPl on IC}

\sssec{}

The key property of the geometric Satake functor
$$\Sat_{G,I}:\on{Fact}^{\on{alg}}(\Rep(\cG))_I\to \Sph_{G,I}$$
is that it is has a natural monoidal structure.

\medskip

The same applies to the modified geometric Satake functor $\Sat'_{T,I}$ for $T$. 

\medskip

Thus, we obtain that the category $\Shv(\fL^+(T)_I\backslash \Gr_{G,I})$ is as acted on by the monoidal category 
$\on{Fact}^{\on{alg}}(\Rep(\cT)\otimes \Rep(\cG))_I$.

\medskip

These actions are compatible under surjective maps of finite sets $I_1\twoheadrightarrow I_2$.  

\sssec{} \label{sss:DrPl on IC}

Consider the object
$$\delta_{1_\Gr,I}:=(s_I)_!(\omega_{X^I})\in \Shv(\fL^+(T)_I\backslash \Gr_{G,I}),$$
where $s_I:X^I\to \Gr_{G,I}$ is the unit section.

\medskip

It follows from the construction of the functor $\Sat_{G,I}$ that $\delta_{\ul{0},I}$ lifts canonically to an object of
$$\on{DrPl}(\Shv(\fL^+(T)_I\backslash \Gr_{G,I})).$$

\sssec{}

Consider the corresponding object
$$\Ind^{\bHecke}_{\on{DrPl}}(\delta_{1_\Gr,I})\in \bHecke(\Shv(\fL^+(T)_I\backslash \Gr_{G,I})).$$

It follows from \propref{p:Sam's} that its image under the forgetful functor
$$\bHecke(\Shv(\fL^+(T)_I\backslash \Gr_{G,I}))\to \Shv(\fL^+(T)_I\backslash \Gr_{G,I})\to  \Shv(\Gr_{G,I})$$
identifies canonically with the object $\IC^\semiinf_I$, constructed in \secref{sss:semiinf I}. 

\sssec{}

Consider now the object
$$\delta_{1_\Gr,\Ran}:=(s_\Ran)_!(\omega_{\Ran})\in \Shv(\fL^+(T)_\Ran\backslash \Gr_{G,\Ran}),$$
where $s_\Ran:\Ran\to \Gr_{G,\Ran}$ is the unit section.

\medskip

It naturally lifts to an object of
$$\on{DrPl}(\Shv(\fL^+(T)_\Ran\backslash \Gr_{G,\Ran})).$$

\medskip

Consider the corresponding object
$$\Ind^{\bHecke}_{\on{DrPl}}(\delta_{1_\Gr,\Ran})\in \bHecke(\Shv(\fL^+(T)_\Ran\backslash \Gr_{G,\Ran})).$$

By \secref{sss:V lambda Ran}, the image of $\Ind^{\bHecke}_{\on{DrPl}}(\delta_{1_\Gr,\Ran})$ under the forgetful functor
$$\bHecke(\Shv(\fL^+(T)_\Ran\backslash \Gr_{G,\Ran}))\to \Shv(\fL^+(T)_\Ran\backslash \Gr_{G,\Ran})\to \Shv(\Gr_{G,\Ran})$$
identifies canonically with the object $'\!\IC^\semiinf_\Ran$, constructed in \secref{sss:semiinf prime}. 

\begin{rem}
The latter could be used to define on the assignment
$$I\rightsquigarrow \IC^\semiinf_I$$
a homotopy-coherent system of compatibilities as $I$ varies over $\on{Fin}^{\on{surj}}$.
\end{rem} 

\sssec{}

Using the isomorphism 
$$'\!\IC^\semiinf_\Ran\simeq \IC^\semiinf_\Ran$$
of \thmref{t:descr as colimit}, we thus obtain a lift of $\IC^\semiinf_\Ran$ to an object of $\bHecke(\Shv(\fL^+(T)_\Ran\backslash \Gr_{G,\Ran}))$.

\medskip

Summarizing, we obtain:

\begin{thm} \label{t:Hecke Ran}
The object $\IC^\semiinf_\Ran\in \Shv(\fL^+(T)_\Ran\backslash \Gr_{G,\Ran}))$ naturally lifts to an object of 
$\bHecke(\Shv(\fL^+(T)_\Ran\backslash \Gr_{G,\Ran}))$.
\end{thm}

\ssec{Recovering the pointwise Hecke structure}

In this subsection we will finally complete the proof of \thmref{t:Hecke}. 

\sssec{}

The constructions in Sects. \ref{ss:over Ran,ctd}-\ref{ss:DrPl} carry over to the situation when $\Ran$ is replaced by $\Ran_x$. 
From \thmref{t:Hecke Ran} we obtain that the object 
$$\IC^\semiinf_{\Ran_x}\in \Shv(\fL^+(T)_{\Ran_x}\backslash \Gr_{G,\Ran_x}))$$
naturally lifts to an object of 
$\bHecke(\Shv(\fL^+(T)_{\Ran_x}\backslash \Gr_{G,\Ran_x}))$.

\sssec{}

Now, we have a symmetric monoidal functor
$$\Rep(\cT)\otimes \Rep(\cG)\to \on{Fact}(\Rep(\cT)\otimes \Rep(\cG))_{\Ran_x}.$$

\medskip

Restricting, we obtain that $\IC^\semiinf_{\Ran_x}$ lifts to an object of $\bHecke(\Shv(\fL^+(T)_{\Ran_x}\backslash \Gr_{G,\Ran_x}))$,
as stated in \thmref{t:Hecke}. 

\section{Local vs global compatibility of the Hecke structure} \label{s:H g-l}

In this section we will establish a compatibility between the Hecke structure on $\ICs_\Ran$ constructed in the
previous section and the corresponding structure on $\ICs_{\on{glob}}$ established in \cite{BG1}. 

\ssec{The relative version of the Ran Grassmannian}

\sssec{}

We introduce a relative version of the prestack $\Gr_{G,\Ran}$ over $\Bun_T$, denoted $\Gr_{G,\Ran}\wt\times\Bun_T$,
as follows.

\medskip

Let $(\Ran\times \Bun_T)^{\on{level}}$ be the prestack that classifies the data of $(\CP_T,\CI,\beta)$, where:

\medskip

\noindent(i) $\CI$ is a finite non-empty collection of points on $X$;

\medskip

\noindent(ii) $\CP_T$ is a $T$-bundle on $X$;

\medskip

\noindent(iii) $\beta$ is a trivialization of $\CP_T$ on the formal neighborhood of $\Gamma_\CI$.

\medskip

The prestack $(\Ran\times \Bun_T)^{\on{level}}$ is acted on by $\fL(T)_\Ran$, and the map
$$(\Ran\times \Bun_T)^{\on{level}}\to \Bun_T\times \Ran$$
is a $\fL^+(T)_\Ran$-torsor, locally trivial in the \'etale (in fact, even Zariski, since $T$ is a torus) topology.

\medskip

We set
$$\Gr_{G,\Ran}\wt\times\Bun_T:=\fL^+(T)_\Ran\backslash \left(\Gr_{G,\Ran} \underset{\Ran}\times (\Ran\times \Bun_T)^{\on{level}}\right).$$

\medskip

We have a tautological projection
$$r:\Gr_{G,\Ran}\wt\times\Bun_T\to \fL^+(T)_\Ran\backslash \Gr_{G,\Ran}.$$

\sssec{}

The right action of the groupoid
\begin{equation} \label{e:Hecke groupoid G}
\fL^+(G)_\Ran\backslash \fL(G)_\Ran/\fL^+(G)_\Ran
\end{equation} 
on $\Gr_{G,\Ran}$ naturally lifts to an action on
$\Gr_{G,\Ran}\wt\times\Bun_T$, in a way compatible with the projection $r$.

\medskip

In addition, by construction, we have an action of the groupoid 
\begin{equation} \label{e:Hecke groupoid T}
\fL^+(T)_\Ran\backslash \fL(T)_\Ran/\fL^+(T)_\Ran
\end{equation}  
on $\Gr_{G,\Ran}\wt\times\Bun_T$, also compatible with the projection $r$.

\medskip

In particular, we obtain that $\Shv(\Gr_{G,\Ran}\wt\times\Bun_T)$ is a bimodule category for $(\Sph_{T,\Ran},\Sph_{G,\Ran})$,
and hence for $(\on{Fact}(\Rep(\cT)_\Ran,\on{Fact}(\Rep(\cG))_\Ran)$, via the Geometric Satake functor,
where we use the functor $\Sat'_{T,\Ran}$ to map
$$\on{Fact}^{\on{alg}}(\Rep(\cT))_\Ran\to \Sph_{T,\Ran}.$$

\medskip

Base-changing along $X^I\to \Ran$ we obtain a compatible family of module categories for 
$(\on{Fact}^{\on{alg}}(\Rep(\cT)_I,\on{Fact}^{\on{alg}}(\Rep(\cG))_I)$, for $I\in\on{Fin}^{\on{surj}}$. 

\sssec{}

Denote:
$$\IC^\semiinf_{\Ran,\Bun_T}:=r^!(\IC^\semiinf_\Ran).$$

From \thmref{t:Hecke Ran}, we obtain that $\IC^\semiinf_{\Ran,\Bun_T}$ naturally lifts to an object of
$$\bHecke(\Shv(\Gr_{G,\Ran}\wt\times\Bun_T));$$
moreover we have:
\begin{equation} \label{e:rel Hecke}
\IC^\semiinf_{\Ran,\Bun_T}\simeq \Ind^{\bHecke}_{\on{DrPl}}(\delta_{1_\Gr,\Ran,\Bun_T}),
\end{equation}
where 
$$\delta_{1_\Gr,\Ran,\Bun_T}=(s_{\Ran,\Bun_T})_!(\omega_{\Ran\times \Bun_T}),$$
and where $s_{\Ran,\Bun_T}$ is the unit section
$$\Ran\times \Bun_T\to \Gr_{G,\Ran}\wt\times \Bun_T.$$

\ssec{Hecke property in the global setting}

\sssec{}

Consider the stack $\BunBb$, and consider its version
$$(\BunBb\times \Ran)_{\on{poles}}$$
defined as follows:

\medskip

A point of $(\BunBb\times \Ran)_{\on{poles}}$ is a quadruple $(\CP_G,\CP_T,\kappa,\CI)$, where

\medskip

\noindent(i) $\CP_G$ is a $G$-bundle on $X$;

\medskip

\noindent(ii) $\CP_T$ is a $T$-bundle on $X$;

\medskip

\noindent(iii) $\CI$ is a finite non-empty collection of points on $X$;

\medskip

\noindent(iv) $\kappa$ is a datum of maps
$$\kappa^{\check\lambda}:\check\lambda(\CP_T)\to \CV^{\check\lambda}_{\CP_G}$$
that are allowed to have poles on $\Gamma_\CI$, and that satisfy the Pl\"ucker relations.

\medskip

Note that we have a closed embedding
$$\BunBb\times \Ran\hookrightarrow (\BunBb\times \Ran)_{\on{poles}},$$
corresponding to the condition that the maps $\kappa^{\check\lambda}$ have no poles. 

\sssec{}

Hecke modifications of the $G$-bundle (resp., $T$-bundle) define a right (resp., left) action of the groupoid \eqref{e:Hecke groupoid G}
(resp., \eqref{e:Hecke groupoid T}) 
on $(\BunBb\times \Ran)_{\on{poles}}$. 

\medskip

In particular, the category $\Shv((\BunBb\times \Ran)_{\on{poles}})$ acquires a natural structure of bimodule category for 
$(\Sph_{T,\Ran},\Sph_{G,\Ran})$, and hence for $(\on{Fact}^{\on{alg}}(\Rep(\cT))_\Ran,\on{Fact}^{\on{alg}}(\Rep(\cG))_\Ran)$. 

\medskip

Base-changing along $X^I\to \Ran$ we obtain a compatible family of module categories for 
$(\on{Fact}^{\on{alg}}(\Rep(\cT))_I,\on{Fact}^{\on{alg}}(\Rep(\cG))_I)$, for $I\in\on{Fin}^{\on{surj}}$. 

\sssec{}

Denote
$$\IC^\semiinf_{\on{glob},\Bun_T}:=\IC_{\BunBb}\boxtimes \,\omega_{\Ran}\subset \Shv((\BunBb\times \Ran)_{\on{poles}}).$$

The following assertion is (essentially) established in \cite[Theorem 3.1.4]{BG1}:

\begin{thm} \label{t:Hecke global}
The object $\IC^\semiinf_{\on{glob},\Bun_T}$ naturally lifts to an object of the category 
$$\bHecke(\Shv((\BunBb\times \Ran)_{\on{poles}})).$$
\end{thm} 

\ssec{Local vs global compatibility}

\sssec{}

Note now that the map
$$\pi_\Ran:\ol{S}{}^0_\Ran\to \BunNb$$
naturally extends to a map
$$\pi_{\Ran,\Bun_T}:\Gr_{G,\Ran}\wt\times\Bun_T\to (\BunBb\times \Ran)_{\on{poles}}.$$

\medskip

We consider the functor
$$(\pi_{\Ran,\Bun_T})'{}^!:\Shv((\BunBb\times \Ran)_{\on{poles}})\to \Shv(\Gr_{G,\Ran}\wt\times\Bun_T)$$
obtained from $(\pi_{\Ran,\Bun_T})^!$ by applying the shift by $[d-\langle \lambda,2\check\rho\rangle]$
over the connected component $\Bun_T^\lambda$ of $\Bun_T$. 

\medskip

A relative version of the calculation performed in the proof of \thmref{t:IC loc glob} shows: 

\begin{thm}  \label{t:IC loc glob rel}
There exists a canonical isomorphism in $\Shv(\Gr_{G,\Ran}\wt\times\Bun_T)$
$$(\pi_{\Ran,\Bun_T})'{}^!(\IC^\semiinf_{\on{glob},\Bun_T})\simeq \IC^\semiinf_{\Ran,\Bun_T}.$$
\end{thm} 

\sssec{}

The map $r$ is compatible with the actions of the groupoids \eqref{e:Hecke groupoid G} and \eqref{e:Hecke groupoid T}.
In particular, the pullback functor
$$(\pi_{\Ran,\Bun_T})^!:\Shv((\BunBb\times \Ran)_{\on{poles}})\to \Shv(\Gr_{G,\Ran}\wt\times\Bun_T)$$
is a map of bimodule categories for $(\Sph_{T,\Ran},\Sph_{G,\Ran})$.

\medskip

Hence, we obtain that the functor $(\pi_{\Ran,\Bun_T})'{}^!$ can be thought of as a map of bimodule categories for 
$(\on{Fact}^{\on{alg}}(\Rep(\cT))_\Ran,\on{Fact}^{\on{alg}}((\Rep(\cG))_\Ran)$.

\sssec{}

We are now ready to state the main result of this section:

\begin{thm} \label{t:Hecke compat}
The isomorphism $(\pi_{\Ran,\Bun_T})'{}^!(\IC^\semiinf_{\on{glob},\Bun_T})\simeq \IC^\semiinf_{\Ran,\Bun_T}$ of
\thmref{t:IC loc glob rel} canonically lifts to an isomorphism of objects of $\bHecke(\Shv(\Gr_{G,\Ran}\wt\times\Bun_T))$.
\end{thm}

\ssec{Proof of \thmref{t:Hecke compat}}  \label{ss:proof of Hecke compat}

\sssec{}

Consider the tautological map
\begin{equation} \label{e:unit map 1}
\delta_{1_\Gr,\Ran,\Bun_T}\to  \Ind^{\bHecke}_{\on{DrPl}}(\delta_{1_\Gr,\Ran,\Bun_T}).
\end{equation}

Under the isomorphism 
$$\Ind^{\bHecke}_{\on{DrPl}}(\delta_{1_\Gr,\Ran,\Bun_T})\simeq \IC^\semiinf_{\Ran,\Bun_T}$$
of \eqref{e:rel Hecke}, this map corresponds to the map
\begin{equation} \label{e:delta to IC}
\delta_{1_\Gr,\Ran,\Bun_T}\to \IC^\semiinf_{\Ran,\Bun_T},
\end{equation} 
arising, by the $((s_{\Ran,\Bun_T})_!,(s_{\Ran,\Bun_T})^!)$ adjunction, from the isomorphism
$$\omega_{\Ran\times \Bun_T}\to (s_{\Ran,\Bun_T})^!(\IC^\semiinf_{\Ran,\Bun_T}).$$

\sssec{}

Consider the composite
\begin{equation} \label{e:unit map 2}
\delta_{1_\Gr,\Ran,\Bun_T}\to \Ind^{\bHecke}_{\on{DrPl}}(\delta_{1_\Gr,\Ran,\Bun_T})\simeq 
\IC^\semiinf_{\Ran,\Bun_T}\to (\pi_{\Ran,\Bun_T})'{}^!(\IC^\semiinf_{\on{glob},\Bun_T}).
\end{equation}

We obtain that the data on the morphism 
$$\IC^\semiinf_{\Ran,\Bun_T}\to (\pi_{\Ran,\Bun_T})'{}^!(\IC^\semiinf_{\on{glob},\Bun_T})$$
of a map of objects of $\bHecke(\Shv(\Gr_{G,\Ran}\wt\times\Bun_T))$ is equivalent to the data 
on \eqref{e:unit map 2} of a map of objects of $\on{DrPl}(\Shv(\Gr_{G,\Ran}\wt\times\Bun_T))$.

\sssec{}

The map \eqref{e:unit map 2} can be explicitly described as follows. By the $((s_{\Ran,\Bun_T})_!,(s_{\Ran,\Bun_T})^!)$
adjunction, it corresponds to the (iso)mophism 
\begin{equation} \label{e:unit map 3}
\omega_{\Ran\times \Bun_T}\to (s_{\Ran,\Bun_T})^!\circ (\pi_{\Ran,\Bun_T})'{}^!(\IC^\semiinf_{\on{glob},\Bun_T})
\end{equation}
constructed as follows: 

\medskip

We note that the map
$$\pi_{\Ran,\Bun_T}\circ s_{\Ran,\Bun_T}:\Ran\times \Bun_T \to (\BunBb\times \Ran)_{\on{poles}}$$
factors as 
$$\Ran\times \Bun_T\to \Ran\times \Bun_B\to \Ran\times \BunBb\to (\BunBb\times \Ran)_{\on{poles}}.$$

Now, the map \eqref{e:unit map 3} is the natural isomorphism coming from the identification
$$\IC^\semiinf_{\on{glob},\Bun_T}|_{\Ran\times \Bun^\lambda_B}[d-\langle \lambda,2\check\rho\rangle]
\simeq \omega_{\Ran\times \Bun^\lambda_B}.$$

\sssec{}

Now, by  unwinding the construction of the Hecke structure on $\IC^\semiinf_{\on{glob},\Bun_T}$ in \cite[Theorem 3.1.4]{BG1}, one shows that
the map \eqref{e:unit map 2} indeed canonically lifts to a map in $\on{DrPl}(\Shv(\Gr_{G,\Ran}\wt\times\Bun_T))$.

\qed

\appendix

\section{Proof of \thmref{t:contr}}  \label{s:app}

With future applications in mind, we will prove a generalization of \thmref{t:contr}. The proof is a paraphrase of the theory developed
in \cite{Bar}. 

\medskip

Throughout this appendix, the curve $X$ will be assumed proper. 

\ssec{The space of $G$-bundles with a generic reduction}

\sssec{}

Let $Y$ be a test affine scheme. We shall say that an open subset of $Y\times X$ is a \emph{domain} if it is dense
in every fiber of the projection $Y\times X\to X$. Note that the intersection of two domains is again a domain. 

\medskip

Observe that for $\CI\in \Maps(Y,\Ran)$, the subscheme $Y\times X-\Gamma_\CI$ is a domain. 

\sssec{}

Let $\Bun_{G\on{-gen}}$ be the prestack that assigns to an affine test-scheme $Y$ the groupoid, whose objects are pairs:

\medskip 

\noindent{(i)} A domain $U\subset Y\times X$;

\medskip 

\medskip 

\noindent{(ii)} A $G$-bundle $\CP_G$ defined on $U$.

\medskip

An (iso)morphism between two such points is by definition an isomorphism of $G$-bundles
defined over a \emph{subdomain} of the intersection of their respective domains of definition.

\begin{rem}  \label{r:domain}
In particular, given $(\CP_G,U)$, if $U'\subset U$ is a sub-domain, then 
the points $(\CP_G,U)$ and $(\CP_G|_{U'},U')$ are canonically isomorphic. Hence, in the definition
of $\Bun_{G\on{-gen}}$ we can combine points (i) and (ii) into:

\medskip 

\noindent{(i')} A $G$-bundle $\CP_G$ defined over \emph{some} domain in $Y\times X$.

\end{rem}

\sssec{}

Let $H\to G$ be a homomorphism of algebraic groups. Consider the prestack
$$\Bun_{H\on{-gen}}\underset{\Bun_{G\on{-gen}}}\times \Bun_G.$$

\medskip

By definition, for a test affine scheme $Y$, its groupoid of $Y$-points has as objects triples:

\medskip 

\noindent{(i)} A $G$-bundle $\CP_G$ on $Y\times X$;

\medskip 

\noindent{(ii)} A domain $U\subset Y\times X$;

\medskip 

\noindent{(iii)} A reduction $\beta$ of $\CP_G$ to $H$ defined over $U\subset Y\times X$;

\medskip

An (iso)morphism between two such points is by definition an isomorphism of $G$-bundles, 
compatible with the reductions \emph{over the intersection of the corresponding domains}. 

\begin{rem}
As in Remark \ref{r:domain} above, we can combine (ii) and (iii) into:

\medskip 

\noindent{(ii')} A reduction $\beta$ of $\CP_G$ to $H$ defined over \emph{some} domain in $Y\times X$.

\end{rem}

\sssec{}

For $H=\{1\}$, we will use the notation
$$\Gr_{G,\on{gen}}:=\on{pt}\underset{\Bun_{G\on{-gen}}}\times \Bun_G.$$ 

\medskip

By definition, for an affine test-scheme $Y$, the set $\Maps(Y,\Gr_{G,\on{gen}})$ consists of pairs
$(\CP_G,\alpha)$, where $\CP_G$ is a $G$-bundle on $Y\times X$, and $\alpha$ is a trivialization
of $\CP_G$ defined on \emph{some} domain in $Y\times X$.

\sssec{}

We have a canonically defined map
$$\Gr_{G,\on{gen}}\to \Bun_{H\on{-gen}}\underset{\Bun_{G\on{-gen}}}\times \Bun_G,$$
obtained by base change along $\Bun_G\to \Bun_{G\on{-gen}}$ from the map 
$$\on{pt}\to  \Bun_{H\on{-gen}}.$$

In addition, we have a canonical map
$$\Gr_{G,\Ran}\to \Gr_{G,\on{gen}}.$$

Composing, we obtain a map 
\begin{equation} \label{e:from Gr Ran to H gen}
\Gr_{G,\Ran}\to \Bun_{H\on{-gen}}\underset{\Bun_{G\on{-gen}}}\times \Bun_G.
\end{equation} 

\sssec{}  \label{sss:UHC}

We recall the following definition from \cite[Sect. 2.5.1]{Ga2}:

\medskip

A map between prestacks $\CX_1\to \CX_2$ is said to be \emph{universally homologically contractible} if for
any affine test-scheme $Y$ and a map $Y\to \CX_2$, 
the !-pullback functor
$$\Shv(Y)\to \Shv(Y\underset{\CX_2}\times \CX_1)$$
is fully faithful. 

\medskip

If this happens, a formal argument shows that for any prestack $\CY$ and a map $\CY\to \CX_2$, 
the !-pullback functor
$$\Shv(\CY)\to \Shv(\CY\underset{\CX_2}\times \CX_1)$$
is also fully faithful. In particular, the pullback functor
$$f^!:\Shv(\CX_2)\to \Shv(\CX_1)$$
is fully faithful. 

\medskip

We shall call a prestack $\CX$ \emph{homologically contractible} if the map $\CX\to \on{pt}$ induces
a fully faithful embedding
$$\Vect\to \Shv(\CY);$$
this is equivalent to the trace map
$$\on{C}_\bullet(\CY):=\on{C}_c^\bullet(\CY,\omega_\CY)\to \sfe$$
being an isomorphism.  It is not difficult to see that this condition implies a stronger one, namely, that 
$\CX\to \on{pt}$ is universally homologically contractible. 

\sssec{}

The goal of this section is to prove:

\begin{thm} \label{t:gen contr}
Assume that $H$ is connected. Then the map \eqref{e:from Gr Ran to H gen} is universally homologically
contractible.
\end{thm}

\sssec{}  \label{sss:proof of contr}

Let us show how \thmref{t:gen contr} implies \thmref{t:contr}. We take $H=N$. Note that there is a canonically defined map
(in fact, a closed embedding)
$$\BunNb\to \Bun_{N\on{-gen}}\underset{\Bun_{G\on{-gen}}}\times \Bun_G.$$

\medskip

Indeed, a $Y$-point of $\Bun_{N\on{-gen}}\underset{\Bun_{G\on{-gen}}}\times \Bun_G$ can be thought of as a data of $(\CP_G,\kappa)$, where $\CP_G$ is
a $G$-bundle on $Y\times X$, and $\kappa$ is a system of bundle maps 
$$\kappa^{\check\lambda}:\CO_X\to \CV^{\check\lambda}_{\CP_G},\quad \check\lambda\in \check\Lambda^+$$
\emph{defined over some domain} $U\subset T\times X$, and satisfying the Pl\"ucker relations. 

\medskip

Such a point belongs to $\BunNb$ if and only if the maps $\kappa^{\check\lambda}$ extend to regular
maps on all of $Y\times X$. 

\medskip

Finally, we note that we have a Cartesian square:
$$
\CD
\ol{S}{}^0_\Ran  @>>>  \Gr_{G,\Ran}  \\
@VVV   @VVV  \\
\BunNb @>>> \Bun_{N\on{-gen}}\underset{\Bun_{G\on{-gen}}}\times \Bun_G. 
\endCD
$$

\qed

\ssec{Towards the proof of \thmref{t:gen contr}}

\sssec{}

The assertion of \thmref{t:gen contr} is obtained as a combination of the following two statements:

\begin{prop} \label{p:gen vs Ran}
The map $\Gr_{G,\Ran}\to \Gr_{G,\on{gen}}$ is universally homologically contractible.
\end{prop}

\begin{thm} \label{t:gen gen contr}
Let $H$ be connected. Then the map $\on{pt} \to \Bun_{H\on{-gen}}$
is universally homologically contractible.
\end{thm}

\sssec{}  \label{sss:pseudo-proper}

Let us recall the notion of what it means for a map of prestacks $\CX_1\to \CX_2$ to be
\emph{pseudo-proper} (cf. \cite[Sect. 1.5]{Ga2}):

\medskip

We shall say that a prestack $\CX$ over an affine scheme $Y$ is \emph{pseudo-proper} if it can be represented 
as a colimit of schemes proper over $Y$.

\medskip

We shall say that a map of prestacks $f:\CY_1\to \CY_2$ is \emph{pseudo-proper} if for any affine test-scheme $Y$ and a map $Y\to \CX_2$,
the map 
$$Y\underset{\CX_2}\times \CX_1\to Y$$
is pseudo-proper.

\medskip

In {\it loc.cit.} it is shown that if $f$ is pseudo-proper, the functor $f_!$, left adjoint to $f^!$, is defined, and satisfies
base change against !-pullbacks and the projection formula with the $\sotimes$ tensor product. 

\medskip

From here we obtain:

\begin{lem} \label{l:UHC pseudo-proper}
Let $\CX_1\to \CX_2$ be pseudo-proper. Then it is universally homologically contractible if and only
if its fibers over field-valued points (potentially, after extending the ground field) are homologically contractible.
\end{lem}

\sssec{Interlude: the relative Ran space}

Let $\CI_0$ be a finite subset of $k$-points of $X$. We define the relative Ran space  $\Ran^{\supset \CI_0}$
as follows:

\medskip

For an affine test-scheme $Y$, the set of $Y$-points of $\Ran^{\supset \CI_0}$ consists of finite non-empty subsets
$$\CI\subset \Hom(Y,X),$$
such that $Y\times \CI_0$ is set-theoretically contained in $\Gamma_{\CI}$. 

\medskip

We claim:

\begin{prop} \label{p:rel Ran contr}
The prestack $\Ran^{\supset \CI_0}$ is homologically contractible.
\end{prop} 

The proof repeats the proof of the homological contractibility of $\Ran$, see \cite[Appendix]{Ga4}.

\sssec{Proof of \lemref{l:pr contr} for $X$ proper}  \label{sss:proof of pr contr}

If $X$ is proper, $\Ran$ is is pseudo-proper. Hence, in this case, the 
map $p^\lambda_\Ran$ is pseudo-proper. Therefore, by \lemref{l:UHC pseudo-proper}, it suffices to show
that the fibers of $p^\lambda_\Ran$ (over field-valued points) are homolgically contractible. 

\medskip

For a given field-valued point $D\in X^\lambda$, let $\CI_0\subset X$ be its support. The fiber of
$p^\lambda_\Ran$ identifies with $\Ran^{\supset \CI_0}$.

\medskip

Now the assertion follows from \propref{p:rel Ran contr}.

\qed 

\sssec{Proof of \propref{p:gen vs Ran}}

It is easy to see that the map $\Gr_{G,\Ran}\to \Gr_{G,\on{gen}}$ is pseudo-proper. Hence, by \lemref{l:UHC pseudo-proper},
it suffices to see that its fibers over field-valued points are homologically contractible. 

\medskip

For a given (field-valued) point of $\Gr_{G,\on{gen}}$, let $U\subset X$ be the maximal open subset over which $\alpha$ is defined. 
Let $\CI_0$ be its set-theoretic 
complement. Then 
$$\on{pt}\underset{\Gr_{G,\on{gen}}}\times \Gr_{G,\Ran}$$
identifies with $\Ran^{\supset \CI_0}$.

\medskip

Now the required assertion follows from \propref{p:rel Ran contr}.

\qed

\ssec{Proof of \thmref{t:gen gen contr}}

\sssec{}

Let $\Bun_{H\on{-gen,triv}}$ be the prestack, whose value on an affine test-scheme $Y$ is the full subgroupoid
of $\Maps(Y,\Bun_{H\on{-gen}})$ consisting of objects isomorphic to the trivial one. In other words, this is the essential
image of the functor
$$*=\Maps(Y,\on{pt})\to \Maps(Y,\Bun_{H\on{-gen}}).$$

\medskip

The assertion of \thmref{t:gen gen contr} is obtained as a combination of the following two statements:

\begin{thm} \label{t:contr 1}
For $H$ connected, 
the map $\on{pt}\to \Bun_{H\on{-gen,triv}}$ is universally homologically contractible.
\end{thm} 

\begin{thm} \label{t:contr 2}
The map $\Bun_{H\on{-gen,triv}}\to \Bun_{H\on{-gen}}$
is universally homologically contractible.
\end{thm} 

\sssec{Proof of \thmref{t:contr 1}}

Let $\BMaps(X,H)_{\on{gen}}$ be the group prestack that attaches to an affine test-scheme $Y$ the group of
maps from a domain in $Y\times X$ to $H$. By definition
$$\Bun_{H\on{-gen,triv}}\simeq B(\BMaps(X,H)_{\on{gen}}).$$

\medskip

Hence, in order to prove \thmref{t:contr 1}, it suffices to show that the prestack $\BMaps(X,H)_{\on{gen}}$ is homologically 
contractible. However, this is essentially what is proved in \cite[Theorem 1.8.2]{Ga4}:

\medskip

In order to formally deduce the homological contractibility of $\BMaps(X,H)_{\on{gen}}$ from \emph{loc. cit.}, we argue as follows:

\medskip

Let $\BMaps(X,H)_\Ran$ be the prestack that assigns to an affine test-scheme $Y$ the set of pairs $(\CI,h)$, where
$\CI$ is a finite non-empty subset in $\Hom(Y,X)$ and $h$ is a map
$$(Y\times X-\Gamma_\CI)\to H.$$

\medskip

We have a tautologically defined map
$$\BMaps(X,H)_\Ran\to \BMaps(X,H)_{\on{gen}},$$
and as in \propref{p:gen vs Ran} we show that this map is universally homologically 
contractible.

\medskip

Now, the assertion of \cite[Theorem 1.8.2]{Ga2} is precisely that for $H$ connected, the prestack 
$\BMaps(X,H)_\Ran$ is homologically contractible. 

\qed

\sssec{}

The remainder of this section is devoted to the proof of \thmref{t:contr 2}. Write
$$1\to H_u\to H\to H_r\to 1,$$
where $H_u$ is the unipotent radical of $H$ and $H_r$ is the reductive quotient. 

\medskip

We factor the map $\Bun_{H\on{-gen,triv}}\to \Bun_{H\on{-gen}}$ as
$$\Bun_{H\on{-gen,triv}} \to \Bun_{H_r\on{-gen,triv}} \underset{\Bun_{H_r\on{-gen}}}\times \Bun_{H\on{-gen}}\to 
\Bun_{H\on{-gen}}.$$ 

We will prove that the maps
\begin{equation} \label{e:unip part}
\Bun_{H\on{-gen,triv}} \to \Bun_{H_r\on{-gen,triv}} \underset{\Bun_{H_r\on{-gen}}}\times \Bun_{H\on{-gen}}
\end{equation}
and
\begin{equation} \label{e:red part}
\Bun_{H_r\on{-gen,triv}}\to \Bun_{H_r\on{-gen}}
\end{equation}
are universally homologically contractible, which would imply the assertion of \thmref{t:contr 2}.

\begin{rem}
Note that in the applications for the present paper, we have $H=N$, so we do not actually need to consider \eqref{e:red part}. 
\end{rem}

\sssec{}

In order to prove the universal homological contractibility property of \eqref{e:unip part}, we can base change with
respect to the (value-wise surjective) map $\on{pt}\to \Bun_{H_r\on{-gen,triv}}$. We obtain a map
$$\Bun_{H_u\on{-gen,triv}} \to \Bun_{H_u\on{-gen}},$$
and the statement that \eqref{e:unip part} is universally homologically contractible amounts to the statement
of \thmref{t:contr 2} for $H$ unipotent. 

\medskip

However, we claim that for $H$ unipotent, the map
$\Bun_{H\on{-gen,triv}}\to \Bun_{H\on{-gen}}$ is actually an isomorphism. Indeed, every $H$-bundle is (non-canonically) trivial
over a domain that is affine. 

\sssec{}

Let us observe that the statement that \eqref{e:red part} is universally homologically contractible is equivalent to the statement
of \thmref{t:contr 2} for $H$ reductive. Hence, for the rest of the argument $H$ will be assumed reductive. 

\ssec{Proof of \thmref{t:contr 2} for $H$ reductive}

\sssec{}

In order to prove that
$$\Bun_{H\on{-gen,triv}}\to \Bun_{H\on{-gen}}$$
is universally homologically contractible, it suffices to show that it becomes an isomorphism
after localization in the h-topology. (We recall that h-covers include fppf covers as well as maps that are proper and surjective
at the level of $k$-points.) 

\medskip

Since \eqref{e:red part} is a value-wise monomorphism, it suffices to show that it is a surjection in the h-topology.  

\sssec{}

Consider the Cartesian square
$$
\CD
\Bun_{H\on{-gen,triv}}\underset{\Bun_{H\on{-gen}}}\times \Bun_H @>>>  \Bun_H  \\
@VVV   @VVV \\
\Bun_{H\on{-gen,triv}} @>>> \Bun_{H\on{-gen}}.
\endCD
$$

It suffices to show that both maps
\begin{equation} \label{e:DS map}
\Bun_{H\on{-gen,triv}}\underset{\Bun_{H\on{-gen}}}\times \Bun_H \to  \Bun_H 
\end{equation}
and 
\begin{equation} \label{e:extend bundle}
\Bun_H \to \Bun_{H\on{-gen}}
\end{equation}
are h-surjections. 

\sssec{}

The fact that map \eqref{e:DS map} is an h-surjection follows from \cite{DS}; in fact the main theorem of {\it loc.cit.}
asserts that this map is an fppf surjection.

\sssec{}

Let us show that \eqref{e:extend bundle} is an h-surjection. 

\medskip

Fix a $Y$-point $(\CP_G,U)$ of $\Bun_{H\on{-gen}}$ for an affine test-scheme $Y$. The fiber product
$$Y\underset{\Bun_{H\on{-gen}}}\times \Bun_H$$
is a prestack that assigns to $Y'\to Y$ the set of extensions of the $G$-bundle $\CP_G|_{Y'\underset{Y}\times U}$ to
all of $Y'\times X$. 

\medskip

It is easy to see that this prestack is (ind)representable by an ind-scheme, ind-proper over $Y$.  Hence, it is
enough to show that the map
$$Y\underset{\Bun_{H\on{-gen}}}\times \Bun_H\to Y$$
is surjective at the level of $k$-points. 

\medskip

However, the latter means that any $H$-bundle on open subset of $X$ can be extended to all of  $X$,
which is well-known.

\end{document}